\documentclass{amsart}
\usepackage{graphicx} 
\usepackage{amssymb}

\usepackage{amsmath}

\usepackage{tikz}
\usetikzlibrary{positioning,quotes}

\usepackage{url}

\usepackage{apptools}

\usepackage{url}

\usepackage{apptools}

\usepackage{graphicx} 
\usepackage[initials]{amsrefs}
\usepackage{amsfonts}
\usepackage{amsthm}
\usepackage{enumitem}
\usepackage{xcolor}

\usepackage{bm, graphicx, hyperref, mathrsfs,indentfirst}
\usepackage{stmaryrd}

\usepackage{todonotes}

\usepackage{comment}

\usepackage{todonotes}

\usepackage{comment} 
\newtheorem{theorem}{Theorem}[section]
\newtheorem{lemma}[theorem]{Lemma}
\newtheorem{corollary}[theorem]{Corollary}
\newtheorem{claim}[theorem]{Claim}
\newtheorem{proposition}[theorem]{Proposition}

\theoremstyle{definition}
\newtheorem{definition}[theorem]{Definition}
\newtheorem{example}[theorem]{Example}

\theoremstyle{problem}
\newtheorem{problem}[theorem]{Problem}

\theoremstyle{remark}

\numberwithin{equation}{section}

\newcommand{\xmax}{x_{\scriptsize\rm max}}
\newcommand{\xmin}{x_{\scriptsize\rm min}}

\begin{document}

\title{Toeplitz Subshifts of Finite Rank }
\author{Su Gao}
\address{School of Mathematical Sciences and LPMC, Nankai University, Tianjin 300071, P.R. China}
\email{sgao@nankai.edu.cn}
\thanks{The first author acknowledges the partial support of his research by the National Natural Science Foundation of China (NSFC) grants 12271263 and 12250710128.}
\author{Ruiwen Li}
\address{School of Mathematical Sciences and LPMC, Nankai University, Tianjin 300071, P.R. China}
\email{rwli@mail.nankai.edu.cn}
\thanks{The second author acknowledges the partial support of his research by the National Natural Science Foundation of China (NSFC) grant 124B2001.}
\author{Bo Peng}
\address{Department of Mathmatics and Statistics, McGill University. 805 Sherbrooke Street West Montreal, Quebec, Canada, H3A 2K6}
\email{bo.peng3@mail.mcgill.ca}
\author{Yiming Sun}
\address{School of Mathematical Sciences and LPMC, Nankai University, Tianjin 300071, P.R. China}
\email{ymsun@mail.nankai.edu.cn}
\date{\today}
\begin{abstract} In this paper we study some basic problems about Toeplitz subshifts of finite topological rank. We define the notion of a strong Toeplitz subshift of finite rank $K$ by combining the characterizations of Toeplitz-ness and of finite topological rank $K$ from the point of view of the Bratteli--Vershik representation or from the $\mathcal{S}$-adic point of view. The characterization problem asks if for every $K\geq 2$, every Toeplitz subshift of topological rank $K$ is a strong Toeplitz subshift of rank $K$. We give a negative answer to the characterization problem by constructing a Toeplitz subshift of topological rank $2$ which fails to be a strong Toeplitz subshift of rank $2$. However, we show that the set of all strong Toeplitz subshifts of finite rank is generic in the space of all infinite minimal subshifts. In the second part we consider several classification problems for Toeplitz subshifts of topological rank $2$ from the point of view of descriptive set theory. We completely determine the complexity of the conjugacy problem, the flip conjugacy problem, and the bi-factor problem by showing that, as equivalence relations, they are hyperfinite and not smooth. We also consider the inverse problem for all Toeplitz subshifts. We give a criterion for when a Toeplitz subshift is conjugate to its own inverse, and use it to show that the set of all such Toeplitz subshifts is a meager set in the space of all infinite minimal subshifts. Finally, we show that the automorphism group of any Toeplitz subshift of finite rank is isomorphic to $\mathbb{Z}\oplus C$ for some finite cyclic group $C$, and for every nontrivial finite cyclic group $C$, $\mathbb{Z}\oplus C$ can be realized as the isomorphism type of an automorphism group of a strong Toeplitz subshift of finite rank greater than $2$.
\end{abstract}

\maketitle
\section{Introduction} This paper is a contribution to the study of minimal Cantor systems. Among all minimal Cantor systems, the odometers are well understood; these are characterized as either the equicontinuous ones or the ones with topological rank $1$ (see, e.g. \cite{DSurvey}). Thus it is a natural next step to consider minimal Cantor systems that are expansive and have finite topological rank greater than $1$.
By the well-known result of Downarowicz--Maass \cite{DM}, these are minimal subshifts of finite topological rank.

Toeplitz subshifts were first defined by Jacobs--Keane \cite{JK} and are arguably the most-studied kind of minimal subshifts. The Toeplitz subshifts we consider in this paper have finite alphabets. According to recent results of Pavlov--Schmieding \cite{PS}, there is a natural Polish topology on the space of all infinite minimal subshifts so that the subset of all Toeplitz subshifts is generic. This is saying in a rigorous way that Toeplitz subshifts are {\em typical} minimal subshifts.

The notion of topological rank, particularly that of finite topological rank, originated from \cite{DM}; the terminology was first used by Durand in \cite{Du10}. Even earlier, Ferenczi \cite{Fe96} has introduced a notion of $\mathcal{S}$-adic subshift and defined the notion of alphabet rank. Only recently, Donoso--Durand--Maass--Petite \cite{DDMP21} showed that the class of all minimal subshifts of finite topological rank and that of all minimal $\mathcal{S}$-adic subshifts of finite alphabet rank coincide up to conjugacy. More recently, Gao--Jacoby--Johnson--Leng--Li--Silva--Wu \cite{GJJLLSW25} introduced yet another notion of spacer rank (or symbolic rank) for subshifts, and Gao--Li \cite{GL25} showed that the class of all minimal subshifts of finite spacer (symbolic) rank again coincides with that of all minimal subshifts of finite topological rank up to conjugacy. Thus, we sometimes refer to this class as {\em minimal subshifts of finite rank}, without specifying exactly which rank we are using to measure the complexity. However, when it comes to a numerical value of the rank, we need to be specific, since the values of these different rank notions can differ. 

It is worth noting that Pavlov--Schmieding \cite{PS} also showed that among all infinite minimal subshifts, those having topological rank $2$ form a generic class. It follows from results of \cite{DDMP21} and \cite{GL25} that this still holds with the topological rank replaced by the other two notions of rank.

In this paper we consider several basic problems about Toeplitz subshifts of finite rank. 

The first problem is called the characterization problem. Previous research have provided characterizations of Toeplitz-ness and finite topological rank from the points of view of Bratteli diagrams and $\mathcal{S}$-adic subshifts. For example, Gjerde--Johansen \cite{GJ00} characterized the Bratteli diagrams of Toeplitz subshifts by a so-called equal path number property; by definition, a minimal Cantor system has finite topological rank $K$ if it has a Bratteli diagram of rank $K$. Thus one naturally wonders whether all Toeplitz subshifts of finite rank $K$ coincide with those having Bratteli diagrams with both the equal path number property and having rank $K$. Similarly, from the $\mathcal{S}$-adic point of view, Arbul\'u--Durand--Espinoza \cite{ADE2024} characterized Toeplitz subshifts by properties of directive sequences generating the subshift (the most important property being constant-length), and results of \cite{DDMP21} characterized finite topological rank $K$ in terms of finite alphabet rank $K$. Thus one wonders if Toeplitz subshifts of finite topological rank $K$ coincide with the $\mathcal{S}$-adic subshifts with both finite alphabet rank $K$ and the constant-length property. It turns out that these classes with the combined properties, either from the Bratteli diagram point of view or from the $\mathcal{S}$-adic point of view, are the same. For the convenience of our discussion, we call them {\em strong Toeplitz subshifts of rank $K$}. The characterization problem is thus formulated as follows.

\begin{problem}[The Characterization Problem] For $K\geq 2$, is every Toeplitz subshift of rank $K$ a strong Toeplitz subshift of rank $K$?
\end{problem}

Unfortunately, the answer is no. We will construct a Toeplitz subshift of topological rank $2$ which is not a strong Toeplitz subshift of rank $2$. We then study in more depth the notion of strong Toeplitz subshift of rank $2$, and show that they also form a generic class in the space of all infinite minimal subshifts. 

The second problem we study in this paper is the conjugacy problem. Here we consider the conjugacy problem as a Borel equivalence relation on a Polish space and study its complexity from the point of view of descriptive set theory. This methodology has been successfully applied to many classification problems in mathematics (for an overview, see \cite{Hj} and \cite{Ga}). The conjugacy problem for all Toeplitz subshifts have also been approached this way, by Thomas \cite{Thomas2013}, Sabok--Tsankov \cite{SabokTsankov2017}, Kaya \cite{Kaya}, and Yu \cite{Yu}.  Most notably, Kaya \cite{Kaya} showed that the conjugacy problem for all Toeplitz subshifts with the so-called {\em growing blocks property} is hyperfinite. Here we show that the conjugacy problem for all Toeplitz subshifts of topological rank $2$ is also hyperfinite (and not smooth). This completely determines the complexity of this classification problem in the Borel reducibility hierarchy. Using recent results of Espinoza \cite{Es23}, we obtain a similar result for the bi-factor problem. The following theorem summarizes our results on the complexity of the classification problems for Toeplitz subshifts of topological rank $2$.   

\begin{theorem} The following classification problems for all Toeplitz subshifts of topological rank $2$ are hyperfinite and not smooth:
\begin{enumerate}
\item[\rm (1)] the conjugacy problem;
\item[\rm (2)] the flip conjugacy problem;
\item[\rm (3)] the bi-factor problem.
\end{enumerate}
\end{theorem}

We remark that from the point of view of descriptive set theory, these problems all belong to the class of so-called {\em countable Borel equivalence relations}, and it is abstractly known that any countable Borel equivalence relation is hyperfinite on a comeager subset. Here we prove that these problems are hyperfinite on a {\em specific} comeager subset, namely the set of all Toeplitz subshifts of topological rank $2$. 

Next we consider the inverse problem for all Toeplitz subshifts. Here we provide a criterion for a given Toeplitz subshift to be conjugate to its own inverse, and use it to show that such Toeplitz subshifts form a meager class in the space of all infinite minimal subshifts. 

Finally we study the automorphism groups of Toeplitz subshifts of finite rank. The following summarizes our main findings on this topic.

\begin{theorem} The automorphism group of any Toeplitz subshift of finite rank is isomorphic to $\mathbb{Z}\oplus C$ for a finite cyclic group $C$. Conversely, for any finite cyclic group $C$, there is a strong Toeplitz subshift of finite rank whose automorphism group is isomorphic to $\mathbb{Z}\oplus C$.
\end{theorem}

Donoso--Durand--Maass--Petite (\cite{DDMP16} and \cite{DDMP21}) had shown that the automorphism group of a minimal Cantor system of topological rank $2$ always has an automorphism group isomorphic to $\mathbb{Z}$. Thus, for nontrivial finite cyclic group $C$, the strong Toeplitz subshift of finite rank whose automorphism group is isomorphic to $\mathbb{Z}\oplus C$ necessarily has alphabet rank greater than $2$.

Before we close this introduction, we would like to note that, although there is a lot of interesting and important research about Toeplitz subshifts that was done recently, our basic tools to study Toeplitz subshifts still come from the seminal papers of Williams \cite{Williams1984} and Downarowicz--Kwiatkowski--Lacroix \cite{DKL}.

The rest of this paper is organized as follows. In Section~\ref{sec:2} we review the preliminaries about descriptive set theory, word combinatorics, minimal Cantor systems, odometers, subshifts, Toeplitz subshifts, $\mathcal{S}$-adic subshifts, and Bratteli--Vershik representations of minimal Cantor systems. In Section~\ref{sec:3} we formulate and answer the characterization problem. In Section~\ref{sec:4} we characterize the strong Toeplitz subshifts of rank $2$, and use the characterization to show that they form a generic class in the space of all infinite minimal subshifts. In Section~\ref{sec:5} we study the classification problems and show that they are hyperfinite and not smooth. In Section~\ref{sec:6} we study the inverse problem. In Section~\ref{sec:8} we study the automorphism groups. Finally, in Section~\ref{sec:7} we make some further remarks about the orbit equivalence, and present some problems left open by our research.

\section{Preliminaries\label{sec:2}}

\subsection{Descriptive set theory}
In this paper we will use some concepts, terminology and notation from descriptive set theory. In this subsection we review these concepts, terminology and notation, which can be found in \cite{Ke} and \cite{Ga}. 

A \textbf{Polish space} is a topological space that is separable and completely metrizable. Let $X$ be a Polish space and let $d_X$ be a compatible complete metric on $X$. Let $K(X)$ be the space of all compact subsets of $X$. Let $d_H$ be the \textbf{Hausdorff metric} defined on $K(X)$ as follows. For $A\in K(X)$ and $x\in X$, let $d(x,A)=\inf\{d(x,y)\colon y\in A\}$. Now for $A, B\in K(X)$, let
$$ d_H(A,B)=\max\left\{ \sup\{d(x, B)\colon x\in A\}, \sup\{d(y,A)\colon y\in B\}\right\}. $$
Then $d_H$ is a metric on $K(X)$ that makes $K(X)$ a Polish space. Moreover, if $X$ is compact, then $K(X)$ is compact. 

Let $X$ be a Polish space. A subset $A$ of $X$ is $\boldsymbol{G_\delta}$ if $A$ is the intersection of countably many open subsets of $X$. A subspace $Y$ of $X$ is Polish if and only if $Y$ is a $G_\delta$ subset of $X$. We say that a subset $A$ of $X$ is \textbf{generic}, or the elements of $A$ are \textbf{generic} in $X$, if $A$ contains a dense $G_\delta$ subset of $X$.

More generally, by a transfinite induction on $1\leq \alpha<\omega_1$, we can define the \textbf{Borel hierarchy} on $X$ as follows:
$$\begin{array}{rcl}
{\bf\Sigma}^0_1&=& \mbox{ the collection of all open subsets of $X$} \\
{\bf\Pi}^0_1&=& \mbox{ the collection of closed subsets of $X$} \\
{\bf\Sigma}^0_\alpha&=& \left\{ \bigcup_{n\in\mathbb{N}} A_n\,:\, A_n\in {\bf\Pi}^0_{\beta_n} \mbox{ for some $\beta_n<\alpha$}\right\} \\
{\bf\Pi}^0_\alpha &=& \left\{ X\setminus A\,:\, A\in {\bf\Sigma}^0_\alpha\right\}
\end{array}
$$
We also define ${\bf\Delta}^0_\alpha={\bf\Sigma}^0_\alpha\cap{\bf\Pi}^0_\alpha$. Thus ${\bf\Delta}^0_1$ is the collection of all clopen subsets of $X$. With this notation, $\bigcup_{\alpha<\omega_1}{\bf\Sigma}^0_\alpha=\bigcup_{\alpha<\omega_1}{\bf\Pi}^0_\alpha=\bigcup_{\alpha<\omega_1}{\bf\Delta}^0_\alpha$ is the collection of all \textbf{Borel} subsets of $X$. The collection of all $G_\delta$ subsets of $X$ is exactly ${\bf\Pi}^0_2$.

Let $X$ be a Polish space. Recall that a subset $A$ of $X$ is \textbf{nowhere dense} in $X$ if the interior of the closure of $A$ is empty. $A$ is {\em meager} in $X$ if $A\subseteq \bigcup_{n\in\mathbb{N}} B_n$ where each $B_n$ is nowhere dense in $X$. $A$ is \textbf{nonmeager} in $X$ if it is not meager in $X$; $A$ is \textbf{comeager} in $X$ if $X\setminus A$ is meager in $X$. 


The following lemma is a folklore in descriptive set theory.

\begin{lemma}\label{lem:categoryquantifiers} Let $X, Y$ be Polish spaces, $V\subseteq Y$ be nonempty open, $\alpha<\omega_1$, and $A\subseteq X\times Y$. Then the following hold.
\begin{enumerate}
\item[\rm (i)] If $A$ is ${\bf\Sigma}^0_\alpha$, then the set
$$ \{ x\in X\,:\, \mbox{ $\{y\in V\,:\, (x,y)\in A\}$ is nonmeager in $V$}\} $$
is ${\bf\Sigma}^0_\alpha$ in $X$.
\item[\rm (ii)] If $A$ is ${\bf\Pi}^0_\alpha$, then the set
$$ \{ x\in X\,:\, \mbox{ $\{y\in V\,:\, (x,y)\in A\}$ is comeager in $V$}\} $$
is ${\bf\Pi}^0_\alpha$ in $X$.
\end{enumerate}
\end{lemma}

A \textbf{Borel space} is a pair $(X, \mathcal{B})$ where $X$ is a set and $\mathcal{B}$ is a $\sigma$-algebra of subsets of $X$. A \textbf{standard Borel space} is a Borel space $(X, \mathcal{B})$ where $\mathcal{B}$ is the $\sigma$-algebra of Borel sets generated by a Polish topology on $X$. When the $\sigma$-algebra $\mathcal{B}$ is clear from the context, we often omit writting it. Thus a Polish space is a standard Borel space. Moreover, if $X$ is a standard Borel space and $A\subseteq X$ is a Borel subset, then the subspace $A$ is standard Borel. 

Let $X, Y$ be standard Borel spaces. A map $f\colon X\to Y$ is \textbf{Borel} if for any Borel subset $V$ of $Y$, $f^{-1}(V)$ is a Borel subset of $X$. $f$ is a \textbf{Borel isomorphism} if $f$ is a Borel bijection. We say that $X$ and $Y$ are \textbf{Borel isomorphic} if there is a Borel isomorphism from $X$ to $Y$. It is a classical result of descriptive set theory (see, e.g. \cite[Corollary 1.3.8]{Ga}) that any two uncountable stardard Borel spaces are Borel isomorphic.

Let $E, F$ be equivalence relations on standard Borel spaces $X, Y$ respectively. We say that $E$ is \textbf{Borel reducible} to $F$, and denote it by $E\leq_B F$, if there is a Borel map $f\colon X\to Y$ such that for all $x, x'\in X$, 
$$ x\,E\, x'\iff f(x)\,F\,f(x'). $$
The function $f$ in this definition is called a \textbf{Borel reduction}. If $E\leq_B F$ and $F\leq_B E$, we say that $E$ and $F$ are \textbf{Borel bireducible}, and denote it as $E\sim_B F$. The notion of Borel reducibility is a way to compare relative complexity of equivalence relations.

An equivalence relation $E$ on a standard Borel space $X$ is \textbf{Borel} if $E$ is a Borel subset of $X\times X$. $E$ is \textbf{countable} if every equivalence class of $E$ is countable. $E$ is \textbf{finite} if every equivalence class of $E$ is finite. $E$ is \textbf{hyperfinite} if there is a sequence $\{F_n\}_{n\geq 0}$ of finite Borel equivalence relations such that $F_n\subseteq F_{n+1}$ for all $n\in\mathbb{N}$ and
$$ E=\bigcup_n F_n. $$
Given $E\subseteq F$ on a standard Borel space, we say that $E$ has \textbf{finite index} in $F$, or $F$ has \textbf{finite index} over $E$, if every equivalence class of $F$ contains only finitely many equivalence classes of $E$. 

We use the following important examples and results about Borel equivalence relations. Let $X$ be an uncountable standard Borel space. The equivalence relation ${\rm id}(X)$ is the identity (or equality) equivalence relation on $X$: $(x, y)\in {\rm id}(X)$ if and only if $x=y$. If $X, Y$ are uncountable standard Borel spaces, then ${\rm id}(X)\sim_B {\rm id}(Y)$ since $X$ and $Y$ are Borel isomorphic. We say that an equivalence relation $E$ is \textbf{smooth} if $E\leq_B {\rm id}(X)$ for some standard Borel space $X$. 

The equivalence relation $E_0$ is defined on $\{0,1\}^\mathbb{N}$ by
$$ x\,E_0\, y\iff \exists N\in\mathbb{N}\ \forall n>N\ \bigl(x(n)=y(n)\bigr). $$
$E_0$ is hyperfinite and not smooth. It is well known (see, e.g. Dougherty--Jackson--Kechris \cite{DJK}) that a countable Borel equivalence relation $E$ is hyperfinite if and only if $E\leq_B E_0$, and $E$ is not smooth if and only if $E_0\leq_B E$. Thus, a countable Borel equivalence relation $E$ is hyperfinite and not smooth if and only if $E\sim_B E_0$.

Let $X$ be an uncountable standard Borel space. The equivalence relation $E_1(X)$ is defined on $X^\mathbb{N}$ by
$$ (x_n)_n\, E_1(X)\, (y_n)_n \iff \exists N\in \mathbb{N}\ \forall n>N\ (x_n=y_n). $$
If $X, Y$ are uncountable standard Borel spaces, then $E_1(X)\sim_B E_1(Y)$ since $X$ and $Y$ are Borel isomorphic. The following is a consequence of the main result of \cite{KL}.

\begin{theorem}[Kechris--Louveau \cite{KL}]\label{thm:KL} Let $X, Y$ be uncountable standard Borel spaces, and let $E$ be a countable Borel equivalence relation on $X$. If $E\leq_B E_1(Y)$, then $E$ is hyperfinite.
\end{theorem}

We will also use the following results. The first is \cite[Proposition 1.3 (vii)]{JKL}; the second is \cite[Proposition 2.1]{Thomas2013}.

\begin{lemma}[Jackson--Kechris--Louveau \cite{JKL}]\label{lem:JKL} Let $X$ be a standard Borel space and let $E\subseteq F$ be countable Borel equivalence relations. Suppose $E$ is hyperfinite and $F$ is finite index over $E$, then $F$ is hyperfinite.
\end{lemma}

\begin{lemma}[Thomas \cite{Thomas2013}]\label{lem:Thomas} Let $X$ be a standard Borel space and let $E\subseteq F$ be countable Borel equivalence relations. If $F$ is smooth, then so is $E$.
\end{lemma}

\subsection{Word combinatorics\label{sec:2.1}}

Throughout the paper we let $\mathsf{A}$ be a finite alphabet. An element of $\mathsf{A}$ is called a \textbf{letter}. Let $\mathsf{A}^*$ be the set of all finite words with alphabet $\mathsf{A}$. Let $\varnothing$ denote the empty word. For any $u\in \mathsf{A}^*$, let $|u|$ denote the \textbf{length} of $u$. For each $n\in\mathbb{N}$, let $\mathsf{A}^n$ denote the set of all words in $\mathsf{A}^*$ of length $n$. Thus for any $u\in \mathsf{A}^*$, $u\in \mathsf{A}^{|u|}$. We always write $u=u(0)\cdots u(|u|-1)$, where $u(i)\in\mathsf{A}$ for $0\leq i<|u|$. Given $u\in \mathsf{A}^*$ and $0\leq i\leq j<|u|$, let $u[i,j)$ denote the unique word $w$ of length $j-i$ such that for any $0\leq k < j-i$, $w(k)=u(i+k)$. A word $w\in \mathsf{A}^*$ is a \textbf{subword} of $u$ if $w=u[i,j)$ for some $0\leq i<j<|u|$; in this case we also say that $w$ \textbf{occurs} in $u$ at \textbf{position} $i$. For $u, v\in \mathsf{A}^*$, we say that $u$ is a \textbf{prefix} of $v$ if $|u|<|v|$ and $v[0,|u|)=u$, and $u$ is a \textbf{suffix} of $v$ if $v[|v|-|u|, |v|)=u$. If $u, v\in \mathsf{A}^*$, we let $uv$ denote the \textbf{concatenation} of $u$ and $v$, which is defined as the unique word $w$ of length $|u|+|v|$ such that $w(k)=u(k)$ for all $0\leq k<|u|$ and $w(|u|+k)=v(k)$ for all $0\leq k<|v|$. The general concatenations of multiple words are similarly defined. 

If $W, V\subseteq \mathsf{A}^*$, we let $WV$ denote the set of all words $wv$, where $w\in W$ and $v\in V$. When $W=\{w\}$, we write $WV$ as $wV$. If $W\subseteq \mathsf{A}^*$, we let $W^*$ denote the set of all words of the form $u_1\cdots u_k$, where $u_1, \dots, u_k\in W$ and $k\geq 0$ (when $k=0$ the word represented is the empty word). In other words, $W^*=\bigcup_{n\in\mathbb{N}}W^n$. More generally, for $W\subseteq \mathsf{A}^*$, let $W^+$ be the set of all nonempty words of the form
$$ w=v_0 u_1\dots u_k v_{k+1}, $$
where $u_1,\dots, u_k\in W$ and $k\geq 0$, and for some $u_0, u_{k+1}\in W$, $v_0$ is a suffix of $u_0$, and $v_{k+1}$ is a prefix of $u_{k+1}$. In this case we say that the word $w$ is \textbf{built} from $W$, and we call the above demonstrated concatenation a \textbf{building} of $w$ from $W$. In other words, $W^+$ consists of all nonempty subwords of elements of $W^*$.

For a bi-infinite word $x\in \mathsf{A}^\mathbb{Z}$ and $i\leq j\in\mathbb{Z}$, we also let $x[i,j)$ be the unique finite word $w$ of length $j-i$ such that for any $0\leq k<j+i$, $w(k)=x(i+k)$; in this case we also say $w$ is a \textbf{subword} of $x$ and that $w$ \textbf{occurs} in $x$ at position $i$. For $W\subseteq \mathsf{A}^*$ and $x\in A^\mathbb{Z}$, we say that $x$ is \textbf{built} from $W$ if $x$ can be written as a bi-infinite concatenation of words from $W$, i.e.,
$$ x=\cdots u_{-2}u_{-1}u_0u_1\cdots, $$
where $u_i\in W$ for all $i\in\mathbb{Z}$. Again this demonstrated concatenation is called a \textbf{building} of $x$ from $W$; this has also been called a \textbf{$W$-factorization} of $x$ in the literature (see e.g. \cite{DDMP21}). To uniquely represent a building of an element $x\in \mathsf{A}^\mathbb{Z}$ we use a dot to represent position $0$; for example, if
$$ x=\cdots u_{-2}u_{-1}.u_0u_1\cdots $$
then position $0$ of $x$ is the beginning position of the word $u_0$. 

We say that $x\in \mathsf{A}^\mathbb{Z}$ is \textbf{periodic} if there exist $i, p\in\mathbb{N}$ such that for all $k\in \mathbb{Z}$ and $0\leq j< p$, $x(i+j)=x(i+j+kp)$; such a $p$ is called a \textbf{period} of $x$. Equivalently, $x$ is periodic if and only if it is built from a single finite word. We say that $x\in\mathsf{A}^\mathbb{Z}$ is \textbf{aperiodic} if $x$ is not periodic. For any subset $X\subseteq \mathsf{A}^{\mathbb{Z}}$, its \textbf{periodic part} is the set of all periodic elements of $X$, and its \textbf{aperiodic part} is the set of all aperiodic elements of $X$; $X$ is \textbf{aperiodic} if it consists only of aperiodic elements. 

If $x\in \mathsf{A}^\mathbb{Z}$ and $i\in\mathbb{Z}$, then $x(-\infty, i)$ denotes the element $y\in \mathsf{A}^{\mathbb{N}^*}$ where $\mathbb{N}^*=\{-n\colon n\in\mathbb{N}\}$ and $y(-n)=x(-n+i-1)$ for all $n\in\mathbb{N}$. For $x, x'\in\mathsf{A}^\mathbb{Z}$, we say that $\{x, x'\}$ is a 
\textbf{left asymptotic pair} if for some $i, i'\in\mathbb{Z}$, $x(-\infty, i)=x'(-\infty, i')$; $\{x, x'\}$ is a \textbf{center left asymptotic pair} if $x(-\infty, 0)=x'(-\infty, 0)$ and $x(0)\neq x'(0)$.

If $u\in\mathsf{A}^*$, we let
$$ \llbracket u\rrbracket=\{x\in \mathsf{A}^{\mathbb{Z}}\colon x[0,|u|)=u\}. $$

\subsection{Minimal Cantor systems, odometers, and subshifts\label{sec:2.2}}
We refer the reader to \cite{Au} for basic concepts, notation, and results on topological dynamical systems. 
By a \textbf{topological dynamical system} we mean a pair $(X, T)$, where $X$ is a compact metrizable space and $T\colon  X\to X$ is a homeomorphism. If $(X, T)$ is a topological dynamical system and $Y\subseteq X$ satisfies $TY=Y$, then $Y$ is called a \textbf{$T$-invariant} subset. 

If $(X, T)$ and $(Y, S)$ are topological dynamical systems and $\varphi: X\to Y$ is a continuous surjection satisfying $\varphi\circ T=S\circ \varphi$, then $\varphi$ is called a \textbf{factor map} and $(Y, S)$ is called a \textbf{factor} of $(X, T)$. If in addition $\varphi$ is a homeomorphism, then it is called a \textbf{conjugacy} (\textbf{map}) and we say that $(X, T)$ and $(Y,S)$ are \textbf{conjugate}. 

If $(X, T)$ is a topological dynamical system and $\rho$ is a compatible metric on $X$, then $\rho$ is necessarily complete since $X$ is compact. Let $(X, T)$ be a topological dynamical system and fix $\rho$ a compatible metric on $X$. We say that $(X, T)$ is \textbf{equicontinuous} if for all $\epsilon>0$ there is $\delta>0$ such that for all $n\in\mathbb{Z}$, if $\rho(x,y)<\delta$ then $\rho(T^nx, T^ny)<\epsilon$. We say that $(X, T)$ is \textbf{expansive} if for some $\delta>0$, whenever $x,y\in X$ satisfy that $\rho(T^nx, T^ny)\leq \delta$ for all $n\in\mathbb{Z}$, we have $x=y$. Since $X$ is compact, the equicontinuity and the expansiveness are topological notions and do not depend on the compatible metric $\rho$.

Every topological dynamical system $(X, T)$ has a \textbf{maximal equicontinuous factor}, i.e., an equicontinuous factor $(Y, S)$ with the factor map $\varphi$ such that if $(Z, G)$ is another equicontinuous factor of $(X, T)$ with factor map $\psi$ then there is a factor map $\theta: (Y,S)\to (Z, G)$ such that $\psi=\theta\circ \varphi$.

If $(X, T)$ is a topological dynamical system and $x\in X$, then the \textbf{orbit} of $x$ is defined as $\{T^kx\colon k\in\mathbb{Z}\}$ and is denoted by $\mathcal{O}(x)$. 

Recall that a \textbf{Cantor space} is a zero-dimensional, perfect, compact metrizable space. Let $X$ be a Cantor space and let $T\colon X\to X$ be a homeomorphism. Then $(X, T)$ is called a \textbf{Cantor system}. $T$ is \textbf{minimal} if every orbit is dense, i.e., for all $x\in X$, $\mathcal{O}(x)$ is dense in $X$. A \textbf{minimal Cantor system} is a pair $(X, T)$ where $X$ is a Cantor space and $T\colon X\to X$ is a minimal homeomorphism.

We define an important class of minimal Cantor systems, known as odometers. Let $P$ be the set of all prime numbers. A \textbf{supernatural number} is an expression of the form
$$ \mathsf{u}=\prod_{p\in P} p^{n_p} $$
where $n_p\in\mathbb{N}\cup\{\infty\}$ for each $p\in P$. A natural number is a supernatural number. Supernatural numbers can be multiplied in the natural way, and it is natural to define (finite or supernatural) factors of a supernatural number. Let $\mathsf{u}$ be a supernatural number and $n\in\mathbb{N}$. We write $n\,|\,\mathsf{u}$ if $n$ is a factor of $\mathsf{u}$. Consider the inverse system that consists of additive groups $\mathbb{Z}_n$ for $n\,|\,\mathsf{u}$ and homorphisms $\pi_{n,m}\colon \mathbb{Z}_n\to \mathbb{Z}_m$ for $m\,|\,n\,|\,\mathsf{u}$, where $\pi_{n,m}(a)\equiv a\ (\!\!\!\!\mod m)$. Let $\mbox{\rm Odo}(\mathsf{u})$ be the inverse limit of this inverse system. Then $\mbox{\rm Odo}(\mathsf{u})$ is an abelian topological group. Let $\boldsymbol{1}$ be the unique element of $\mbox{\rm Odo}(\mathsf{u})$ that projects to $1$ in every $\mathbb{Z}_n$ for $n\,|\,\mathsf{u}$. Let $S\colon\mbox{\rm Odo}(\mathsf{u})\to \mbox{\rm Odo}(\mathsf{u})$ be the map $S(x)=x+\boldsymbol{1}$. Then $(\mbox{\rm Odo}(\mathsf{u}), S)$ is a minimal Cantor system. We call  $(\mbox{\rm Odo}(\mathsf{u}), S)$ an \textbf{odometer} and call $\mathsf{u}$ the \textbf{scale} of $(\mbox{\rm Odo}(\mathsf{u}), S)$. 

Now let $(p_n)_{n\geq 0}$ be a strictly increasing sequence of natural numbers such that $p_n\,|\, p_{n+1}$ for all $n\in\mathbb{N}$. The least common multiple of all $p_n$ on the sequence is a supernatural number. We denote this supernatural number by $\mathsf{u}=\mbox{\rm lcm}(p_n)_{n\geq 0}$. The inverse system $(\mathbb{Z}_{p_n}; \pi_{p_{n+1}, p_n}\colon \mathbb{Z}_{p_{n+1}}\to \mathbb{Z}_{p_{n}})_{n\geq 0}$ has an inverse limit, which we denote as ${\rm Odo}((p_n)_{n\geq 0})$. $\mbox{\rm Odo}((p_n)_{n\geq 0})$ is isomorphic to $\mbox{\rm Odo}(\mathsf{u})$ as a topological group, and $(\mbox{\rm Odo}((p_n)_{n\geq 0}), S)$ and $(\mbox{\rm Odo}(\mathsf{u}), S)$ are conjugate as topological dynamical systems. 

Let $\mathsf{A}$ be a finite alphabet. We consider the \textbf{Bernoulli shift} on $\mathsf{A}^{\mathbb{Z}}$, which is the homeomorphism $S\colon \mathsf{A}^{\mathbb{Z}}\to \mathsf{A}^{\mathbb{Z}}$ defined by
$$ S(x)(i)=x(i+1) $$
for all $x\in \mathsf{A}^\mathbb{Z}$ and $i\in\mathbb{Z}$. Since $\mathsf{A}^{\mathbb{Z}}$ is homeomorphic to the Cantor space, $(\mathsf{A}^{\mathbb{Z}}, S)$ is a Cantor system. If $X$ is a closed $S$-invariant subset of $\mathsf{A}^{\mathbb{Z}}$, then we call $(X, S)$ a \textbf{subshift}. 

By a classical result of Hedlund, a Cantor system is expansive if and only if it is conjugate to a subshift. It was shown in \cite{DM} that every minimal Cantor system of finite topological rank (the notion of finite topological rank is to be defined in Subsection~\ref{sec:2.5}) is conjugate to either an odometer or a subshift.

Another classical result of Curtis--Hedlund--Lyndon is the following characterization of factor map in terms of block codes.

\begin{theorem}[Curtis--Hedlund--Lyndon]\label{thm:CHL} Let $\mathsf{A}$ be a finite alphabet, let $X, Y\subseteq A^\mathbb{Z}$ be subshifts, and let $\varphi\colon X\to Y$ be a factor map from $(X, S)$ to $(Y, S)$. Then there exists $n\in\mathbb{N}$ and a function $C\colon \mathsf{A}^{2n+1}\to \mathsf{A}$ such that for all $x\in X$ and $i\in\mathbb{Z}$, 
$$ \varphi(x)(i)=C(x[i-n, i+n+1)). $$
\end{theorem}

The function $C$ in the above theorem is called a \textbf{block code} for $\varphi$, and $2n+1$ is the \textbf{width} of the block code $C$. If $\varphi\colon X\to Y$ is a conjugacy map between $(X, S)$ and $(Y, S)$, then by Theorem~\ref{thm:CHL}, there are block codes $C$ for $\varphi$ and $C'$ for $\varphi^{-1}$. We denote by $|\varphi|$ the larger value between the width of $C$ and the width of $C'$. Then without loss of generality we may assume both $C$ and $C'$ have width $|\varphi|$.

\subsection{Toeplitz subshifts\label{sec:2.3}}
Toeplitz subshifts and Toeplitz sequences are standard methods of constructing 
minimal subshifts. See, e.g. Downarowicz \cite{DSurvey} for a survey of the topic. We will use similar notation as in Williams \cite{Williams1984}.

An element $x \in \mathsf{A}^{\mathbb{Z}}$ is a \textbf{Toeplitz sequence} if for all $i \in \mathbb{Z}$ there exists $p \ge 1$ such that $x(i) = x(i + kp)$ for all $k \in \mathbb{Z}$. If $x\in\mathsf{A}^\mathbb{Z}$ is a Toeplitz sequence, then the closure of its orbit in $\mathsf{A}^\mathbb{Z}$, denoted $\overline{\mathcal{O}(x)}$, is a subshift; we call it the \textbf{Toeplitz subshift} genereated by $x$. It is well known that every Toeplitz subshift is minimal, and it is aperiodic if and only if it is infinite if and only if any Toeplitz sequence generating it is aperiodic. In any Toeplitz subshift, the set of Toeplitz sequences is comeager.

Let $x \in \mathsf{A}^\mathbb{Z}$ and let $p > 1$ be an integer. Define the \textbf{$p$-periodic part} of $x$ as
    $$
        {\rm Per}_p(x) = \{i \in \mathbb{Z}\colon  x(i)=x(i + kp) \mbox{ for all }  k \in \mathbb{Z}\}.
    $$
The \textbf{periodic part} of $x$ is 
$$ {\rm Per}(x)=\bigcup_{p>1}{\rm Per}_p(x) $$
and the \textbf{aperiodic part} of $x$ is ${\rm Aper}(x)=\mathbb{Z}\setminus {\rm Per}(x)$. By definition, $x\in \mathsf{A}^\mathbb{Z}$ is a Toeplitz sequence if and only if ${\rm Per}(x)=\mathbb{Z}$. A number $i \in \mathbb{Z}$ is called a \textbf{$p$-hole} of $x$ if $i \not \in {\rm Per}_p(x)$. We usually use a  \textbf{blank symbol} $\square$ to represent $p$-holes. More precisely, define the \textbf{$p$-skeleton} of $x$ by  
    \begin{equation*}
        {\rm Skel}(x, p)(i) = 
        \begin{cases}
            x(i), & \textrm{ if } i \in {\rm Per}_p(x), \\
            \square, & \textrm{ otherwise.}
        \end{cases}
    \end{equation*}
Then ${\rm Skel}(x, p)$ is a $p$-periodic sequence in the alphabet $\mathsf{A}\cup\{\square\}$.
    
 Now let $x$ be a Toeplitz sequence. Call $p$ an \textbf{essential period} of $x$ if the $p$-skeleton of $x$ is not periodic with any smaller period. 
    A \textbf{period structure} for $x$ is a strictly increasing sequence $(p_n)_{n \ge 0}$ of essential periods of $x$  such that 
    $p_n\,|\,p_{n+1}$ for every $n$, and 
    $$
        \bigcup_{n \ge 1} {\rm Per}_{p_n}(x) = \mathbb{Z}.   
    $$
It was proven in \cite{Williams1984} that if $p$ is an essential period for $x$ then $p$ is an essential period for any Toeplitz sequence $y\in \overline{\mathcal{O}(x)}$. Thus we can speak of an essential period for a Toeplitz subshift. Consequently, if $(p_n)_{n\geq 0}$ is a period structure for $x$ then it is a period structure for any Toeplitz sequence $y\in \overline{\mathcal{O}(x)}$; thus we can also speak of a period structure for a Toeplitz subshift.

 Let $x$ be a Toeplitz sequence. 
The \textbf{scale} of $x$ is the supernatural number $\mathsf{u}_x=\mbox{\rm lcm}(p_n)_{n\geq 0}$ where $p_n$ is an enumeration of all essential periods of $x$. It was provd in \cite{Williams1984} that the odometer $(\mbox{\rm Odo}(\mathsf{u}_x), S)$ is the maximal equicontinuous factor of the Toeplitz subshift $(\overline{\mathcal{O}(x), S)}$. In particular, the scale does not depend on the individual Toeplitz sequence $x$, and we can speak of the scale for a Toeplitz subshift.

 Let $X=\overline{\mathcal{O}(x)}$ be the Toeplitz subshift generated by an aperiodic Topelitz sequence $x$. Following \cite{Williams1984}, for $p>1$ and $0\leq k<p$, define
$$ A(x, p, k)=\{S^i(x)\colon i\equiv k\ (\!\!\!\!\mod p)\}. $$
Then 
$$ \overline{A(x,p,k)}=\{y\in X\colon {\rm Skel}(y, p)= S^k {\rm Skel}(x,p)\}; $$
in particular, all elements of $\overline{A(x,p,k)}$ have the same $p$-skeleton. Moreover, 
$$\left\{\overline{A(x,p,k)}\colon 0\leq k<p\right\}$$ is a partition of $X$. Denote this partition by ${\rm Parts}(X, p)$. If $W\in {\rm Parts}(X, p)$, then all $y\in W$ have the same $p$-skeleton; thus we can define ${\rm Skel}(W,p)$ to be the $p$-skeleton of any $y\in W$.

By Theorem~\ref{thm:CHL}, if a subshift is conjugate to a Toeplitz subshift, then it is itself a Toeplitz subshift. 

For conjugacy between Toeplitz subshifts, we have the following chracterization (\cite[Theorem 1]{DKL}).

\begin{theorem}[Downarowicz--Kwitakowski--Lacroix \cite{DKL}]  \label{thm:DKL} Let $\mathsf{A}$ be a finite alphabet, let $X=\overline{\mathcal{O}(x)}, Y=\overline{\mathcal{O}(y)}\subseteq \mathsf{A}^\mathbb{Z}$ be Toeplitz subshifts, where $x, y$ are Toeplitz sequences. Then there exists a conjugacy map $\varphi$ between $(X, S)$ and $(Y, S)$ with $\varphi(x)=y$ if and only if there exist a positive integer $p$ and a permutation $\phi\colon \mathsf{A}^p\to \mathsf{A}^p$ such that for all $k\in\mathbb{Z}$, 
$$ y[kp, (k+1)p)=\phi(x[kp, (k+1)p)). $$
\end{theorem}

This theorem motivated the following definition of Kaya \cite{Kaya}. Let $\mathsf{A}$ be a finite set and $p>1$ be an integer. Let ${\rm Sym}(\mathsf{A}^p)$ denote the set of all permutations on $\mathsf{A}^p$. For any $\phi\in {\rm Sym}(\mathsf{A}^p)$, define $\widehat{\phi}\colon \mathsf{A}^\mathbb{Z}\to \mathsf{A}^\mathbb{Z}$ by
$$\widehat{\phi}(x)[kp,(k+1)p)=\phi (x[kp,(k+1)p)) $$
for any $x\in \mathsf{A}^\mathbb{Z}$ and any $k\in\mathbb{Z}$. Then $\widehat{\phi}$ is a homeomorphism. Define an equivalence relation $E_p$ on $K(\mathsf{A}^{\mathbb{Z}})$ by 
    $$
    K\, E_p\, L \Leftrightarrow 
L = \widehat{\phi}(K) \textrm{ for some } \phi \in {\rm Sym}(\mathsf{A}^p).
    $$
Since there are only finitely many elements of ${\rm Sym}(\mathsf{A}^p)$, $E_p$ is a finite equivalence relation.
We will also use the following equivalence relation $E_p^{\rm fin}$ introduced in \cite{Kaya}. For finite sets $\mathcal{K}, \mathcal{L} \subseteq K(\mathsf{A}^{\mathbb{Z}})$, define
   $$
    \mathcal{K}\, E_p^{\rm fin}\, \mathcal{L}
    \Leftrightarrow
    \left\{ [K]_{E_p} \colon K \in \mathcal{K} \right\}
    =
    \left\{ [L]_{E_p} \colon L \in \mathcal{L} \right\}.
    $$   
Then $E_p^{\rm fin}$ is also a finite equivalence relation.

\subsection{$\mathcal{S}$-adic subshifts\label{sec:2.4}}

We recall the basic definition of $\mathcal{S}$-adic subshifts and related notions following Donoso--Durand--Maass--Petite \cite{DDMP21}. 
If $A, B$ are finite alphabets, a \textbf{morphism} $\tau\colon A^*\to B^*$ is a map satisfying that $\tau(\varnothing)=\varnothing$ and for all $u, v\in A^*$, $\tau(uv)=\tau(u)\tau(v)$. A morphism $\tau\colon A^*\to B^*$ is \textbf{erasing} if $\tau(a)=\varnothing$ for some $a\in A$; otherwise it is \textbf{non-erasing}. In this paper we tacitly assume that all morphisms we consider are non-erasing. 

For a morphism $\tau\colon A^*\to B^*$ and for any $x\in \mathsf{A}^\mathbb{Z}$, define $\tau(x)$ by
$$ \tau(x)=\cdots \tau(x(-2))\tau(x(-1)).\tau(x(0))\tau(x(1))\cdots. $$
Then $\tau$ is a continuous map from $A^\mathbb{Z}$ to $B^\mathbb{Z}$, and every element of $\tau(A^\mathbb{Z})$ is built from $\tau(A)=\{\tau(a)\colon a\in A\}\subseteq B^*$. If $X\subseteq A^\mathbb{Z}$ is a subshift, then the smallest subshift containing $\tau(X)$ is $\bigcup_{k\in\mathbb{Z}}S^k\tau(X)$. It is clear that an element $x\in B^\mathbb{Z}$ is built from $\tau(A)$ if and only if $x \in \bigcup_{k\in\mathbb{Z}}S^k\tau(A^\mathbb{Z})$.

A \textbf{directive sequence} is a sequence of morphisms $\boldsymbol{\tau}=(\tau_n\colon A_{n+1}^*\to A_n^*)_{n\geq 0}$. For $0\leq n<N$, let
$$\tau_{[n,N)}=\tau_{[n,N-1]}=\tau_n\circ\tau_{n+1}\circ\cdots\circ \tau_{N-1}. $$
Then $\tau_{[n,N)}\colon A_N^*\to A_n^*$ is a morphism. 
For any $n\geq 0$, define
$$ L^{(n)}(\boldsymbol{\tau})=\{w\in A_n^*\colon \mbox{$w$ occurs in $\tau_{[n,N)}(a)$ for some $a\in A_N$ and $N>n$}\} $$
and
$$ X^{(n)}_{\boldsymbol{\tau}}=\{x\in A_n^\mathbb{Z}\,:\, \mbox{every finite subword of $x$ is a subword of some $w\in L^{(n)}(\boldsymbol{\tau})$}\}. $$
$X^{(n)}_{\boldsymbol{\tau}}$ is a subshift on the alphabet $A_n$, and we denote the shift map by $S$. Now let $X_{\boldsymbol{\tau}}=X^{(0)}_{\bf\tau}$. Then $(X_{\boldsymbol{\tau}},S)$ is the \textbf{$\mathcal{S}$-adic subshift} generated by the directive sequence $\boldsymbol{\tau}$.

For an arbitrary strictly increasing sequence $(n_k)_{k\geq 0}$ with $n_0=0$, the \textbf{contraction} or \textbf{telescoping} of a directive sequence $\boldsymbol{\tau}=(\tau_n\colon A_{n+1}^*\to A_n^*)_{n\geq 0}$ with respect to $(n_k)_{k\geq 0}$ is a directive sequence $\boldsymbol{\tau}'=(\tau'_k\colon (A'_{k+1})^*\to (A'_k)^*)_{k\geq 0}$, where for each $k\geq 0$, $A'_k=A_{n_k}$, $\tau'_k=\tau_{[n_k,n_{k+1})}$. It is clear that if $\boldsymbol{\tau}'$ is a contraction of $\boldsymbol{\tau}$, then $X_{\boldsymbol{\tau}'}=X_{\boldsymbol{\tau}}$. 

A directive sequence $\boldsymbol{\tau}=(\tau_n\colon A_{n+1}^*\to A_n^*)_{n\geq 0}$ has \textbf{finite alphabet rank} if $\liminf_n |A_n|<+\infty$; here $K=\liminf_n |A_n|$ is the \textbf{alphabet rank} of $\boldsymbol{\tau}$. It is clear that a directive sequence $\boldsymbol{\tau}$ has alphabet rank $K$ if and only there is a contraction $\boldsymbol{\tau}'=(\tau'_k\colon (A'_{k+1})^*\to (A'_k)^*)_{k\geq 0}$ of $\boldsymbol{\tau}$ such that $|A'_k|=K$ for all $k\geq 1$. The \textbf{alphabet rank} of a subshift $(X, S)$ is the minimum possible value of the alphabet rank of a directive sequence $\boldsymbol{\tau}$ such that $X=X_{\boldsymbol{\tau}}$. 

A directive sequence $\boldsymbol{\tau}=(\tau_n)_{n\geq 0}$ is \textbf{primitive} if for any $n\geq 0$ there exists $N>n$ such that all letters in $A_n$ occur in $\tau_{[n,N)}(a)$ for all $a\in A_N$.

A morphism $\tau\colon A^*\to B^*$ is \textbf{proper} if there exists $p, q\in B$ such that for all $a\in A$, $\tau(a)$ starts with $p$ and ends with $q$; A directive sequence $\boldsymbol{\tau}=(\tau_n)_{n\geq 0}$ is \textbf{proper} if all morphisms $\tau_n$ for $n\geq 0$ are proper.

A morphism $\tau\colon A^*\to B^*$ has \textbf{constant length} if the length of $\tau(a)$ does not depend on $a$. In this case, denote by $|\tau|$ the \textbf{length} of $\tau(a)$ for any $a\in A$. A directive sequence $\boldsymbol{\tau}=(\tau_n)_{n\geq 0}$ has \textbf{constant length} if all morphisms $\tau_n$ for $n\geq 0$ have constant length.

Suppose $\tau\colon A^*\to B^*$ is a constant-length morphism. A \textbf{coincidence} of $\tau$ relative to $A'\subseteq A$ is an integer $i\in [0,|\tau|)$ such that the map that sends a letter $a\in A'$ to the $i$-th letter of $\tau(a)$ is constant. When $A'=A$ these integers are called \textbf{coincidences}. The set of all coincidences is denoted $\mbox{coinc}(\tau)$.
A constant-length directive sequence $\boldsymbol{\tau}=(\tau_n)_{n\geq 0}$ has \textbf{coincidences} if there is a contraction $(\tau_k')_{k\geq 0}$ such that $\mbox{coinc}(\tau'_k)\neq\varnothing$ for all $k\geq 0$.

For the following notion of recognizability, we follow Berth\'e--Steiner--Thuswaldner--Yassawi \cite{BSTY19}. Let $\tau\colon A^*\to  B^*$ be a morphism and let $x\in B^\mathbb{Z}$. A \textbf{$\tau$-representation} of $x$ is a pair $(y,k)$, where $y\in A^{\mathbb{Z}}$ and $k\in\mathbb{Z}$, such that $x=S^k\tau(y)$. Clearly, a $\tau$-representation of $x$ corresponds to a building of $x$ from $\tau(A)$; thus $x$ has a $\tau$-representation if and only if $x$ is in the smallest subshift containing $\tau(A^\mathbb{Z})$. A $\tau$-representation $(y,k)$ is \textbf{centered} if $0\leq k<|\tau(y(0))|$. For a subshift $Y\subseteq A^\mathbb{Z}$, if a $\tau$-representation $(y,k)$ satisfies $y\in Y$, then we say that $(y,k)$ is \textbf{in $Y$}. We say that $\tau$ is \textbf{recognizable in $Y$} if every $x\in B^\mathbb{Z}$ has at most one centered $\tau$-representation in $Y$. If every aperiodic point $x\in B^\mathbb{Z}$ has at most one centered $\tau$-representation in $Y$, we say that $\tau$ is \textbf{recognizable in $Y$ for aperiodic points}. If $\tau$ is recognizable in $A^\mathbb{Z}$ (for aperiodic points), then we say that $\tau$ is \textbf{fully recognizable} (for aperiodic points). 

It is well known (see e.g. \cite[Lemma 3.2]{DDMP21}) that a morphism $\tau: A^*\to B^*$ is recognizable in $Y$ if and only if
$$ \mathcal{P}=\{S^k \tau(\llbracket a\rrbracket\cap Y)\colon a\in A,\ 0\leq k<|\tau(a)|\} $$
defines a partition of the subshift $\bigcup_{k\in\mathbb{Z}}S^k\tau(Y)$. Similarly, $\tau$ is recognizable in $Y$ for aperiodic points if and only if $\mathcal{P}$ gives rise to a partition of the aperiodic part of $\bigcup_{k\in\mathbb{Z}}S^k\tau(Y)$.

We will use the following basic fact about recognizability for morphisms, which appeared as a part of \cite[Theorem 2.5]{BSTY19} and also as a part of \cite[Lemma 3.2]{DDMP21}.

\begin{lemma}[Berth\'e--Steiner--Thuswaldner--Yassawi \cite{BSTY19}] \label{lem:rec} Let $\tau\colon A^*\to B^*$ be a morphism and let $Y\subseteq A^\mathbb{Z}$ be a subshift. Then $\tau$ is recognizable in $Y$ if and only if
there exists a positive integer $r$ such that for any $x, x'\in\bigcup_{k\in\mathbb{Z}}S^k\tau(Y)$ and any centered $\tau$-representations $(y,k), (y',k')$  for $x, x'$ respectively, if $x[-r,r)=x'[-r,r)$, then $k=k'$ and $y(0)=y'(0)$.
\end{lemma}

A directive sequence $\boldsymbol{\tau}=(\tau_n)_{n\geq 0}$ is \textbf{recognizable} if for each $n\geq 0$, $\tau_n$ is recognizable in $X^{(n+1)}_{\boldsymbol{\tau}}$. It follows from \cite[Lemma 3.5]{BSTY19} that $\boldsymbol{\tau}$ is recognizable if and only if for any $n\geq 1$, any $x\in X_{\boldsymbol{\tau}}$ has a unique building from $\{\tau_{[0,n)}(a)\colon a\in A_n\}$.

\subsection{Bratteli--Vershik representations\label{sec:2.5}}
The concepts and terminology reviewed in this subsection are from Herman--Putnam--Skau \cite{HPS}, Giordano--Putnam--Skau  \cite{GPS} and \cite{DM}. Some notations are from \cite{DDMP21}.
Recall that a \textbf{Bratteli diagram} is an infinite graph $(V,E)$ with the following properties:
\begin{itemize}
\item The vertex set $V$ is decomposed into pairwise disjoint nonempty finite sets $V=V_0\cup V_1\cup V_2\cup\cdots$, where $V_0$ is a singleton $\{v_0\}$;
\item The edge set $E$ is decomposed into pairwise disjoint nonempty finite sets $E=E_1\cup E_2\cup\cdots$;
\item For any $n\geq 1$, each $e\in E_n$ connects a vertex $u\in V_{n-1}$ with a vertex $v\in V_n$. In this case we write $\mathsf{s}(e)=u$ and $\mathsf{r}(e)=v$. Thus $\mathsf{s}, \mathsf{r}: V\to E$ are maps such that $\mathsf{s}(E_n)\subseteq V_{n-1}$ and $\mathsf{r}(E_n)\subseteq V_n$ for all $n\geq 1$.
\item $\mathsf{s}^{-1}(v)\neq\varnothing$ for all $v\in V$ and $\mathsf{r}^{-1}(v)\neq\varnothing$ for all $v\in V\setminus V_0$.
\end{itemize}

An \textbf{ordered Bratteli diagram} is a Bratteli diagram $(V,E)$ together with a partial ordering $\preceq$ on $E$ so that edges $e$ and $e'$ are $\preceq$-comparable if and only if $\mathsf{r}(e)=\mathsf{r}(e')$. 

A finite or infinite \textbf{path} in a Bratteli diagram $(V, E)$ is a sequence $(e_1, e_2, \dots)$ where  $\mathsf{r}(e_i)=\mathsf{s}(e_{i+1})$ for all $i\geq 1$. 
Given a Bratteli diagram $(V, E)$ and $0\leq n<m$, let $E_{n,m}$ be the set of all finite paths connecting vertices in $V_n$ and those in $V_m$. If $p=(e_{n+1},\dots, e_m)\in E_{n,m}$, define $\mathsf{r}(p)=\mathsf{r}(e_m)$ and $\mathsf{s}(p)=\mathsf{s}(e_{n+1})$. If in addition the Bratteli diagram is partially ordered by $\preceq$, then we also define a partial ordering $p\preceq' q$ for $p=(e_{n+1},\dots, e_m), q=(f_{n+1},\dots, f_m)\in E_{n,m}$ as either $p=q$ or there exists $n+1\leq i\leq m$ such that $e_i\neq f_i$, $e_i\preceq f_i$ and $e_j=f_j$ for all $i<j\leq m$. For an arbitrary strictly increasing sequence $(n_k)_{k\geq 0}$ of natural numbers with $n_0=0$, the \textbf{contraction} or \textbf{telescoping} of a Bratteli diagram $(V,E)$ with respect to $(n_k)_{k\geq 0}$ is a Bratteli diagram $(V',E')$ where
$V'_k=V_{n_k}$ for $k\geq 0$ and $E'_k=E_{n_{k-1}, n_k}$ for $k\geq 1$. If in addition the given Bratteli diagram is ordered, then by contraction or telescoping we also obtain an ordered Bratteli diagram $(V',E',\preceq')$ with the order $\preceq'$ defined above. 

A Bratteli diagram $(V, E)$ is \textbf{simple} if there is a strictly increasing sequence $(n_k)_{k\geq 0}$ of natural numbers with $n_0=0$ such that the telescoping $(V', E')$ of $(V, E)$ with respect to $(n_k)_{k\geq 0}$ satisfies that for all $n\geq 1$, $u\in V'_{n-1}$ and $v\in V'_n$, there is $e\in E'_n$ with $\mathsf{s}(e)=u$ and $\mathsf{r}(e)=v$. This is equivalent to the property that for any $n\geq 1$ there is $m>n$ such that every pair of vertices $u\in V_n$ and $v\in V_m$ are connected by a finite path. 

A Bratteli diagram $B=(V, E)$ is of \textbf{finite rank} if $\liminf_n|V_n|<+\infty$; here $K=\liminf_n|V_n|$ is called the \textbf{rank} of $B$. It is clear that the rank of $B$ is $K<+\infty$ if and only if for some contraction $B'=(V', E')$, $|V'_n|=K$ for all $n\geq 1$. 

A Bratteli diagram $B=(V,E)$ has the \textbf{equal path number property} or is of \textbf{Toeplitz type}, if for any $n\in\mathbb{N}$ and for any $u, v\in V_{n+1}$, $|\{e\in E_n\colon \mathsf{r}(e)=u\}|=|\{e\in E_n\colon \mathsf{r}(e)=v\}|$.

Given a Bratteli diagram $B=(V,E)$, define
$$ X_B=\{(e_n)_{n\geq 1}\,:\, e_n\in E_n, \mathsf{r}(e_n)=\mathsf{s}(e_{n+1}) \mbox{ for all $n\geq 1$}\}. $$
Since $X_B$ is a subspace of the product space $\prod_{n\geq 1}E_n$, we equip $X_B$ with the subspace topology of the product topology on $\prod_{n\geq 1}E_n$. An ordered Bratteli diagram $B=(V,E,\preceq)$ is \textbf{properly ordered} if there are unique elements $\xmax=(e_n)_{n\geq 1}, \xmin=(f_n)_{n\geq 1}\in X_B$ such that for every $n\geq 1$, $e_n$ is a $\preceq$-maximal element and $f_n$ is a $\preceq$-minimal element. 

Given a simple properly ordered Bratteli diagram $B=(V,E,\preceq)$, the \textbf{Vershik map} $\lambda_B: X_B\to X_B$ is defined as follows: $\lambda_B(\xmax)=\xmin$; if $(e_n)_{n\geq 1}\in X_B$ and $(e_n)_{n\geq 1}\neq \xmax$, then let 
$$ \lambda_B((e_1,e_2,\dots, e_k, e_{k+1},\dots))=(f_1,f_2,\dots, f_k, e_{k+1},\dots), $$
where $k$ is the least such that $e_k$ is not $\preceq$-maximal, $f_k$ is the $\preceq$-successor of $e_k$, and $(f_1,\dots, f_{k-1})$ is the unique path from $v_0$ to $\mathsf{s}(f_k)=\mathsf{r}(f_{k-1})$ such that   $f_i$ is $\preceq$-minimal for each $1\leq i\leq k-1$. Then $(X_B, \lambda_B)$ is a minimal Cantor system (\cite{GPS}), which we call the \textbf{Bratteli--Vershik system} generated by $B$; 

Given a minimal Cantor system $(X, T)$ and a simple properly ordered Bratteli diagram $B$, we say that $B$ is a \textbf{Bratteli--Vershik representation} of $(X, T)$ if $(X, T)$ is conjugate to the Bratteli--Vershik system $(X_B, \lambda_B)$. A minimal Cantor system $(X, T)$ has \textbf{finite topological rank} if it has a Bratteli--Vershik representation which has finite rank $K$, the minimum possible value of $K$ is called the {\em topological rank} of $(X, T)$. 

A minimal Cantor system $(X, T)$ has topological rank $1$ if and only if it is conjugate to an odometer. 

Given a simple properly ordered Bratteli diagram $B=(V,E,\preceq)$, we associate to it a directive sequence $\boldsymbol{\tau}_B=(\tau_n\colon A_{n+1}^*\to A_n^*)_{n\geq 0}$ as follows. Let $A_0=E_1$ and $A_n=V_n$ for all $n\geq 1$. Define $\tau_0\colon A_1^*\to A_0^*$ by letting $\tau_0(v)=e_1(v)\cdots e_{\ell}(v)$ for each $v\in V_1=A_1$, where $e_1(v), \dots, e_{\ell}(v)$ is an enumeration of $\{e\in E_1\colon \mathsf{r}(e)=v\}$ in the $\preceq$-order. For general $n\geq 1$, define $\tau_n\colon A_{n+1}^*\to A_n^*$ by letting $\tau_n(v)=\mathsf{s}(e_1(v))\cdots\mathsf{s}(e_{\ell}(v))$ for each $v\in V_{n+1}=A_{n+1}$, where $e_1(v), \dots, e_{\ell}(v)$ is an enumeration of $\{e\in E_n\colon \mathsf{r}(e)=v\}$ in the $\preceq$-order. $\boldsymbol{\tau}_B$ is called the directive sequence \textbf{read on} $B$, and $X_{\boldsymbol{\tau}_B}$ is the $\mathcal{S}$-adic subshift \textbf{read on} $B$.

\section{The Characterization Problem\label{sec:3}} 

In this section we study the Characterization Problem of Toeplitz subshifts of finite topological rank. We state two equivalent forms of the Characterization Problem in Subsection \ref{sec:3.1}, and show in Subsection \ref{sec:3.3} that they have a negative answer. Our counterexample is a Toeplitz subshift of topological rank $2$.

\subsection{Two forms of the Characterization Problem\label{sec:3.1}}

We are interested in characterizing Toeplitz subshifts of finite topological rank. Both these properties (i.e., being Toeplitz and having finite topological rank) can be characterized by properties of their $\mathcal{S}$-adic representations. The following characterization of Toeplitz subshifts is a variation of  \cite[Proposition~2.5]{ADE2024}.

\begin{proposition}[Arbul\'u--Durand--Espinoza \cite{ADE2024}]\label{prop:ADE} Let $(X,S)$ be a subshift. The following are equivalent.

\begin{enumerate}
\item[\rm (1)] $(X,S)$ is a Toeplitz subshift;
\item[\rm (2)] $(X,S)$ is generated by a constant-length, primitive, proper and recognizable directive sequence; 
\item[\rm (3)] $(X,S)$ is generated by a constant-length, primitive and recognizable directive sequence with coincidences;
\item[\rm (4)] $(X, S)$ is generated by a constant-length, primitive directive sequence with coincidences;
\item[\rm (5)] $(X, S)$ is generated by a constant-length, primitive and proper directive sequence.
\end{enumerate}
\end{proposition}

\begin{proof} The implications (1)$\Rightarrow$(2)$\Rightarrow$(3)$\Rightarrow$(1) are proved in \cite{ADE2024}. It is obvious that (3)$\Rightarrow$(4) and (2)$\Rightarrow$(5). The proof of (3)$\Rightarrow$(1) in \cite{ADE2024} does not use recognizability; this gives the implication (4)$\Rightarrow$(1). For (5)$\Rightarrow$(1), note that properness implies coincidences. 
\end{proof} 

On the other hand, we have the following theorem from \cite[Theorem~4.1(1)]{DDMP21}.

\begin{theorem}[Donoso--Durand--Maass--Petite \cite{DDMP21}]\label{thm:DDMP21} An expansive minimal Cantor system of finite topological rank $K$ is conjugate to an $\mathcal{S}$-adic subshift generated by a primitive and recognizable directive sequence with alphabet rank $k\leq K$. 
\end{theorem}

Thus it is natural to wonder whether one can combine both properties from the $\mathcal{S}$-adic viewpoint and obtain a characterization of all Toeplitz subshifts of finite topological rank $K$ to be exactly those conjugate to an $\mathcal{S}$-adic subshift generated by a constant-length, primitive, proper and recognizable directive sequence of alphabet rank $k\leq K$. This  is the first form of our \textbf{Characterization Problem}.

To facilitate further discussions we make the following definition.

\begin{definition} Let $K\geq 2$ be an integer. We call a subshift a \textbf{strong Toeplitz subshift of rank $K$} if it is conjugate to an $\mathcal{S}$-adic subshift generated by a constant-length, primitive, proper and recognizable directive sequence of alphabet rank $k\leq K$. If a subshift is a strong Toeplitz subshfit of rank $K$ for some $K$, then we call it a \textbf{strong Toeplitz subshift of finite rank}.
\end{definition}

The first form of the Characterization Problem can be formulated as follows.

\begin{problem} Given $K\geq 2$. Is every Toeplitz subshift of topological rank $K$ a strong Toeplitz subshift of rank $K$?
\end{problem}

We now turn to the second form of the Characterization Problem. This time we note that both being Toeplitz and having finite topological rank can be characterized by properties of Bratteli--Vershik representations.

The following theorem (\cite[Theorem 8]{GJ00}) characterizes Toeplitz-ness of a subshift from the viewpoint of its Bratteli--Vershik representations.

\begin{theorem}[Gjerde--Johansen \cite{GJ00}] Any Toeplitz subshift has a Bratteli--Vershik representation with the equal path number property. Conversely, an expansive minimal Cantor system which has a Bratteli--Vershik representation with the equal path number property is conjugate to a Toeplitz subshift.
\end{theorem}

By definition, a minimal Cantor system has topological rank $K$ if it has a Bratteli--Vershik representation with rank $K$. Thus it is again natural for us to wonder if any Toeplitz subshift of topological rank $K$ must have a Bratteli--Vershik representation which has both the equal path number property and rank $K$. This is the second form of our Characterization Problem. 

The two forms of the Characterization Problem are equivalent by the following result of \cite{DDMP21}.

\begin{proposition}[\cite{DDMP21}]\label{prop:equiv} Let $K\geq 2$.
\begin{enumerate}
\item[\rm (i)] Let $B$ be a simple properly ordered Bratteli diagram with rank $K$. Suppose the Bratteli--Vershik system $(X_B, \lambda_B)$ is expansive. Then there is a directive sequence $\boldsymbol{\tau}$ which is primitive, proper, recognizable, and has alphabet rank $K$, such that the $\mathcal{S}$-adic subshift $(X_{\boldsymbol{\tau}}, S)$ is conjugate to $(X_B, \lambda_B)$. Moreover, if $B$ has the equal path number property, then $\boldsymbol{\tau}$ has constant length.
\item[\rm (ii)] Let $\boldsymbol{\tau}$ be a primitive, proper and recognizable directive sequence with alphabet rank $K$. Then there is a simple properly ordered Bratteli diagram $B$ of rank $K$ such that  $(X_{\boldsymbol{\tau}}, S)$ is conjugate to $(X_B, \lambda_B)$. Moreover, if $\boldsymbol{\tau}$ has constant length, then $B$ has the equal path number property.
\end{enumerate}
\end{proposition}
\begin{proof} For (i), the required $\boldsymbol{\tau}$ is a contraction of the directive sequence read on $B$. The verification of this follows from the proof of  \cite[Proposition 4.6]{DDMP21}; only note that the equal number property of $B$ implies that the directive sequence $\boldsymbol{\tau}_B$ read on $B$, as well as any contraction of $\boldsymbol{\tau}_B$, has constant length.

For (ii), we follow \cite[Proposition 4.5]{DDMP21}, whose proof appeared as that of Durand--Leroy \cite[Proposition 2.2]{DL12}. In fact, let $\boldsymbol{\tau}=(\tau_n\colon A_{n+1}^*\to A_n^*)_{n\geq 0}$ be a primitive, proper and recognizable directive sequence with alphabet rank $K$. Define $V_0=\{v_0\}$ and $V_n=A_n$ for $n\geq 1$. For $v\in V_1=A_1$, let $e=(v_0,v,k)\in E_0$ for all $0\leq k<|\tau_0(v)|$ with $\mathsf{s}(e)=v_0$, $\mathsf{r}(e)=v$ and $e=(v_0, v, k)\preceq e'=(v_0,v,k')$ if and only if $k\leq k'$. For $n\geq 1$, $v\in V_n=A_n$, $w\in V_{n+1}=A_{n+1}$ and $0\leq k<|\tau_n(w)|$, let $e=(v,w,k)\in E_n$ if $v$ occurs in $\tau_n(w)$ as its $k$-th letter; in this case define $\mathsf{s}(e)=v$, $\mathsf{r}(e)=w$, and $e=(v,w,k)\preceq e'=(v',w,k')$ if and only if $k\leq k'$. This defines a Bratteli diagram $B$ which is simple and properly ordered, and the Bratteli--Vershik system $(X_B, \lambda_B)$ is conjugate to $(X_{\boldsymbol{\tau}}, S)$. From the construction it is clear that the rank of $B$ is the same as the alphabet rank of $\boldsymbol{\tau}$. Moreover, if $\boldsymbol{\tau}$ has constant length, then $B$ has the equal path number property.
\end{proof}

Strong Toeplitz subshifts of rank $K$ have been implicitly studied in Durand--Frank--Maass \cite{DFM15}.

\subsection{A negative answer to the Characterization Problem\label{sec:3.3}}

In this subsection we show that both forms of the Characterization Problem have a negative answer. In fact, we will construct a Toeplitz subshift of topological rank $2$ which is not a strong Toeplitz subshift of rank $2$. We first note the following property for all strong Toeplitz subshfits of rank $2$.

\begin{lemma}\label{lem:onepair} Let $\boldsymbol{\tau}=(\tau_n\colon A_{n+1}^*\to A_n^*)_{n\geq 0}$ be a constant-length, primitive, proper and recognizable directive sequence with alphabet rank $2$. Then there is a unique centered left asymptotic pair $\{x, \tilde{x}\}$ in $(X_{\boldsymbol{\tau}}, S)$. Moreover, $(-\infty, 0)\subseteq \mbox{\rm Per}(x)\cap \mbox{\rm Per}(\tilde{x})$.
\end{lemma}
\begin{proof}  Let $X=X_{\boldsymbol{\tau}}$. Without loss of generality assume $A_{n+1}=\{1,2\}$ for all $n\in\mathbb{N}$. Let $\mathsf{A}=A_0$. For each $n\in\mathbb{N}$ and $j=1,2$, let $w_{n,j}=\tau_{[0,n+1)}(j)$ and let $p_n=|w_{n,1}|=|w_{n,2}|$. 
Since $X$ is aperiodic, $w_{n,1}\ne w_{n,2}$ for every $n\in\mathbb{N}$. Without loss of generality, assume for all $n\geq 1$, if $i_n\in\mathbb{N}$ is the least such that $0\leq i_n<|\tau_n(1)|=|\tau_n(2)|$ and $\tau_n(1)(i_n)\neq \tau_n(2)(i_n)$, then we have $\tau_n(1)(i_n)=1$ and $\tau_n(2)(i_n)=2$.

For each $n\in\mathbb{N}$, let $u_n$ be the maximal common prefix of $w_{n,1}$ and $w_{n,2}$. Write $w_{n,1}=u_nv_{n,1}$ and $w_{n,2}=u_nv_{n,2}$. Then for each $n\in\mathbb{N}$, $k_n=|u_n|$, $|v_{n,1}|=|v_{n,2}|=p_n-k_n$, $u_n$ is a suffix of $u_{n+1}$, and for each $j=1,2$, $v_{n,j}$ is a prefix of $v_{n+1, j}$. By properness, we have $k_n\to +\infty$ and $p_n-k_n\to +\infty$ as $n\to +\infty$. Let $x\in A^\mathbb{Z}$ be such that $x[-k_n, p_n-k_n)=w_{n,1}$ for every $n\in\mathbb{N}$ and let $\tilde{x}\in A^\mathbb{Z}$ be such that $\tilde{x}[-k_n, p_n-k_n)=w_{n,2}$ for every $n\in\mathbb{N}$. Then $x$ and $\tilde{x}$ are well defined. It is routine to check that $\{x,\tilde{x}\}$ is a centered left asymptotic pair and $[-k_n,0)\subseteq{\rm Per}_{p_n}(x)\cap {\rm Per}_{p_n}(\tilde{x})$ for every $n\in\mathbb{N}$, so we have that $(-\infty,0)\subseteq {\rm Per}(x)\cap {\rm Per}(\tilde{x})$.

Now assume $\{y, \tilde{y}\}$ is another centered left asymptotic pair in $(X_{\boldsymbol{\tau}}, S)$. We show that $\{x,\tilde{x}\}=\{y, \tilde{y}\}$. Fix any $n\geq 1$. By the recognizability of $\boldsymbol{\tau}$ and by Lemma~\ref{lem:rec}, each of $y$ and $\tilde{y}$ has a unique building from $\{w_{n,1}, w_{n,2}\}$, and for some $\ell<0$, these buildings agree on $(-\infty, \ell)$. Let $\ell<0$ be the greatest such integer. Since $|w_{n,1}|=|w_{n,2}|=p_n$, we must have $\ell=-|u_n|=-k_n$. Thus $\{y[-k_n, p_n-k_n), \tilde{y}[-k_n, p_n-k_n)\}=\{w_{n,1}, w_{n,2}\}$. Without loss of generality, we may assume $y[-k_n, p_n-k_n)=w_{n,1}$ and $\tilde{y}[-k_n, p_n-k_n)=w_{n,2}$. It follows that for all $m\geq 1$, we indeed have $y[-k_m, p_m-k_m)=w_{m,1}$ and $\tilde{y}[-k_m, p_m-k_m)=w_{m,2}$. Thus $y=x$ and $\tilde{y}=\tilde{x}$.  
\end{proof}

We also need the following observation about $\mathcal{S}$-adic representations of Toeplitz subshifts. 

\begin{lemma}\label{gcd}
Let $\boldsymbol{\tau}=(\tau_n\colon A^*_{n+1}\to A^*_n)_{n\geq 0}$ be a primitive, proper and recognizable directive sequence of alphabet rank $2$. Suppose that $(X_{\boldsymbol{\tau}},S)$ is an aperiodic Toeplitz subshift. For each $n\in\mathbb{N}$, let $d_n=\mbox{\rm gcd}(|\tau_{[0,n+1)}(a)|\colon a\in A_{n+1})$. Then $\mbox{\rm lcm}(d_n)_{n\geq 0}$ is the scale for $(X_{\boldsymbol{\tau}}, S)$. 
\end{lemma}

\begin{proof} Without loss of generality assume $A_{n+1}=\{1,2\}$ for all $n\in\mathbb{N}$. For each $n\in\mathbb{N}$ and $j=1,2$, let $w_{n,j}=\tau_{[0,n+1)}(j)$. Then $d_n=\mbox{\rm gcd}(|w_{n,1}|, |w_{n,2}|)$. Let $X=X_{\boldsymbol{\tau}}$ and fix a Toeplitz sequence $x\in X$. We show that every essential period of $x$ is a factor of some $d_n$, and that every $d_n$ is a factor of some essential period of $x$.

Let $p$ be an essential period of $x$.  Then by \cite[Lemma 2.3]{Williams1984}, the set $U=\{y\in X\colon {\rm Skel}(y,p)={\rm Skel}(x,p)\}$ is a clopen neighbourhood of $x$ in $X$, $\{S^kU:0\le k<p\}$ is a partition of $X$, and $S^{p}U=U$. It follows that there is $m\in \mathbb{N}$ such that $\{y\in X\colon y[-m,m)=x[-m,m)\}$ is a subset of $U$. Since $(X,S)$ is minimal, there is $m'>2m$ such that for every $y\in X$ and every subword $s$ of $y$ whose length is at least $m'$, we have that $x[-m,m)$ is a subword of $s$. Let $n\in\mathbb{N}$ be such that $|w_{n,1}|, |w_{n,2}|\ge m'$. We claim that $p$ divides both $|w_{n+1,1}|$ and $|w_{n+1,2}|$. To see this, note that $(X,S)$ is aperiodic, so there is $z\in X$ such that $z[0,|w_{n+1,1}|+|w_{n+1,2}|)=w_{n+1,1}w_{n+1,2}$. Since $\boldsymbol{\tau}$ is proper, we may assume that both $w_{n+1,1}$ and $w_{n+1,2}$ begin with $w_{n,j_0}$ for some $j_0\in\{0,1\}$. Then by the definition of $m'$, $x[-m,m)$ is a subword of $w_{n,j_0}$; in other words, there is $t\in[0,|w_{n,j_0}|-2m)$ such that $w_{i,j_0}[t,t+2m)=x[-m,m)$. So $S^{t+m}(z)[-m,m)=x[-m,m)$ and $S^{t+m+|w_{n+1,1}|}(z)[-m,m)=x[-m,m)$. It follows that $S^{t+m}(z)\in U$ and $S^{t+m+|w_{n+1,1}|}(z)\in U$, and thus $U\cap S^{|w_{n+1,1}|}U\ne\varnothing$. Since $\{S^kU\colon 0\le t<p\}$ is a partition of $X$ and $S^{p}U=U$, we have that $p$ divides $|w_{n+1,1}|$. By a similar argument, $p$ also divides $|w_{n+1,2}|$.

Conversely, fix $n\in\mathbb{N}$. By the recognizability of $\boldsymbol{\tau}$, each element of $X$ has a unique building from $\{w_{n,1}, w_{n,2}\}$. In particular, $x$ has a unique building from $\{w_{n,1}, w_{n,2}\}$. Since $x$ is aperiodic, both $w_{n,1}$ and $w_{n,2}$ occur in this building. Suppose $w_{n,1}$ occurs at position $i$ in this building of $x$. By Lemma~\ref{lem:rec} there is a positive integer $r>|i|$ such that for any $y\in X$, if $y[-r,r)=x[-r,r)$, then in the unique building of $y$ from $\{w_{n,1}, w_{n,2}\}$, $w_{n,1}$ occurs at position $i$ also. 
Now let $p$ be an essential period of $x$ such that $[-r,r)\subseteq\mbox{\rm Per}_p(x)$. Then $S^p(x)[-r,r)=x[-r,r)$, and thus in the unique building of $S^p(x)$ from $\{w_{n,1},w_{n,2}\}$, $w_{n,1}$ occurs at position $i$. Since the unique building of $S^p(x)$ is obtained from the unique building of $x$ by a shift of $p$ positions, we conclude that in the unique building of $x$ from $\{w_{n,1}, w_{n,2}\}$, $w_{n,1}$ occurs at both positions $i$ and $i+p$. Therefore $x[i, i+p)$ is a word  in $\{w_{n,1}, w_{n,2}\}^*$, and we conclude that $p$ can be written in the form $a|w_{n,1}|+b|w_{n,2}|$ for some $a,b\in\mathbb{N}$. Thus $d_n=\mbox{\rm gcd}(|w_{n,1}|,|w_{n,2}|)$ divides $p$.
\end{proof}

We are now ready to prove the main result of this section.

\begin{theorem}\label{thm:char}
There exists a Toeplitz subshift of topological rank $2$ which is not a strong Toeplitz subshift of rank $2$.
\end{theorem}

The rest of this subsection is devoted to a proof of Theorem~\ref{thm:char}. 

We work with the alphabet $\mathsf{A}=\{0,1\}$. We inductively define words $w_{n,1}, w_{n,2}\in\mathsf{A}^*$ for $n\in\mathbb{N}$, and then let
$$\begin{array}{rl} X=\{ x\in \mathsf{A}^{\mathbb{Z}}\colon& \!\!\!\!\!\! \mbox{ every finite subword of $x$ is a subword}\\ 
&\mbox{ of $w_{n,j}$ for some $n\in\mathbb{N}$ and $j\in\{1,2\}$}\}.
\end{array} $$

Let $w_{0,1}=0$ and $w_{0,2}=1$. Let $d_0=\mbox{\rm gcd}(|w_{0,1}|, |w_{0,2}|)=1$.  In general, suppose $w_{n,1}$ and $w_{n,2}$ have been defined, and $d_n=\mbox{\rm gcd}(|w_{n,1}|, |w_{n,2}|)$. We proceed to define $w_{n+1,1}, w_{n+1,2}\in\{w_{n,1}, w_{n,2}\}^*$.

Find $m_n\gg |w_{n,1}|+|w_{n,2}|$. We will have $|w_{n+1,1}|=4m_nd_n$ and $|w_{n+1,2}|=8m_nd_n$.
$w_{n+1,1}$ will begin with $w_{n,1}^2(w_{n,1}w_{n,2})^4$ and end with $w_{n,1}^2$. The position $2m_nd_n$ in $w_{n+1,1}$ will be the beginning of an occurrence of $w_{n,1}^2$. Since $m_n\gg |w_{n,1}|+|w_{n,2}|$, $2m_nd_n-6|w_{n,1}|-4|w_{n,2}|$ is a multiple of $2d_n$ which is much larger than $|w_{n,1}|+|w_{n,2}|$, and so it can be written in the form $2a|w_{n,1}|+2b|w_{n,2}|$ for some  $a,b\in\mathbb{N}$. As a result, we can construct a word $\alpha_n\in\{w_{n,1}^2,w_{n,2}^2\}^*$ such that $|\alpha_n|=2m_nd_n-6|w_{n,1}|-4|w_{n,2}|$. Similarly, there is a word $\beta_n\in \{w_{n,1}^2, w_{n,2}^2\}^*$ such that $|\beta_n|=2m_nd_n-4|w_{n,1}|$. Let 
\begin{equation}\label{2.2}
w_{n+1,1}=w_{n,1}^2(w_{n,1}w_{n,2})^4\alpha_n w_{n,1}^2\beta_n w_{n,1}^2.
\end{equation}

Next we define $w_{n+1,2}$. Let $\gamma_n, \delta_n, \eta_n, \lambda_n\in\{w_{n,1}^2, w_{n,2}^2\}^*$ such that 
$$\begin{array}{l}
|\gamma_n|=2m_nd_n-2|w_{n,1}|-2|w_{n,2}|, \\
|\delta_n|=2m_nd_n-2|w_{n,1}|, \\
|\eta_n|=2m_nd_n-2|w_{n,2}|, \\
|\lambda_n|=2m_nd_n-4|w_{n,1}|.
\end{array}
$$
Then let 
\begin{equation}\label{2.3}
w_{n+1,2}=(w_{n,1}w_{n,2})^2\gamma_n w_{n,1}^2\delta_n w^2_{n,2}\eta_n w_{n,1}^2\lambda_n w_{n,1}^2.
\end{equation}
In other words, $w_{n+1, 2}$ begins with $(w_{n,1}w_{n,2})^2$ and ends with $w_{n,1}^2$, with $w_{n,1}^2$ occurring in $w_{n+1,2}$ at positions $2m_nd_n$ and $6m_nd_n$, and $w_{n,2}^2$ occurring at position $4m_nd_n$.  

This finishes the inductive definition of the words $w_{n,1}, w_{n,2}$, and therefore also of the subshift $(X, S)$. It is natural to translate this definition into the definition of a directive sequence $\boldsymbol{\tau}=(\tau_n\colon A_{n+1}^*\to A_n^*)_{n\geq 0}$, where $A_0=\mathsf{A}$ and $A_{n+1}=\{1,2\}$ for each $n\in\mathbb{N}$. For example, for $n\geq 1$ we have
$$ \tau_n(1)=1112121212\tilde{\alpha}11\tilde{\beta}11 \mbox{ and } \tau_n(2)=1212\tilde{\gamma}11\tilde{\delta}22\tilde{\eta}11\tilde{\lambda}11 $$
for some suitable $\tilde{\alpha}, \tilde{\beta}, \tilde{\gamma},\tilde{\delta}, \tilde{\eta}, \tilde{\lambda}\in \{1, 2\}^*$. It is clear that $w_{n,j}=\tau_{[0,n)}(j)$ for any $n\in\mathbb{N}$ and $j=1,2$,  and $X_{\boldsymbol{\tau}}=X$. 

It is clear that $\boldsymbol{\tau}$ is primitive and proper.  By \cite[Theorem 4.6]{BSTY19}, $\boldsymbol{\tau}$ is recognizable. It follows from Proposition~\ref{prop:equiv} (ii) that $(X, S)$ has topological rank $2$.

\begin{claim} $(X,S)$ is an aperiodic Toeplitz subshift.
\end{claim}
\begin{proof} Since $\boldsymbol{\tau}$ is primitive, $(X, S)$ is minimal. We only need to construct an aperiodic Toeplitz sequence in $X$. Recall that for every $n\in\mathbb{N}$, $w_{n+1,1}[2m_nd_n,2m_nd_n+|w_{n,1}|)=w_{n,1}$. Let $\ell_0=0$ and $\ell_n=\sum_{0\le k< n}2m_kd_k$ for $n\ge1$. Let $x\in\mathsf{A}^{\mathbb{Z}}$ be such that $x[-\ell_n,-\ell_n+|w_{n,1}|)=w_{n,1}$ for every $n\in\mathbb{N}$. By the definition of $(w_{n,j})_{n\in\mathbb{N},1\le j\le2}$, $x$ is well defined and $x\in X$. Since $w_{n,1}$ occurs in $w_{n+1,1}$ at position $2m_nd_n$, and $w_{n,1}$ occurs in $w_{n+1,2}$ at positions $2m_nd_n$ and $6m_nd_n$, we have that $$[-\ell_n,-\ell_n+|w_{n,1}|)\subseteq{\rm Per}_{4m_nd_n}(x).$$ Since $\ell_n\to +\infty$ and $|w_{n,1}|-\ell_n\to +\infty$ as $n\to +\infty$, we conclude that  $x$ is a Toeplitz sequence. To see that $x$ is aperiodic, we note that if it was periodic, then some $4m_nd_n$ would be a period of $x$,  but $w_{n+1,2}$ does not have such a period, a contradiction.
\end{proof}

\begin{claim}\label{cl:3}There is a centered left asymptotic pair $\{y,\tilde{y}\}$ in $(X,S)$ such that $(-\infty,0)\cap {\rm Aper}(y)\cap {\rm Aper}(\tilde{y})\neq\varnothing$.
\end{claim}
\begin{proof}
By the definition of $(w_{n,j})_{n\in\mathbb{N},1\le j\le2}$ we have $|w_{0,1}|=|w_{0,2}|=1$ and $|w_{n,2}|=2|w_{n,1}|$ for $n\ge1$. Recall that $d_n=\mbox{\rm gcd}(|w_{n,2}|,|w_{n,1}|)$, thus we have $d_n=|w_{n,1}|$ for $n\in\mathbb{N}$. By our construction, for every $n\in\mathbb{N}$ and $j=1,2$, there is an occurrence of $w_{n,j}$ in $w_{n+1,j}$ at position $d_n$.  Let $q_0=0$ and  $q_n=\sum_{0\le k<n}d_k$ for $n\geq 1$. It is straightforward to check by induction on $n$ that $q_n$ is the length of the maximal common prefix of $w_{n,1}$ and $w_{n,2}$ for every $n\in\mathbb{N}$. Let $y,\tilde{y}\in X$ be such that $y[-q_n,-q_n+|w_{n,1}|)=w_{n,1}$ and $\tilde{y}[-q_n,-q_n+|w_{n,2}|)=w_{n,2}$ for every $n\in\mathbb{N}$. By the definition of $(w_{n,j})_{n\in\mathbb{N},1\le j\le2}$, $y$ and $\tilde{y}$ are both well defined. We have $y,\tilde{y}\in X$. Moreover, for every $n\in\mathbb{N}$, $y[-q_n,0)=\tilde{y}[-q_n,0)$ is the maximal common prefix of $w_{n,1}$ and $w_{n,2}$, so $y(-\infty,0)=\tilde{y}(-\infty,0)$. Now $y(0)=0$ and $\tilde{y}(0)=1$. Thus $(y,\tilde{y})$ is a left asymptotic pair.

Before further discussions, we show that for any $n\geq 1$,
\begin{equation}\label{2.4}
w_{n,1}(q_n-1)\ne w_{n,2}(d_n+q_n-1).
\end{equation}
This is proved by induction on $n$. When $n=1$, from the definition we have $w_{1,1}(q_1-1)=w_{1,1}(0)\ne  w_{1,2}(d_1)=w_{1,2}(d_1+q_1-1)$. In general, suppose (\ref{2.4}) holds for some $n\ge1$. The word $w_{n+1,1}$ begins with $w^2_{n,1}$, so $w_{n+1,1}(q_{n+1}-1)=w_{n,1}(q_n-1)$. Recall that $w_{n,2}$ occurs in $w_{n+1,2}$ at position $d_{n+1}$. So we have 
$$w_{n+1,2}(d_{n+1}+q_{n+1}-1)=w_{n,2}(q_{n+1}-1)=w_{n,2}(d_n+q_n-1). $$
By the induction hypothesis, we conclude that (\ref{2.4}) holds for $n+1$. 

We claim that $-1\in{\rm Aper}(y)\cap {\rm Aper}(\tilde{y})$. By Lemma \ref{gcd}, we just need to show that for every $n\in\mathbb{N}$, $-1\notin{\rm Per}_{d_n}(y)\cup {\rm Per}_{d_n}(\tilde{y})$. Fix $n\geq 1$. By the recognizability of $\boldsymbol{\tau}$, $y$ has a unique building from $\{w_{n,1}, w_{n,2}\}$. In fact, the occurrence of $w_{n,1}$ in $y$ at position $-q_n$ is a part of this unique building of $y$. It follows that there is $k\in\mathbb{N}$ such that $y(-1)=w_{n,1}(q_n-1)$ and $y(kd_n-1)=w_{n,2}(d_n+q_n-1)$. By (\ref{2.4}), $y(-1)\neq y(kd_n-1)$, and thus $-1\notin{\rm Per}_{d_n}(y)$. Similarly, $-1\notin{\rm Per}_{d_n}(\tilde{y})$. 
\end{proof}

Now by Lemma~\ref{lem:onepair} and Claim~\ref{cl:3}, $(X,S)$ is not a strong Toeplitz subshift of rank 2. The proof of Theorem~\ref{thm:char} is complete.

\section{Strong Toeplitz Subshifts of Rank 2\label{sec:4}}

In this section we study strong Toeplitz subshifts of rank $2$.  In Subsection{sec:4.1} we give some equivalent formulations of strong Toeplitz subshifts of rank $2$, and then in Subsection{sec:4.2} we use them to show that the class of all strong Toeplitz subshifts of rank $2$ is generic in the space of all infinite minimal subshifts. 

\subsection{Strong rank-2 cuts\label{sec:4.1}}

We will work with the following notion.

\begin{definition}
For an aperiodic $x\in \mathsf{A}^\mathbb{Z}$, a \textbf{strong rank-2 cut} of $x$ is a pair $(p,t)\in\mathbb{N}^2$, where $0\le t<p$, such that the set $\left\{x[t+kp,t+(k+1)p)\colon k\in\mathbb{Z}\right\}$ has exactly two elements.
\end{definition}

\begin{lemma}\label{lem:cut} Let $\boldsymbol{\tau}=(\tau_n\colon A^*_{n+1}\to A^*_n)_{n\geq 0}$ be a constant-length directive sequence with $|A_{n+1}|=2$ for all $n\in\mathbb{N}$. Suppose $X_{\boldsymbol{\tau}}$ is aperiodic. For any $n\in\mathbb{N}$, let $p_n=|\tau_{[0,n+1)}(a)|$ for any $a\in A_{n+1}$. Then for any $x\in X_{\boldsymbol{\tau}}$ and any $n\in\mathbb{N}$, there is a strong rank-2 cut $(p_n, t)$ of $x$ such that $\{x[t+kp_n, t+(k+1)p_n)\colon k\in\mathbb{Z}\}=\{\tau_{[0,n+1)}(a)\colon a\in A_{n+1}\}$. 
\end{lemma}

\begin{proof} Fix $x\in X_{\boldsymbol{\tau}}$ and $n\in\mathbb{N}$. For any $m\in\mathbb{N}$, $x[-m,m)$ is a subword of some $\tau_{[0,N_m)}(a)$ for some $N_m> n$ and $a\in A_{N_m}$; in particular, $x[-m, m)$ is a subword of some $\tau_{[0,n+1)}(u_m)$ for some $u_m\in A_{n+1}^*$. Let $w_m=\tau_{[0,n+1)}(u_m)$ and let $0\leq i_m<|w_m|$ be such that $w_m[i_m-m, i_m+m)=x[-m,m)$. Let $0\leq s_m<p_n$ be unique such that $p_n$ divides $i_m+s_m$. 

Since $0\leq s_m<p_n$ for all $m\in\mathbb{N}$, we have that for an infinite set $M\subseteq \mathbb{N}$, $s_m=s_{m'}$ for all $m, m'\in M$. Without loss of generality, we may assume $M=\mathbb{N}$. Let $t=s_m$ for any $m\in M=\mathbb{N}$. Then it is easy to see that for any $k\in\mathbb{Z}$, $x[t+kp_n, t+(k+1)p_n)\in \{\tau_{[0,n+1)}(a)\colon a\in A_{n+1}\}$. Since $|A_{n+1}|=2$, we have that $\{x[t+kp_n, t+(k+1)p_n)\colon k\in\mathbb{Z}\}$ also has exactly two elements, since otherwise $x$ is periodic with period $p$. Thus $(p_n, t)$ is a strong rank-$2$ cut of $x$ with the desired property.
\end{proof}

\begin{lemma}\label{coin1}
Let $x\in \mathsf{A}^\mathbb{Z}$ be aperiodic and let $0\le t<t'<p$. Suppose $(p,t)$ is a strong rank-2 cut of $x$. Then $(p,t')$ is also a strong rank-2 cut of $x$ if and only if $[t,t')$ or $[t',t+p)$ is a subset of ${\rm Per}_p(x)$.
\end{lemma}
\begin{proof} Denote $\left\{x[t+kp,t+(k+1)p):k\in\mathbb{Z}\right\}$ by $W_t$ and its two elements by $w_0$ and $w_1$. Similarly denote $\left\{x[t'+kp,t'+(k+1)p)\colon k\in\mathbb{Z}\right\}$ by $W_{t'}$. Then 
$$ W_{t'}\subseteq\left\{w_i[t'-t,p)w_j[0,t'-t)\colon 0\le i,j\le1\right\}.$$ 

($\Leftarrow$) First assume $[t,t')\subseteq{\rm Per}_p(x)$. Then $w_0[0,t'-t)=w_1[0,t'-t)$, and
$$W_{t'}\subseteq\left\{w_0[t'-t,p)w_0[0,t'-t),w_1[t'-t,p)w_0[0,t'-t)\right\};$$ 
thus $W_{t'}$ has at most two elements. Since $x$ is aperiodic, $W_{t'}$ has at least two elements, and thus $(p,t')$ is a strong rank-2 cut of $x$. The argument for $[t',t+p)\subseteq{\rm Per}_p(x)$ is similar.

($\Rightarrow$) We prove the contrapositive. Assume neither $[t,t')$ nor $[t',t+p)$ is a subset of ${\rm Per}_p(x)$. We have 
$$w_0[0,t'-t)\ne w_1[0,t'-t) \mbox{ and } w_0[t'-t,p)\ne w_1[t'-t,p). $$
So $\{w_i[t'-t,p)w_j[0,t'-t):0\le i,j\le1\}$ has exactly four elements. We observe that for $0\le i,j\le1$, $w_i[t'-t,p)w_j[0,t'-t)\in W_{t'} $ if and only if there is $k\in\mathbb{Z}$ such that
$$x[t+kp,t+(k+1)p)=w_i \mbox{ and } x[t+(k+1)p,t+(k+2)p)=w_j. $$
Since $x$ is aperiodic, it is easy to see that for at least 3 distinct pairs $(i,j)\in\{0,1\}^2$, there is $k\in\mathbb{Z}$ such that $x[t+kp,t+(k+1)p)=w_i$ and $x[t+(k+1)p,t+(k+2)p)=w_j$. Then $W_{t'}$ has at least 3 elements, and $(p,t')$ is not a strong rank-2 cut of $x$.
\end{proof}

\begin{definition} Let $x\in \mathsf{A}^\mathbb{Z}$ be aperiodic. Let $p\leq q$ be postive integers. Suppose $(p, t)$ and $(q, s)$ be two strong rank-2 cuts of $x$. We say that $(p,t)$ and $(q,s)$ \textbf{coincide} if $p$ divides $q$ and for the unique $k\in\mathbb{N}$ such that $t+kp\le s<t+(k+1)p$, we have $[t+kp,s)$ or $[s,t+(k+1)p)$ is a subset of ${\rm Per}_q(x)$. 
\end{definition}

With this definition, Lemma~\ref{coin1} can be restated as: if $(p, t)$ and $(p, s)$ are both strong rank-$2$ cuts of $x$, then they must coincide. The following lemma is a generalization of Lemma~\ref{coin1}. 

\begin{lemma}\label{coin2}
Let $x\in \mathsf{A}^\mathbb{Z}$ be aperiodic. Suppose $(p,t)$ and $(q, s)$ are two strong rank-2 cuts of $x$. If $p$ divides $q$, then $(p,t)$ and $(q,s)$ coincide.
\end{lemma}
\begin{proof}
Toward a contradiction, assume that $(p, t)$ and $(q, s)$ are strong rank-$2$ cuts of $x$, $p$ divides $q$, and $(p,t)$ and $(q,s)$ do not coincide. Let $k_0\in\mathbb{N}$ be the unique integer such that $t+k_0p\le s<t+(k_0+1)p$. Let $\ell=s-(t+k_0p)$. Denote the set $\{x[t+kp,t+(k+1)p)\colon k\in\mathbb{Z}\}$ by $W_t$ and its two elements of by $w_0$ and $w_1$. For $k\in\mathbb{Z}$, let $u_k=x[s+kq, s+(k+1)q)$ and let $W_s=\{u_k\colon k\in\mathbb{Z}\}$. By our assumption, $W_s$ has exactly two elements.

Since $p$ divides $q$ and $(p,t)$ and $(q,s)$ do not coincide, we have that $w_0[0,\ell)\ne w_1[0,\ell)$ and $w_0[\ell,p)\ne w_1[\ell,p)$. To facilitate further discussions we make the following definition. For each $k\in\mathbb{Z}$ and $(i,j)\in\{0,1\}^2$, we say that $u_k$ has \textbf{type} $(i,j)$ if 
$$x[kq+s-\ell,kq+s-\ell+p)=w_i \mbox{ and }x[(k+1)q+s-\ell,(k+1)q+s-\ell+p)=w_j. $$ 
Note that if $(i,j)\neq (i',j')$, $u_k$ has type $(i,j)$ and $u_{k'}$ has type $(i',j')$, then $u_k\neq u_{k'}$. 
Since $(p,t)$ and $(q,s)$ do not coincide, we get $k, k'\in\mathbb{Z}$ such that $u_k$ has type $(0,j_0)$ for some $j_0\in\{0,1\}$ and $u_{k'}$ has type $(1,j_1)$ for some $j_1\in\{0,1\}$. Since $W_s$ has exactly two elements, there is a unique $j_0$ and a unique $j_1$ such that some $u_k\in W_s$ has type $(0,j_0)$ and some $u_k\in W_s$ has type $(1,j_1)$. By a similar argument we have that $j_0\neq j_1$, and hence $\{j_0,j_1\}=\{0,1\}$. We consider two cases.

Case 1: $W_s$ contains a word of type $(0,0)$ and a word of type $(1,1)$. Assume further that some $u_k$ has type $(0,0)$. Then it follows that $u_{k-1}$ and $u_{k+1}$ both have type $(0,0)$. By iterating, we get that every $u_k$ has type $(0,0)$, and thus $x$ is periodic with period $q$, a contradiction. Similarly, we get a contradiction if some $u_k$ has type $(1,1)$.

Case 2: $W_s$ contains a word of type $(0,1)$ and a word of type $(1,0)$. Note that if $u_k$ has type $(0,1)$, then $u_{k-1}$ and $u_{k+1}$ both have type $(1,0)$. Similarly, if some $u_k$ has type $(1,0)$, then $u_{k-1}$ and $u_{k+1}$ both have type $(0,1)$. In any case, it follows that $x$ is periodic with period $2q$, a contradiction. 
\end{proof}

Now we are ready to present some equivalent formulations for strong Toeplitz subshifts of rank $2$.

\begin{theorem}\label{thm:STS2}
Let $(X,S)$ be an aperiodic Toeplitz subshift with scale $\mathsf{u}$. Then the following are equivalent:
\begin{enumerate}
\item[\rm (1)] $(X,S)$ is a strong Toeplitz subshift of rank $2$;
\item[\rm (2)] For every $x\in X$, $p\,|\,\mathsf{u}$ and $m\in \mathbb{N}$, there is a strong rank-2 cut $(q,t)$ of $x$ such that $p\,|\,q\,|\,\mathsf{u}$ and the lengths of the maximal common prefix and suffix of the two elements in $\{x[t+kq,t+(k+1)q)\colon k\in\mathbb{Z}\}$ are greater than $m$;
\item[\rm (3)] There is $x\in X$ such that for every $p\,|\,\mathsf{u}$ and $m\in \mathbb{N}$, there is a strong rank-2 cut $(q,t)$ of $x$ such that $p\,|\,q\,|\,\mathsf{u}$ and the lengths of the maximal common prefix and suffix of the two elements in $\{x[t+kq,t+(k+1)q)\colon k\in\mathbb{Z}\}$ are greater than $m$.
\end{enumerate}
\end{theorem}
\begin{proof}
 \noindent (1)$\Rightarrow$(2). Suppose $X=X_{\boldsymbol{\tau}}$, where $\boldsymbol{\tau}=(\tau_n\colon A^*_{n+1}\to A^*_n)_{n\geq 0}$ is a constant-length, primitive, proper and recognizable directive sequence with $A_{n+1}=\{1, 2\}$ for every $n\in\mathbb{N}$. For each $n\in\mathbb{N}$ and $j=1,2$, let $w_{n,j}=\tau_{[0,n+1)}(j)$. Fix $p\,|\,\mathsf{u}$ and $m\in \mathbb{N}$. By Lemma~\ref{gcd}, there is $n\in\mathbb{N}$ such that $|w_{n,1}|=|w_{n,2}|>m$ and $p$ divides $|w_{n,1}|$. Since $\boldsymbol{\tau}$ is proper, we have that the lengths of the maximal common prefix and suffix of $w_{n+1,1}$ and $w_{n+1,2}$ are greater than $|w_{n,1}|=|w_{n,2}|> m$. Let $q=|w_{n+1,1}|$. Then $p$ divides $q$. By Lemma \ref{gcd} again, $q$ divides $\mathsf{u}$. Now let $x\in X$ be arbitrary. By Lemma~\ref{lem:cut},  we can find a strong rank-2 cut $(q,t)$ of $x$ such that $$\{x[t+kq,t+(k+1)q)\colon k\in\mathbb{Z}\}=\{w_{i+1,1},w_{i+1,2}\}.$$
Then $(q,t)$ is the strong rank-2 cut as required.

The implication (2)$\Rightarrow$(3) is obvious.

(3)$\Rightarrow$(1). Let $(p_n)_{n\geq 0}$ be a period structure of $(X, S)$. In particular, each $p_n$ is an essential period. Fix $x\in X$ satisfying (3). We inductively define a sequence of strong rank-2 cuts of $x$, $(q_n, t_n)$ for $n\geq 1$. Let $q_0=1$, $t_0=0$. Suppose for $i\geq 0$, $q_i$ and $t_i$ have been defined so that there is $n_i>i$ with $q_i$ dividing $p_{n_i}$. We proceed to define $q_{i+1}$ and $t_{i+1}$. By (3),  there is a strong rank-2 cut $(q_{i+1}, s)$ of $x$ such that $p_{n_i}$ divides $q_{i+1}$, $q_{i+1}$ divides $\mathsf{u}$, and the lengths of the maximal common prefix and suffix of the two elements in $\{x[s+kq_{i+1},s+(k+1)q_{i+1})\colon k\in\mathbb{Z}\}$ are greater than $2q_i$. By Lemma \ref{coin2}, $(q_{i}, t_i)$ and $(q_{i+1}, s)$ coincide. Let $k\in\mathbb{N}$ be unique such that $t_i+kq_i\leq s<t_i+(k+1)q_i$. Without loss of generality assume that $[kq_i+t_i, s)\subset{\rm Per}_{q_{i+1}}(x)$ (as usual, the argument for the other case is similar). Let $t_{i+1}=kq_i+t_i$.  Then $(q_{i+1}, t_{i+1})$ is a strong rank-2 cut of $x$, the maximal common prefix and suffix of the two elements in $\{x[t_{i+1}+kq_{i+1},t_{i+1}+(k+1)q_{i+1})\colon k\in\mathbb{Z}\}$ are greater than $q_i$. This finishes the inductive definition of $(q_n, t_n)$ for $n\geq 1$. 

Now for each $n\geq 1$, let $w_{n,1}$ and $w_{n,2}$ be such that
$$\{x[t_{n}+kq_{n},t_{n}+(k+1)q_{n})\colon k\in\mathbb{Z}\}=\{w_{n,1},w_{n,2}\}. $$ 
Let $A_0$ be the alphabet of the subshift $X$, and for each $n\geq 1$, let $A_n=\{w_{n,1}, w_{n, 2}\}$. Note that $|w_{n,1}|=|w_{n,2}|=q_n$. Since for all $n\geq 0$, $t_{n+1}-t_n$ is a multiple of $q_n$, $w_{n+1,1}, w_{n+1, 2}\in \{w_{n,1}, w_{n,2}\}^*$. This allows us to define a natural morphism $\tau_n\colon A_{n+1}^*\to A_n^*$ for each $n\in\mathbb{N}$, resulting in a directive sequence $(\tau_n)_{n\geq 0}$ which is constant-length, proper, and of alphabet rank $2$. Finally, since $X$ is aperiodic, we have that for any $n\geq 1$, both $w_{n,1}$ and $w_{n,2}$ occur in both $w_{n+1, 1}$ and $w_{n+1, 2}$. Thus the contraction of $(\tau_n)_{n\geq 0}$ by omitting $\tau_0$ is a directive sequence that is primitive, and still constant-length, proper, and of alphabet rank $2$. By \cite[Theorem 4.6]{BSTY19}, $\boldsymbol{\tau}$ is recognizable. By the minimality, $(X, S)$ is generated by $\boldsymbol{\tau}$. Hence $(X, S)$ is a strong Toeplitz subshift of rank $2$.
\end{proof}

\subsection{Genericity of strong Toeplitz subshifts of rank 2\label{sec:4.2}}
In this subsection we prove that the set of all strong Toeplitz subshifts of rank 2 is generic in the space of all minimal subshifts. For this we first prove a characterization of this class of subshifts in terms of their languages. 

We will use the following concept in word combinatorics. Let $\mathsf{A}$ be a finite alphabet and let $W\subseteq \mathsf{A}^*$ be finite. Recall that $W^+$ is the set of all nonempty words $w$ built from $W$, i.e., $w$ is a subword of some word in $W^*$. For $w\in W^+$, we say that $w$ is \textbf{uniquely built} from $W$ if there is a unique sequence of words $(\beta_0, \alpha_1, \dots, \alpha_k, \alpha_{k+1})$ such that 
$$ w=\beta_0\alpha_1\cdots\alpha_k\beta_1 $$
where $k\geq 0$, $\alpha_1,\dots, \alpha_k\in W$, $\beta_0$ is a suffix of some element of $W$, and $\beta_1$ is a prefix of some element of $W$. Let $W^+_1$ be the set of all $w$ which are uniquely built from $W$.

\begin{lemma}\label{lem:unique} Let $w_1, w_2, w_3\in\mathsf{A}^*$. Suppose $|w_1|, |w_2|, |w_3|>2\sup\{|u|\colon u\in W\}$. Suppose $w_1w_2$, $w_2$ and $w_2w_3$ are all uniquely built from $W$. Then $w_1w_2w_3$ is uniquely built from $W$.
\end{lemma}

\begin{proof} Suppose $(\beta_0, \alpha_1, \dots,\alpha_k, \beta_1)$ gives the unique building of $w_2$. Since $|w_2|>2\sup\{|u|\colon u\in W\}$, $k\geq 1$. The unique building of $w_1w_2$ must be of the form $(\gamma_0, \eta_1,\dots, \eta_\ell, \gamma_1)$ where $\ell>k$, $\beta_1=\gamma_1$, $\ell_{\ell-j}=\alpha_{k-j}$ for all $0\leq j\leq k-1$, and $\beta_0$ is a suffix of $\eta_{\ell-k}$. Likewise, the unique building of $w_2w_3$ must be of the form $(\epsilon_0, \delta_1,\dots, \delta_r, \epsilon_1)$ where $r>k$, $\epsilon_0=\beta_0$, $\alpha_j=\delta_j$ for all $1\leq j\leq k$, and $\beta_1$ is a prefix of $\delta_{k+1}$. Then 
$$ s=(\gamma_0, \eta_1,\dots, \eta_{\ell}, \delta_{k+1}, \dots, \delta_r, \epsilon_1) $$
gives a building of $w_1w_2w_3$. If another sequence also gives a building of $w_1w_2w_3$, then by the uniqueness of the building of $w_1w_2$, its first $\ell+1$ many terms must coincide with the first $\ell+1$ many terms of the sequence $s$; similarly, by the uniqueness of the building of $w_2w_3$, its last $r+1$ terms must coincide with  the last $r+1$ many terms of $s$; therefore the sequence must coincide with $s$. This shows the uniqueness of the building of $w_1w_2w_3$. 
\end{proof}

For a subshift $X\subseteq \mathsf{A}^\mathbb{Z}$ and positive integer $n$, let
$$ L_n(X)=\{w\in\mathsf{A}^n\colon \exists x\in X\ \mbox{($w$ is a subword of $x$)}\}. $$
Then we have the following characterization of strong Toeplitz subshifts of rank 2.

\begin{lemma}\label{lem:STS2} Let $(X, S)\subseteq \mathsf{A}^\mathbb{Z}$ be an aperiodic Toeplitz subshift with scale $\mathsf{u}$. Then $(X,S)$ is a strong Toeplitz subshift of rank $2$ if and only if for every $p\,|\,\mathsf{u}$ and $m\in\mathbb{N}$, there exist positive integers $q, r$ and words $u, v\in\mathsf{A}^*$ such that
\begin{enumerate}
\item[\rm (i)] $|u|=|v|=q$;
\item[\rm (ii)] $p\,|\,q\,|\,\mathsf{u}$ and $q<r$;
\item[\rm (iii)] the lengths of the maximal common prefix and suffix of $u$ and $v$ are greater than $m$;
\item[\rm (iv)] $L_{2r}(X)\cup L_{4r}(X)\subseteq \{u,v\}^+_1$, i.e., every word in $L_{2r}(X)\cup L_{4r}(X)$ is uniquely built from $\{u,v\}$. 
\end{enumerate}
\end{lemma}

\begin{proof} $(\Rightarrow)$ Suppose $X=X_{\boldsymbol{\tau}}$, where $\boldsymbol{\tau}=(\tau_n\colon A^*_{n+1}\to A^*_n)_{n\geq 0}$ is a constant-length, primitive, proper and recognizable directive sequence with $A_0=\mathsf{A}$ and $A_{n+1}=\{1, 2\}$ for every $n\in\mathbb{N}$. For each $n\in\mathbb{N}$ and $j=1,2$, let $w_{n,j}=\tau_{[0,n+1)}(j)$. Fix $p\,|\,\mathsf{u}$ and $m\in \mathbb{N}$. By Lemma~\ref{gcd}, there is $n\in\mathbb{N}$ such that $|w_{n,1}|=|w_{n,2}|>m$ and $p$ divides $|w_{n,1}|$. Since $\boldsymbol{\tau}$ is proper, we have that the lengths of the maximal common prefix and suffix of $w_{n+1,1}$ and $w_{n+1,2}$ are greater than $|w_{n,1}|=|w_{n,2}|> m$. Let $q=|w_{n+1,1}|$, $u=w_{n+1,1}$ and $v=w_{n+1,2}$. Then $p$ divides $q$. By Lemma \ref{gcd} again, $q$ divides $\mathsf{u}$. Thus (i)--(iii) are satisfied. By the recognizability of $\boldsymbol{\tau}$, we have that for any $k\geq 1$, $\tau_k$ is recognizable in $X^{(k+1)}_{\boldsymbol{\tau}}$. By \cite[Lemma 3.5]{BSTY19}, $\tau_{[0,n+1)}$ is recognizable in $X^{(n+1)}_{\boldsymbol{\tau}}$. By Lemma~\ref{lem:rec} there is a positive integer $r$ such that for any $x, x'\in X$, if $x[-r, r)=x'[-r,r)$, then in the unique buildings of $x$ and $x'$, the occurrence of the word $u$ or $v$ containing position $0$ in $x$ coincide with the occurrence of the same word containing position $0$ in $x'$. Since $|u|=|v|$, it follows that for any $x, x'\in X$, if $x[-r,r)=x'[-r,r)$, then the unique buildings of $x$ and $x'$ coincide entirely on $[-r, r)$. Thus every word in $L_{2r}(X)$ is uniquely built from $\{u,v\}$. The argument still works with $r$ replaced by $2r$, and thus we also get $L_{4r}(X)\subseteq \{u,v\}^+_1$.

$(\Leftarrow)$ Suppose $p\,|\,\mathsf{u}$ and $m\in\mathbb{N}$. Let $q, r$ and $u, v$ satisfy (i)--(iv). Fix $x\in X$. We find a strong rank-2 cut $(q,t)$ of $x$ such that $\{x[t+kq, t+(k+1)q)\colon k\in\mathbb{Z}\}=\{u,v\}$. Then $(X,S)$ is a strong Toeplitz subshift of rank 2 by Theorem~\ref{thm:STS2}. 

We first claim that for any positive integer $n$, $L_{2nr}(X)\subseteq \{u,v\}^+_1$. This is proved by induction on $n$. The cases of $n=1$ and $n=2$ are given by (iv). Now suppose $n\geq 2$ and $L_{2r}(X)\cup L_{4r}(X)\cup\dots \cup L_{2nr}(X)\subseteq \{u, v\}^+_1$ by the inductive hypothesis. Consider $w\in L_{2(n+1)r}(X)$. Then $w[0,2nr)\in L_{2nr}(X)$, $w[2(n-1)r, 2nr)\in L_{2r}(X)$, and $w[2(n-1)r, 2(n+1)r)\in L_{4r}(X)$. By our assumption, each of these words is uniquely built from $\{u, v\}$. By Lemma~\ref{lem:unique}, $w[0,2(n+1)r)$ is uniquely built from $\{u,v\}$. This finishes the proof of the claim.

Now, since $x[-r,r)$ is uniquely built from $\{u,v\}$, there is a unique integer $t\in [0,q)$ such that for any $k\in\mathbb{Z}$ such that $[t+kq, t+(k+1)q)\subseteq [-r,r)$, $x[t+kq, t+(k+1)q)\in \{u,v\}$. From the above claim, we get that for any positive integer $n$, $x[-nr, nr)$ is uniquely built from $\{u,v\}$. It follows that for any $k\in\mathbb{Z}$ such that $[t+kq, t+(k+1)q)\subseteq [-nr, nr)$, $x[t+kq, t+(k+1)q)\in \{u, v\}$. Letting $n\to +\infty$, we conclude that $(q,t)$ is a strong rank-2 cut of $x$ with $\{x[t+kq, t+(k+1)q)\colon k\in\mathbb{Z}\}=\{u,v\}$.
\end{proof}

Following Pavlov--Schmieding \cite{PS}, let $\boldsymbol{S}$ be the space of all subshifts with some alphabet $\mathsf{A}\subseteq \mathbb{Z}$. Define a metric $d$ on $\boldsymbol{S}$ by 
$$ d(X, Y)=2^{-\inf\{n\in\mathbb{N}\colon L_n(X)\neq L_n(Y)\} }.$$
Then $d$ is equivalent to the Hausdorff metric on $\boldsymbol{S}$, which makes $\boldsymbol{S}$ a Polish space. Let $\boldsymbol{M}$ be the subset of $\boldsymbol{S}$ consisting of all infinite minimal subshifts. Then by \cite[Theorem 5.7]{PS}, $\boldsymbol{M}$ is a $G_\delta$ subset of $\boldsymbol{S}$, and hence is itself a Polish space. By \cite[Theorem 5.11]{PS},  \cite[Corollary 5.28]{PS} and our Proposition~\ref{prop:equiv} (ii), the set of all Toeplitz subshifts of topological rank $2$ is a generic subset of $\boldsymbol{M}$. In fact, consider the \textbf{universal scale} $\mathsf{u}$, which is defined as
$$ \mathsf{u}=\displaystyle\prod_{p\in P}p^\infty. $$
Then it was proved in \cite{PS} that the set of all Toeplitz subshifts with the universal scale and topological rank 2 is a generic subset of $\boldsymbol{M}$.

\begin{theorem}\label{thm:gen} The subset of $\boldsymbol{M}$ consisting of all strong Toeplitz subshifts of rank $2$ is generic in $\boldsymbol{M}$.
\end{theorem}

\begin{proof} By \cite[Theorem 5.4]{PS}, it suffices to show that 
\begin{enumerate}
\item[(a)] the class of all strong Toeplitz subshifts of rank $2$ is closed under any injective, constant-length morphism, i.e., if $\tau\colon A^*\to B^*$ is an injetive, constant-length morphism and $X\subseteq A^\mathbb{Z}$ is a strong Toeplitz subshift of rank $2$, then $\bigcup_{k\in\mathbb{Z}}S^k\tau(X)$ is a strong Toeplitz subshift of rank $2$; and
\item[(b)] the class of all strong Toeplitz subshifts of rank $2$ is a relatively $G_\delta$ subset of the class of all elements of $\boldsymbol{M}$ which are Toeplitz subshifts with the universal scale and topological rank $2$.
\end{enumerate}

For (a), suppose $X=X_{\boldsymbol{\tau}}$, where $\boldsymbol{\tau}=(\tau_n\colon A_{n+1}^*\to A_n^*)_{n\geq 0}$ is a constant-length, primitive, proper and recognizable directive sequence with $|A_{n+1}|=2$ for all $n\in\mathbb{N}$. Let $\tau\colon A_0^*\to B^*$ be an injective, constant-length morphism. Define $B_0=B$ and $B_{n+1}=A_{n+1}$ for all $n\in\mathbb{N}$. Define $\tau'_0\colon B_1^*\to B_0^*$ by $\tau'_0=\tau\circ \tau_0$. For all $n\in\mathbb{N}$, let $\tau'_{n+1}=\tau_{n+1}$. Then $\boldsymbol{\tau}'=(\tau'_n\colon B_{n+1}^*\to B_n^*)_{n\geq 0}$ is a constant-length, primitive and proper directive sequence. By \cite[Theorem 4.6]{BSTY19}, $\boldsymbol{\tau}'$ is recognizable. Now it is clear that the subshift $\bigcup_{k\in\mathbb{Z}}S^k\tau(X)$ is generated by $\boldsymbol{\tau}'$, and thus it is a strong Toeplitz subshift of rank $2$. 

For (b), let $\boldsymbol{T}_{\infty, 2}$ denote the subset of $\boldsymbol{S}$ consisting of all Toeplitz subshifts with the universal scale $\mathsf{u}$ and topological rank $2$. Then by Lemma~\ref{lem:STS2}, the set of all strong Toeplitz subshifts of rank $2$ is the intersection of the following set with $\boldsymbol{T}_{\infty, 2}$:
$$ \displaystyle\bigcap_{p, m\in\mathbb{N}}\bigcup_{p\,|\, q, \\ u, v\in L_1(X)^q}\bigcup_{r>q} \left\{X\in\boldsymbol{M}\colon L_{2r}(X)\cup L_{4r}(X)\subseteq \{u,v\}^+_1\right\}. $$ 
Since this set is $G_\delta$ in $\boldsymbol{M}$, we conclude that the subset of $\boldsymbol{M}$ consisting of all strong Toeplitz subshifts of rank $2$ is a relatively $G_\delta$ in $\boldsymbol{T}$.
\end{proof}

\begin{corollary} The set of all strong Toeplitz subshifts of finite rank is generic in the space of all infinite minimal subshifts.
\end{corollary}

\section{The Complexity of Classification Problems\label{sec:5}}
In this section we consider some classification problems for Toeplitz subshfits of topological rank $2$. In Subsection~\ref{sec:5.1} we prove the main technical characterization for the conjugacy problem for Toeplitz subshifts of topological rank $2$. In Subsection~\ref{sec:5.2} we deduce that the conjugacy problem, the flip conjugacy problem and the bi-factor problem are all hyperfinite.

\subsection{The conjugacy problem\label{sec:5.1}} We will work with some additional concepts in word combinatorics. 

Let $\mathsf{A}$ be a finite alphabet and let $u, v\in \mathsf{A}^*$ be nonempty words. Following \cite{GJJLLSW25}, we say that $\{u, v\}$ is a \textbf{Euclidean pair} if there exist $w\in \mathsf{A}^*$ and $k,\ell\in\mathbb{N}$ such that $u=w^k$ and $v=w^{\ell}$; otherwise $\{u, v\}$ is a \textbf{non-Euclidean pair}.

The following lemma appeared as \cite[Lemma 7.1]{DDMP21} and is a weak form of the Fine--Wolf theorem.

\begin{lemma}[\cite{DDMP21}]\label{lem:distinguish}
Let $\{u, v\}\subseteq \mathsf{A}^*$ be a non-Eucliean pair. Then there exist $n<|u|+|v|$ and $\alpha\in\mathsf{A}^n$ such that for any $x\in u\{u,v\}^*$ and $y\in v\{u,v\}^*$ with $|x|, |y|\geq n+1$, $\alpha$ is a common prefix of $x$ and $y$, and $x(n)\neq y(n)$. 
\end{lemma}

We call this word $\alpha$ the \textbf{distinguished prefix} of $u$ and $v$, similarly, the \textbf{distinguished suffix} of $u$ and $v$ is also well defined. Note that the distinguished prefix of $u$ and $v$ must be longer than or equal to the maximal common prefix of $u$ and $v$, similarly for the distinguished suffix of $u$ and $v$.

Let $u, v\in \mathsf{A}^*$ be nonempty words. We consider the morphism $\tau_{u,v}\colon \{0,1\}^*\to \mathsf{A}^*$ given by $\tau_{u,v}(0)=u$ and $\tau_{u,v}(1)=v$. By \cite[Theorem 3.1]{BSTY19}, $\tau_{u,v}$ is fully recognizable for aperiodic points. Let $X_{u,v}$ denote the subshift 
$$\bigcup_{k\in\mathbb{Z}}S^k\tau_{u,v}(\{0,1\}^\mathbb{Z}). $$ 
Then $x\in X_{u,v}$ if and only if $x$ is built from $\{u, v\}$. It is also easy to see that $\{u,v\}$ is a non-Eucliean pair if and only if $X_{u,v}$ contains an aperiodic element. 

    \begin{lemma} \label{lem:lr}
Let $\mathsf{A}$ be a finite alphabet and let $\{u,v\}\subseteq \mathsf{A}^*$ be a non-Euclidean pair.  Suppose the distinguished prefix and suffix of $u$ and $v$ are $\alpha$ and $\beta$ respectively. Then there exists $R\in\mathbb{N}$ such that for any aperiodic $x, x'\in X_{u,v}$, for any $r>R$ and any $i,i'\in \mathbb{Z}$, if $x[i, i+r)=x'[i', i'+r)$ but $x(i-1)\neq x'(i'-1)$ and $x(i+r)\neq x'(i'+r)$, then $x[i, i+r)$ and $x'[i', i'+r)$ have common prefix $\beta$ and common suffix $\alpha$, and the unique building of $x$ from $\{u, v\}$ on $[i+|\beta|,i+r-|\alpha|)$ coincides with the unique building of $x'$ from $\{u, v\}$ on $[i'+|\beta|, i'+r-|\alpha|)$.
    \end{lemma}

\begin{proof} Let $r_0$ be the positive integer obtained by applying Lemma~\ref{lem:rec} to $\tau_{u,v}$. Let $R=2r_0$. Suppose $x, x'\in X_{u,v}$ are aperiodic. 
By shifting $x$ and $x'$ if necessary, we may assume without loss of generality that for $a\leq -r_0$ and $b\geq r_0$, $x[a,b)=x'[a,b)$ but $x(a-1)\neq x'(a-1)$ and $x(b)\neq x'(b)$. Since $\tau_{u,v}$ is fully recognizable for aperiodic points, both $x$ and $y$ have unique centered $\tau_{u,v}$-representations, which we denote as $(y,k)$ and $(y',k')$ respectively. Since $x[-r_0,r_0)=x'[-r_0,r_0)$, we have $k=k'$ and $y(0)=y'(0)$ by Lemma~\ref{lem:rec}. 

Let $\ell\leq 0$ be the least and $m>0$ be the largest such that $y[\ell, m)=y'[\ell,m)$. Thus $y(\ell-1)\neq y'(\ell-1)$ and $y(m)\neq y'(m)$. Let $p$ be the beginning position of the occurrence of $\tau(y(m))$ in $x$. Then $p$ is also the beginning position of the occurrence of $\tau(y'(m))$ in $x'$. By Lemma~\ref{lem:distinguish} and our assumption, we must have that $p+|\alpha|=b$ and $x[p, b)=x'[p, b)=\alpha$. Similarly, let $q$ be the beginning position of the occurrence of $\tau(y(\ell))$ in $x$. Then $q$ is also the beginning position of the occurrence of $\tau(y'(\ell))$ in $x'$. By Lemma~\ref{lem:distinguish}, we must have that $a+|\beta|=q$ and $x[a, q)=x'[a,q)=\beta$.  This proves the lemma.
\end{proof}

We will also use the following observation about $p$-holes in conjugate subshifts.

  \begin{lemma}\label{lem:hole}
         Let $X,Y$ be Toeplitz subshifts, let $\varphi \colon X \to Y$ be a conjugacy map, let $x \in X$ and let $p$ be a positive integer. Recall that $|\varphi|$ is the larger value between the width of a block code for $\varphi$ and the width of a block code for $\varphi^{-1}$. 
    If $i\in\mathbb{Z}$ is a $p$-hole of $x$, then there is $j\in\mathbb{Z}$ such that $|i-j| <|\varphi|$ and 
    $j$ is a $p$-hole of $\varphi(x)$.
    \end{lemma}

\begin{proof} Let $C$ be a block code for $\varphi^{-1}$ of width $|\varphi|=2n+1$. Toward a contradiction, assume that none of the integers $j$, where $i-|\varphi|<j<i+|\varphi|$, is a $p$-hole of $\varphi(x)$. Then $[i-|\varphi|+1, i+|\varphi|)\subseteq \mbox{\rm Per}_p(\varphi(x))$. It follows from Theorem~\ref{thm:CHL} that for all $k\in\mathbb{Z}$, 
$$\begin{array}{rcl}
x(i)&=&\varphi^{-1}(\varphi(x))(i) \\
&=&C(\varphi(x)[i-n, i+n+1)) \\
&=&C(\varphi(x)[i+kp-n, i+kp+n+1)) \\
&=& \varphi^{-1}(\varphi(x))(i+kp) \\
&=& x(i+kp). 
\end{array}$$
This contradicts the assumption that $i$ is a $p$-hole of $x$.
\end{proof}

We will use some notations from Kaya \cite{Kaya} and define some new ones.

Let $X=\overline{\mathcal{O}(x)}$ be a Toeplitz subshift generated by an aperiodic Toeplitz sequence $x\in\mathsf{A}^\mathbb{Z}$. Let $p$ be an essential period of $X$. Recall that ${\rm Parts}(X, p)$ is the partition of $X$ according to the $p$-skeletons, i.e.,
$$ {\rm Parts}(X, p)=\{ W\colon \exists 0\leq k<p\ (y\in W\mbox{ if and only if } {\rm Skel}(y,p)=S^k{\rm Skel}(x,p))\}. $$
Furthermore, we define ${\rm Parts}_*(X,p)$ to be the following subset of ${\rm Parts}(X,p)$:
     $$
      {\rm Parts}_*(X,p)=\{W\in {\rm Parts}(X,p)\colon  {\rm Skel}(W,p)(-1)= \square,\  {\rm Skel}(W,p)(0)\neq \square\}.
     $$
     For $W\in {\rm Parts}_*(X, p)$, define ${\rm length}(W)$ to be the least $i\in\mathbb{N}$ such that ${\rm Skel}(W,p)(i)=\square$. Let 
$$ \ell(X, p)=\max\{ {\rm length}(W)\colon W\in {\rm Parts}_*(X,p) \} $$
and define 
$$ {\rm Parts}_M(X,p)=\{W\in {\rm Parts}_*(X,p)\colon {\rm length}(W)=\ell(X, p) \}. $$
Finally, let 
$$ \chi(X,p)=\left\{S^{\lfloor\frac{\ell(X,p)}{2}\rfloor}W\colon W\in {\rm Parts}_M(X,p)\right\}. $$
Roughly speaking, any element $W\in {\rm Parts}(X,p)$ corresponds to a unique $p$-skeleton, and those in ${\rm Parts}_*(X,p)$ have their blocks aligned properly at position $0$. The elements of ${\rm Parts}_M(X, p)$ are those in ${\rm Parts}_*(X,p)$ with the longest blocks aligned at position $0$, and then $\chi(X,p)$ are obtained from ${\rm Parts}_M(X,p)$ by recentering the longest blocks. Note that all these objects are finite sets of $K(\mathsf{A}^\mathbb{Z})$.

We will use the following result from \cite{Kaya} (\cite[Lemma 9]{Kaya}), which gives more details than Theorem~\ref{thm:DKL} does.

\begin{lemma}[\cite{Kaya}]\label{lem:kaya} Let $\mathsf{A}$ be a finite alphabet, let $x, y\in\mathsf{A}^\mathbb{Z}$ be Toeplitz sequences, and let $X, Y$ be Toeplitz subshifts generated by $x, y$, respectively. Suppose $\varphi\colon X\to Y$ is a conjugacy map with $\varphi(x)=y$. Then for any positive integer $p$ such that $[-|\varphi|, |\varphi|]\subseteq \mbox{\rm Per}_p(x)\cap \mbox{\rm Per}_p(y)$, there exists $\phi\in \mbox{\rm Sym}(\mathsf{A}^p)$ such that $y=\widehat{\phi}(x)$, i.e., for all $k\in \mathbb{Z}$,
$$ y[kp, (k+1)p)=\phi(x[kp, (k+1)p)). $$
\end{lemma}

The following fact is also essentially from \cite{Kaya}. We give a proof for the convenience of the reader.

\begin{lemma}\label{lem:Part} Let $X, Y$ be aperiodic Toeplitz subshifts, let $\varphi\colon X\to Y$ be a conjugacy map, let $p$ be a positive integer, and let $W\in \mbox{\rm Parts}(X, p)$. Suppose for all integer $i\in [-2|\varphi|, 2|\varphi|]$, $\mbox{\rm Skel}(W, p)(i)\neq\square$. Then $\varphi(W)\in {\rm Parts}(Y, p)$ and there exists $\phi\in\mbox{\rm Sym}(\mathsf{A}^p)$ such that $\varphi(W)=\widehat{\phi}(W)$.
\end{lemma}

\begin{proof} Let $x\in X$ be a Toeplitz sequence. It is easy to see that $\varphi(x)$ is also a Toeplitz sequence. Without loss of generality, assume $W=\overline{A(x,p,0)}=\overline{\{S^{kp}(x)\colon k\in\mathbb{Z}\}}$. Then $\varphi(W)=\overline{A(\varphi(x),p,0)}\in {\rm Parts}(Y, p)$. Also, $x\in W$, and therefore $[-2|\varphi|, 2|\varphi|]\subseteq \mbox{\rm Per}_p(x)$. It follows that $[-|\varphi|, |\varphi|]\subseteq \mbox{\rm Per}_p(\varphi(x))$. Thus by Lemma~\ref{lem:kaya}, there exists $\phi\in\mbox{\rm Sym}(\mathsf{A}^p)$ such that $\varphi(x)=\widehat{\phi}(x)$.

Note that $\widehat{\phi}$ commutes with $S^p$. It follows that 
$$ \varphi(W)=\varphi(\overline{A(x,p,0)})=\overline{\varphi(A(x,p,0))}=\overline{\widehat{\phi}(A(x,p,0))}=\widehat{\phi}(\overline{A(x,p,0)})=\widehat{\phi}(W). $$
\end{proof}

We are now ready to prove the main technical result of this subsection.
  
    \begin{theorem}\label{thm:main}
      Let $(X,S)$ and $(Y,S)$ be aperiodic Toeplitz subshifts of topological rank $2$. Suppose $(p_n)_{n\geq 0}$ is an increasing sequence of positive integers such that $p_n\,|\, p_{n+1}$ for all $n\in\mathbb{N}$ and ${\rm lcm}(p_n)_{n\geq 0}$ is the scale for both $(X, S)$ and $(Y, S)$. Then $(X,S)$ and $(Y,S)$ are conjugate if and only if there exists $N\in\mathbb{N}$ such that for all $n\geq N$, $$\chi(X,p_n)\,\, E_{p_n}^{\rm fin}\,\, \chi(Y,p_n).$$
          \end{theorem}
    \begin{proof} $(\Rightarrow)$ We will show that every conjugacy map will eventually preserve the maximality of the length of  $W\in {\rm Parts}_M(X,p_n)$. By Proposition~\ref{prop:equiv} (i), we may assume $X=X_{\boldsymbol{\tau}}$, where $\boldsymbol{\tau}=(\tau_n\colon A_{n+1}^*\to A_n^*)_{n\geq 0}$ is a primitive, proper and recognizable directive sequence with $A_0=\mathsf{A}$ and $A_{n+1}=\{1,2\}$ for all $n\in\mathbb{N}$. For each $n\in\mathbb{N}$ and $j=1,2$, let $w_{n,j}=\tau_{[0,n+1)}(j)$. Let $\varphi\colon X\to Y$ be a conjugacy map between $(X, S)$ and $(Y, S)$. In view of Theorem~\ref{thm:CHL}, let $C, C'$ be block codes for $\varphi, \varphi^{-1}$ respectively, both having width $|\varphi|=2c+1$.

Since $\boldsymbol{\tau}$ is proper, there exists a large enough $m\in\mathbb{N}$ such that the common prefix and the common suffix of $w_{m,1}$ and $w_{m,2}$ are both longer than $2|\varphi|$. Fix such an $m$, and denote the common prefix of $w_{m,1}$ and $w_{m,2}$ by $\alpha$ and the common suffix of $w_{m,1}$ and $w_{m,2}$ by $\beta$. Then $|\alpha|, |\beta|>2|\varphi|=4c+2$. 

Let $\beta_0=\beta[|\beta|-c, |\beta|)$ and $\alpha_0=\alpha[0, c)$. Applying the block code $C$ to $\beta_0w_{m,1}\alpha_0$ and $\beta_0w_{m,2}\alpha_0$, we obtain words $u_1$ and $u_2$ such that $|w_{m,1}|=|u_1|$ and $|w_{m,2}|=|u_2|$. Let $\delta$ be the common prefix of $u_1$ and $u_2$, and let $\gamma$ be the common suffix of $u_1$ and $u_2$. Since $|\alpha|, |\beta|>2|\varphi|$, we have $|\delta|, |\gamma|>|\varphi|$. Similar to the above, we can recover $w_{m,1}$ and $w_{m,2}$ by applying $C'$ in the reverse direction. 

Let $\alpha_w$ be the distinguished prefix of $w_{m,1}$ and $w_{m,2}$ given by Lemma~\ref{lem:distinguish}. Similarly, let $\beta_w$ be the distinguished suffix of $w_{m,1}$ and $w_{m,2}$, let $\alpha_u$ be the distinguished prefix of $u_1$ and $u_2$, and let $\beta_u$ be the distinguished suffix of $u_1$ and $u_2$. By the correspondence between the pairs $\{u_1, u_2\}$ and $\{w_{m,1}, w_{m,2}\}$, we have that $\Bigl\lvert|\alpha_w|-|\alpha_u|\Bigr\rvert<|\varphi|$ and $\Bigl\lvert|\beta_w|-|\beta_u|\Bigr\rvert<|\varphi|$.

 Since every $x\in X$ is built from $\{w_{m,1}, w_{m,2}\}$, and $\varphi$ is surjective, we have that every $y\in Y$ is built from $\{u_1, u_2\}$. Since $X, Y$ are aperiodic, $\{w_{m,1}, w_{m,2}\}$ and $\{u_1, u_2\}$ are non-Euclidean pairs. Also, by \cite[Theorem 3.1]{BSTY19}, both $\tau_{w_{m,1}, w_{m,2}}$ and $\tau_{u_1, u_2}$ are fully recognizaible for aperiodic points. It follows that every $x\in X$ is uniquely built from $\{w_{m,1}, w_{m,2}\}$ and every $y\in X$ is uniquely built from $\{u_1, u_2\}$. Moreover, for any $x\in X$, the unique centered $\tau_{w_{m,1}, w_{m,2}}$-representation of $x$ is exactly the same as the unique centered $\tau_{u_1, u_2}$-representation of $\varphi(x)$.

Let $N\in\mathbb{N}$ be sufficiently large such that 
$$\ell(X,p_N)>4(|w_{m,1}|+|w_{m,2}|)+21|\varphi|. $$
Then by Lemma~\ref{lem:distinguish}, $\ell(X, p_N)>4\max\{|\alpha_w|, |\beta_w|, |\alpha_u|, |\beta_u|\}$.

Now fix $n\geq N$. We have that $\ell(X, p_n)\geq \ell(X, p_N)$. We will show that $\chi(X, p_n)\, E^{\rm fin}_{p_n}\,\chi(Y, p_n)$. However, for the clarity of our discussions, we consider an arbitrary $W_0\in \mbox{\rm Parts}_*(X,p_n)$ where $\mbox{\rm length}(W_0)>4(|w_{m,1}|+|w_{m,2}|)$, and let 
$$ W=S^{\left\lfloor\frac{{\rm length}(W_0)}{2}\right\rfloor}W_0. $$
Let 
$$ r=\left\lfloor \frac{{\rm length}(W_0)}{2}\right\rfloor \mbox{ and } s={\rm length}(W_0)-r. $$

We claim that for all $x, x'\in W$, 
\begin{enumerate}
\item[(i)] $x[-r, -r+|\beta_w|)=x'[-r, -r+|\beta_w|)=\beta_w$;
\item[(ii)] $x[s-|\alpha_w|, s)=x'[s-|\alpha_w|,s)=\alpha_w$;
\item[(iii)] the unique building of $x$ from $\{w_{n,1}, w_{n,2}\}$ on $[-r+|\beta_w|, s-|\alpha_w|)$ coincides with the unique building of $x'$ from $\{w_{m,1}, w_{m,2}\}$ on $[-r+|\beta_w|, s-|\alpha_w|)$.
\end{enumerate}
To prove the claim, let $x, x'\in W$. Then $x[-r,s)=x'[-r,s)$. Since ${\rm Skel}(W,p_n)(-r-1)=\square$, there exists $\tilde{x}\in W$ such that $x(-r-1)\neq \tilde{x}(-r-1)$. Suppose $t\geq s$ is the largest such that $x[-r, t)=\tilde{x}[-r,t)$. Then $x(t)\neq \tilde{x}(t)$. By Lemma~\ref{lem:lr}, we get that $x[-r, t)$ and $\tilde{x}[-r, t)$ have common prefix $\beta_w$ and common suffix $\alpha_w$, and the unique building of $x$ from $\{w_{m,1}, w_{m,2}\}$ on $[-r+|\beta_w|, t-|\alpha_w|)$ coincides with the unique building of $\tilde{x}$ from $\{w_{m,1}, w_{m,2}\}$ on $[-r+|\beta_w|, t-|\alpha_w|)$. In particular, $x[-r, r+|\beta_w|)=\beta_w$. Repeating this argument for $x'$, we get (i). Since ${\rm Skel}(W, p_n)(s)=\square$, a similar argument yields (ii). For (iii), note that either $x'(-r-1)\neq x(-r-1)$ or $x'(-r-1)\neq \tilde{x}(-r-1)$. Thus, in either case, Lemma~\ref{lem:lr} gives that the unique building of $x'$ from $\{w_{m,1}, w_{m,2}\}$ on $[-r+|\beta_w|, s-|\alpha_w|)$ coincides with the unique building of $x$ from $\{w_{m,1}, w_{m,2}\}$ on $[-r+|\beta_w|, s-|\alpha_w|)$. This completes the proof of the claim.

Now it follows from the claim that for all $x, x'\in W$,
\begin{enumerate}
\item[(iv)] $\varphi(x)[-r+|\beta_w|, s-|\alpha_w|)=\varphi(x')[-r+|\beta_w|, s-|\alpha_w|)$;
\item[(v)] the unique building of $\varphi(x)$ from $\{u_1, u_2\}$ on $[-r+|\beta_w|, s-|\alpha_w|)$ coincides with the unique building of $\varphi(x')$ from $\{u_1, u_2\}$ on $[-r+|\beta_w|, s-|\alpha_w|)$;
\item[(vi)] $\varphi(x)[-r+|\beta_w|-|\beta_u|, -r+|\beta_w|)=\varphi(x')[-r+|\beta_w|-|\beta_u|, -r+|\beta_w|)=\beta_u$;
\item[(vii)] $\varphi(x)[s-|\alpha_w|, s-|\alpha_w|+|\alpha_u|)=\varphi(x')[s-|\alpha_w|,s-|\alpha_w|+|\alpha_u|)=\alpha_u$.
\end{enumerate}    
Since $W$ is closed under $S^p$, we also get that for any $x\in W$, 
$$ \left[-r+|\beta_w|-|\beta_u|, s-|\alpha_w|+|\alpha_u|\right)\subseteq {\rm Per}_{p_n}(\varphi(x)). $$
Next we verify that 
$$ -r+|\beta_w|-|\beta_u|-1\not\in {\rm Per}_{p_n}(\varphi(x)) \mbox{ and } s-|\alpha_w|+|\alpha_u|\not\in {\rm Per}_{p_n}(\varphi(x)). $$
Toward a contradiction, let $t=-r+|\beta_w|-|\beta_u|-1$ and assume $t\in {\rm Per}_{p_n}(\varphi(x))$. Since $\varphi(x)$ is uniquely built from $\{u_1, u_2\}$, it follows from (v) that for some $u\in\{u_1,u_2\}$, $$\varphi(x)[t-|u|+1+kp_n, t+1+kp_n)=u$$ for all $k\in\mathbb{Z}$. In particular, $$[t-|u|+1, t+1)\subseteq {\rm Per}_{p_n}(\varphi(x)).$$ However, since $t$ is a $p_n$-hole of $x$, we know by Lemma~\ref{lem:hole} that $\varphi(x)$ has a $p_n$-hole within the interval $[t-|\varphi|+1, t+|\varphi|)$. Since $|u|>|\varphi|$, this is a contradiction. Thus $t\not\in {\rm Per}_{p_n}(\varphi(x))$. The proof of $s-|\alpha_w|+|\alpha_u|\not\in {\rm Per}_{p_n}(\varphi(x))$ is similar.

By Lemma~\ref{lem:Part}, $\varphi(W)\in {\rm Parts}(Y, p_n)$. Thus, to summarize, we have shown that $$\varphi(W)=S^{r-|\beta_w|+|\beta_u|}Z_0, $$
where $Z_0\in {\rm Parts}_*(Y, p_n)$ with
$$ {\rm length}(Z_0)= s-|\alpha_w|+|\alpha_u|+r-|\beta_w|+|\beta_u|={\rm length}(W_0)-|\alpha_w|+|\alpha_u|-|\beta_w|+|\beta_u|. $$
Note that 
$$ \bigl\lvert{\rm length}(Z_0)-{\rm length}(W_0)\bigr\rvert<2|\varphi|. $$
Let 
$$ Z=S^{\left\lfloor\frac{{\rm length}(Z_0)}{2}\right\rfloor}Z_0. $$
Then $\varphi(W)=S^iZ$ where
$$ i=r-|\beta_w|+|\beta_u|-\left\lfloor\frac{{\rm length}(Z_0)}{2}\right\rfloor.$$
We have $|i|<2|\varphi|$.

By a symmetric argument, we also get that for any $Z_0\in {\rm Parts}_*(Y, p_n)$ where ${\rm length}(Z_0)>4(|u_1|+4|u_2|)=4(|w_{m,1}|+|w_{m,2}|)$, there is $W_0\in {\rm Parts}_*(X, p_n)$ with 
$${\rm length}(W_0)={\rm length}(Z_0)+|\alpha_w|-|\alpha_u|+|\beta_w|-|\beta_u|$$ 
so that $\varphi^{-1}(Z)=S^jW$ for some $|j|<2|\varphi|$. 

Finally, suppose $W\in \chi(X, p_n)$. Then $${\rm length}(W_0)=\ell(X, p_n)> 4(|w_{m,1}|+|w_{m,2}|)+21|\varphi|,$$ and we get that ${\rm length}(Z_0)>4(|u_1|+|u_2|)$. In this case, we can conclude that $Z\in \chi(Y, p_n)$. This is because, assume there is some $Z_0'\in \mbox{\rm Parts}_*(Y,p_n)$ with ${\rm length}(Z_0')>{\rm length}(Z_0)$, then we would get $W_0'\in\mbox{\rm Parts}_*(X, p_n)$ with ${\rm length}(W_0')>{\rm length}(W_0)$, which is a contradiction.

Since $|i|<2|\varphi|$, we have that $|S^{-i}\varphi|<5|\varphi|$. Since $\ell(X, p_n)>21|\varphi|$, we have that for all integer $j\in [-2|S^{-i}\varphi|, 2|S^{-i}\varphi|]$, $\mbox{\rm Skel}(W, p_n)(j)\neq \square$. Thus by Lemma~\ref{lem:Part}, we conclude that there is $\phi\in\mathsf{A}^{p_n}$ such that $Z=S^{-1}\varphi(W)=\widehat{\phi}(W)$. This shows that $Z\, E_{p_n}\, W$. 

Thus we have shown that for all $W\in \chi(X, p_n)$, there is $Z\in \chi(Y, p_n)$ such that $Z\, E_{p_n}\, W$. By a symmetric argument, we can get that for all $Z\in \chi(Y, p_n)$, there is $W\in \chi(X, p_n)$ such that $W\, E_{p_n}\, Z$. Thus $\chi(X, p_n)\, E^{{\rm fin}}_{p_n}\, \chi(Y, p_n)$.

       $(\Leftarrow)$ Supose $\chi(X,p_n)\,\, E_{p_n}^{\rm fin}\,\, \chi(Y,p_n)$. Take any $W\in \chi(X,p_n)$ and any Topelitz sequence $x\in W$. There is $Z\in \chi(Y, p_n)$ such that $W\, E_p\, Z$. Let $\phi\in\mathsf{A}^{p_n}$ be such that $Z=\widehat{\phi}(W)$. Then $\widehat{\phi}(x)$ is a Toeplitz sequence in $Y$. By Theorem~\ref{thm:DKL}, $X$ and $Y$ are conjugate.
        \end{proof}

\subsection{The hyperfiniteness of classification problems\label{sec:5.2}} Recall from Subsection~\ref{sec:4.2} that $\boldsymbol{S}$ is the Polish space of all subshifts with some alphabet $\mathsf{A}\subseteq \mathbb{Z}$. Let $\boldsymbol{T}$ denote the subspace of $\boldsymbol{S}$ consisting of all aperiodic Toeplitz subshifts and let $\boldsymbol{T}_2$ denote the subspace of  $\boldsymbol{S}$ consisting of all aperiodic Toeplitz subshifts of topological rank $2$. Both $\boldsymbol{T}$ and $\boldsymbol{T}_2$ are Borel subsets of $\boldsymbol{S}$, and therefore they are standard Borel spaces. 

On the other hand, consider the Polish space $\mathbb{N}^{\mathbb{N}}$ with the product topology. Let
$$ \boldsymbol{P}=\left\{ (p_n)_{n\geq 0}\in \mathbb{N}^{\mathbb{N}}\colon p_n>0 \mbox{ and } p_n\,|\,p_{n+1} \mbox{ for all $n\in\mathbb{N}$}\right\}. $$
Then $\boldsymbol{P}$ is a closed subset of $\mathbb{N}^{\mathbb{N}}$, and hence a Polish space. Also let $\boldsymbol{N}$ be the set of all supernatural numbers. Consider the bijection $\mu\colon (\mathbb{N}\cup\{\infty\})^\mathbb{N}\to \boldsymbol{N}$ defined by
$$ \mu((a_n)_{n\geq 0})=\prod_n p_n^{a_n},$$
where $p_n$ is the $(n+1)$-th prime number in the increasing order, i.e., $p_0=2$, $p_1=3$, etc. Since $(\mathbb{N}\cup\{\infty\})^\mathbb{N}$ with the product topology of discrete topologies is a Polish space, we naturally regard $\boldsymbol{N}$ as a Polish space as well. It is easy to see that there is a Borel map $\pi\colon \boldsymbol{N}\to \boldsymbol{P}$ such that for all $\mathsf{u}\in\boldsymbol{N}$, ${\rm lcm}(\pi(\mathsf{u}))=\mathsf{u}$ (for details, see \cite[Section 2.6]{Kaya}).

The following lemma follows from the main result of \cite{Williams1984}.

\begin{lemma}[\cite{Williams1984}]\label{lem:mef} There is a Borel map $\delta\colon \boldsymbol{T}\to \boldsymbol{N}$ such that for any $X\in \boldsymbol{T}$, ${\rm Odo}(\delta(X))$ is the maximal equicontinuous factor of $X$. In particular, if $X, Y\in \boldsymbol{T}$ are conjugate, then $\delta(X)=\delta(Y)$. 
\end{lemma}

Let $X$ be an uncountable standard Borel space. Since $X$ is Borel isomorphic to $\mathbb{R}$, we can fix a Borel linear ordering $\prec$ on $X$. Suppose $0\in X$ is a distinguished element of $X$. Let $X^{{\rm fin}}$ be the space of all finite subsets of non-zero elements of $X$. Then we can regard $X^{{\rm fin}}$ as the following Borel subset of $X^\mathbb{N}$:
$$ \bigl\{(x_n)_{n\geq 0}\colon \exists N\in\mathbb{N}\ [\,\forall n\in\mathbb{N}\  (n\leq N\leftrightarrow  x_n\neq 0) \mbox{ and } \forall n<N\ (x_n\prec x_{n+1})\,] \bigr\}. $$
Thus $X^{{\rm fin}}$ becomes a standard Borel space. 

The following result determines the complexity of the conjugacy problem for Toeplitz subshifts of topological rank $2$.

\begin{theorem}\label{thm:con}
The conjugacy relation on $\boldsymbol{T}_2$ is hyperfinite. 
\end{theorem}
\begin{proof} By Theorem~\ref{thm:CHL}, each $X\in \boldsymbol{T}$ is conjugate to at most countably many $Y\in \boldsymbol{T}$. Thus the conjugacy relation on $\boldsymbol{T}$ is countable. 

Let $\theta=\pi\circ \delta$, where $\delta$ is from Lemma~\ref{lem:mef} and $\pi$ is from the paragraph preceding Lemma~\ref{lem:mef}. Then $\theta\colon \boldsymbol{T}\to \boldsymbol{P}$ is Borel, and for conjugate $X, Y\in \boldsymbol{T}$, $\theta(X)=\theta(Y)$. 

Now let $\boldsymbol{K}$ be the space of all compact subsets of $\mathsf{A}^\mathbb{Z}$ for some $\mathsf{A}\subseteq\mathbb{Z}$. Equipped with the Hausdorff metric, $\boldsymbol{K}$ is a Polish space. Let $\varnothing$ be the distinguished element of $\boldsymbol{K}$. Then $\boldsymbol{K}^{{\rm fin}}$ is a standard Borel space. Let $\varnothing$ be the distinguished element of $\boldsymbol{K}^{{\rm fin}}$ again. Then $(\boldsymbol{K}^{{\rm fin}})^{{\rm fin}}$ is again a standard Borel space. 

It is now routine to check that, given any positive integer $p$, the map $X\mapsto \chi(X, p)$ from $\boldsymbol{T}$ to $\boldsymbol{K}^{{\rm fin}}$ is Borel. Let $\mathbb{N}_+$ be the space of all positive integers with the discrete topology. It is also routine to check that the map $\lambda\colon \mathbb{N}_+\times\boldsymbol{K}^{{\rm fin}}\to (\boldsymbol{K}^{\rm fin})^{\rm fin}$ defined by 
$$ \lambda(p, \mathcal{K})=\{[K]_{E_p}\colon K\in \mathcal{K}\} $$
is Borel. It is also clear that for $p\in\mathbb{N}_+$ and $\mathcal{K}, \mathcal{L}\in \boldsymbol{K}^{\rm fin}$, 
$$\mathcal{K}\, E_p^{\rm fin}\, \mathcal{L}\iff \lambda(p, \mathcal{K})=\lambda(p, \mathcal{L}). $$
Thus the map 
$$ X\mapsto \Biggl( \theta(X), \lambda\Bigl( (\theta(X))_n, \chi\bigl(X, (\theta(X))_n\bigr)\Bigr)\Biggr)_{n\geq 0} $$
is a Borel map from $\boldsymbol{T}$ to $Y^\mathbb{N}$, where $Y=\boldsymbol{P}\times (\boldsymbol{K}^{\rm fin})^{\rm fin}$. By Theorem~\ref{thm:main}, this map restricted on $\boldsymbol{T}_2$ is a Borel reduction from the conjugacy relation on $\boldsymbol{T}_2$ to $E_1(Y)$. 

By Theorem~\ref{thm:KL}, the conjugacy relation on $\boldsymbol{T}_2$ is hyperfinite.
\end{proof}

Given Toeplitz subshifts $X, Y\subseteq \mathsf{A}^\mathbb{Z}$, we say that $X$ and $Y$ are \textbf{flip conjugate} if $(X, S)$ is conjugate to either $(Y, S)$ or $(Y, S^{-1})$. 

\begin{corollary} The flip conjugacy relation on $\boldsymbol{T}_2$ is hyperfinite.
\end{corollary}

\begin{proof} Obviously, the flip conjugacy relation has finite index over the conjugacy relation. By Theorem~\ref{thm:con} and 
Lemma~\ref{lem:JKL}, the flip conjugacy relation on $\boldsymbol{T}_2$ is hyperfinite.
\end{proof}

Given Toeplitz subshifts $X, Y\subseteq \mathsf{A}^\mathbb{Z}$, define $X\sim_{\rm bf}Y$ if $(X, S)$ is a factor of $(Y, S)$ and $(Y, S)$ is a factor of $(X, S)$. We call $\sim_{\rm bf}$ the \textbf{bi-factor} relation.

\begin{corollary} The bi-factor relation on $\boldsymbol{T}_2$ is hyperfinite.
\end{corollary}

\begin{proof} By Theorem~\ref{thm:CHL}, the bi-factor relation on $\boldsymbol{T}$ is a countable Borel equivalence relation. By Espinoza \cite[Theorem 1.7]{Es23}, each minimal subshift of finite topological rank has finitely many aperiodic subshift factors up to conjugacy. Thus the bi-factor relation on $\boldsymbol{T}_2$ has finite index over the conjugacy relation. By Theorem~\ref{thm:con} and Lemma~\ref{lem:JKL}, the bi-factor relation on $\boldsymbol{T}_2$ is hyperfinite.
\end{proof}

In this rest of this subsection we show that the conjugacy relation, the flip conjugacy relation, and the bi-factor relation on $\boldsymbol{T}_2$ are all not smooth. By Lemma~\ref{lem:Thomas} it suffices to prove this for the conjugacy relation on $\boldsymbol{T}_2$.

We consider one of the simplest kind of Toeplitz subshifts below. Let $\mathsf{A}$ be a finite alphabet and let $x\in\mathsf{A}^\mathbb{Z}$ be a Toeplitz sequence with a period structure $(p_n)_{n\geq 0}$. Say that $x$ has \textbf{single holes} with respect to $(p_n)_{n\geq 0}$ if for all $n\in\mathbb{N}$, there is exactly one $p_n$-hole of $x$ in the interval $[0,p_n)$. A Toeplitz subshift $X$ is said to have \textbf{single holes} if there is a period structure $(p_n)_{n\geq 0}$ of $X$ and a Toeplitz sequence $x\in X$ which has single holes with respect to $(p_n)_{n\geq 0}$. 

\begin{lemma}\label{lem:sh} Let $X\subseteq \{0,1\}^{\mathbb{Z}}$ be an aperiodic Toeplitz subshift with single holes. Then $X$ is a strong Toeplitz subshift of rank $2$.
\end{lemma}

\begin{proof} Let $x\in X$ be an aperiodic Toeplitz sequence which has single holes with respect to some period structure $(p_n)_{n\geq 0}$. Note that if $(q_n)_{n\geq 0}$ is a subsequence of $(p_n)_{n\geq 0}$, then $x$ still has single holes with respect to $(q_n)_{n\geq 0}$. Without loss of generality, assume $p_{n+1}>3p_n$, and $[-p_n, p_n)\subseteq {\rm Per}_{p_{n+1}}(x)$. Thus the single $p_{n+1}$-hole in the interval $[0, p_{n+1})$ occurs in $[p_n, p_{n+1}-p_n)$. For each $n\in \mathbb{N}$, 
let $W_n=\{x[kp_n, (k+1)p_n)\colon k\in\mathbb{Z}\}$. Then $W_n$ has at most two elements. Since $X$ is aperiodic, each $W_n$ has exactly two elements. If $n\leq m$, then $W_m\subseteq W_n^*$. It follows that there is a constant-length, primitive and proper directive sequence $\boldsymbol{\tau}$ of alphabet rank $2$ that generates the subshift $X$. By \cite[Theorem 4.6]{BSTY19}, $\boldsymbol{\tau}$ is recognizable. Thus $X$ is a strong Toeplitz subshift of rank $2$.
\end{proof}

\begin{proposition} The following relations on the class of all strong Toeplitz subshifts of rank $2$ are not smooth:
\begin{enumerate}
\item[\rm (i)] the conjugacy relation;
\item[\rm (ii)] the flip conjugacy relation;
\item[\rm (iii)] the bi-factor relation.
\end{enumerate}
\end{proposition}

\begin{proof} In the proof of \cite[Theorem 4.2]{Thomas2013}, Thomas constructed a Borel reduction from $E_0$ to the conjugacy relation between aperiodic $\{0,1\}$-Toeplitz subshifts with single holes. By Lemma~\ref{lem:sh}, the conjugacy relation for strong Toeplitz subshifts of rank $2$ is not smooth. This proves (i). The cases of (ii) and (iii) follow from Lemma~\ref{lem:Thomas}.
\end{proof}

\section{The Inverse Problem of Toeplitz Subshifts\label{sec:6}}

In this section we give a description of the Toeplitz subshifts which are conjugate to their inverses, and show that a generic Toeplitz subshift is not conjugate to its inverse.

Let $\mathsf{A}$ be a finite alphabet. For any $u\in\mathsf{A}^*$, define $u^{\bot}\in \mathsf{A}^{|u|}$ by
$$ u^{\bot}(k)=u(|u|-1-k) $$
for $0\leq k<|u|$. We call $u^{\bot}$ the \textbf{reverse} of $u$. For any $x\in \mathsf{A}^\mathbb{Z}$, define $x^{\bot}\in\mathsf{A}^\mathbb{Z}$ by
$$ x^{\bot}(i)=x(-i) $$
for all $i\in\mathbb{Z}$. Again call $x^{\bot}$ the \textbf{reverse} of $x$. Easily, if $y=S^j(x)$ then $y^{\bot}=S^{-j}(x^{\bot})$. If $x$ is a Toeplitz sequence, then so is $x^{\bot}$.

Now suppose $X\subseteq \mathsf{A}^{\mathbb{Z}}$ is a subshift. Then the \textbf{inverse} subshift of $(X, S)$ is $(X, S^{-1})$. If $X=\overline{\mathcal{O}(x)}$ is a Toeplitz subshift generated by a Toeplitz sequence $x\in \mathsf{A}^{\mathbb{Z}}$, then $(X, S^{-1})$ is conjugate to the Toeplitz subshift $\overline{\mathcal{O}(x^{\bot})}$ generated by the Toeplitz sequence $x^{\bot}$; the conjugacy map is the reverse map $y\mapsto y^{\bot}$.

The Inverse Problem asks when a Toeplitz subshift is conjugate to its inverse subshift. We give an answer below.

\begin{definition}\label{def:sym} Let $x\in \mathsf{A}^\mathbb{Z}$ be a Toeplitz sequence. Let $p<q$ be positive integers such that $p\,|\,q$. We say that $x$ has \textbf{nice symmetries} with respect to $(p, q)$ if there is $1<m\leq q+1$ such that
\begin{enumerate}
\item[(a)] for any $0\leq k<q/p$, $$[kp, (k+1)p)\subseteq {\rm Per}_q(x)$$ 
if and only if 
$$[m-(k+1)p, m-kp)\subseteq {\rm Per}_q(x), $$
and
\item[(b)] for any $0\leq k, k'<q/p$, if $$[kp, (k+1)p), [k'p, (k'+1)p)\subseteq {\rm Per}_q(x),$$ then $$x[kp, (k+1)p)=x[k'p, (k'+1)p)$$ if and only if $$x[m-(k+1)p, m-kp)=x[m-(k'+1)p, m-k'p).$$
\end{enumerate}
\end{definition}

Note that all the intervals in the above definition are subintervals of $[-q, q]$. It is easy to see that if $x$ has nice symmetries with respect to $(p_0, q)$ and $p_0\,|\,p_1\,|\,q$, then $x$ has nice symmetries with respect to $(p_1, q)$. In addition, if $p\,|\, q_0\,|\,q_1$ and $x$ has nice symmetries with respect with $(p,q_1)$, witnessed by $1<m_1\leq q_1+1$, then the unique $m_0$ such that $1<m_0\leq q_0+1$ and $q_0\,|\, (m_1-m_0)$ witnesses that $x$ has nice symmetries with respect to $(p, q_0)$. 

\begin{theorem}\label{thm:inv} Let $(X, S)$ be a Toeplitz subshift with scale $\mathsf{u}$. Then the following are equivalent:
\begin{enumerate}
\item[\rm (1)] $(X, S)$ is conjugate to $(X, S^{-1})$;
\item[\rm (2)] For every Toeplitz sequence $x\in X$, there is $p\,|\,\mathsf{u}$ such that for any $q\,|\,\mathsf{u}$ with $p\,|\,q$, $x$ has nice symmetries with respect to $(p, q)$;
\item[\rm (3)] For a nonmeager set of $x\in X$, there is $p\,|\,\mathsf{u}$ such that for any $q\,|\,\mathsf{u}$ with $p\,|\,q$, $x$ has nice symmetries with respect to $(p, q)$;
\item[\rm (4)] There exist a Toeplitz sequence $x\in X$ and $p\,|\,\mathsf{u}$ such that for any $q\,|\,\mathsf{u}$ with $p\,|\,q$, $x$ has nice symmetries with respect to $(p, q)$.
\end{enumerate}
\end{theorem}

\begin{proof} (1)$\Rightarrow$(2). Let $x\in X\subseteq \mathsf{A}^\mathbb{Z}$ be a Toeplitz sequence. Suppose $\varphi$ is a conjugacy map from $X=\overline{\mathcal{O}(x)}$ to $\overline{\mathcal{O}(x^{\bot})}$. By Theorem~\ref{thm:DKL}, there exist a positive integer $p$ and $\phi\in {\rm Sym}(\mathsf{A}^p)$ such that for all $k\in \mathbb{Z}$, $\varphi(x)=\widehat{\phi}(x)$. Take an arbitrary $q\,|\,\mathsf{u}$ such that $p\,|\,q$. Suppose $\varphi(x)\in \overline{A(x^{\bot},q,i)}$ for $0\leq i<q$. Let $m=q-i+1$. We next verify that $m$ witnesses that $x$ has nice symmetries with respect to $(p,q)$.

For Definition~\ref{def:sym} (a), note that for any $0\leq k<q/p$, 
$$\begin{array}{rl}
& [kp, (k+1)p)\subseteq {\rm Per}_q(x) \\
\iff & [kp, (k+1)p)\subseteq {\rm Per}_q(\varphi(x)) \\
\iff & [kp, (k+1)p)\subseteq {\rm Per}_q(S^i(x^{\bot})) \\
\iff & (q-i-(k+1)p, q-i-kp]\subseteq {\rm Per}_q(x) \\
\iff & [m-(k+1)p, m-kp)\subseteq {\rm Per}_q(x).
\end{array}
$$

For Definition~\ref{def:sym} (b), suppose $0\leq k, k'<q/p$ and 
$$ [kp, (k+1)p), [k'p, (k'+1)p)\subseteq {\rm Per}_q(x). $$
Then
$$\begin{array}{rl}
& x[kp, (k+1)p)=x[k'p, (k'+1)p) \\
\iff & \varphi(x)[kp, (k+1)p)=\varphi(x)[k'p, (k'+1)p) \\
\iff & S^i(x^{\bot})[kp, (k+1)p)=S^i(x^{\bot})[k'p, (k'+1)p) \\
\iff & x[q-i-(k+1)p+1, q-i-kp+1) \\
& =x[q-i-(k'+1)p+1, q-i-k'p+1) \\
\iff & x[m-(k+1)p, m-kp) =x[m-(k'+1)p, m-k'p).
\end{array}
$$

(2)$\Rightarrow$(3) is obvious.

(3)$\Rightarrow$(4). Since the set of all Toeplitz sequences in $X$ is comeager, any nonmeager set contains a Toeplitz sequence.

(4)$\Rightarrow$(1). Fix a Toeplitz sequence $x\in X$ and a positive integer $p$ satisfying (3). We define a $\phi\in {\rm Sym}(\mathsf{A}^p)$ so that $\widehat{\phi}(x)\in \overline{\mathcal{O}(x^{\bot})}$. By Theorem~\ref{thm:DKL}, $X$ and $\overline{\mathcal{O}(x^{\bot})}$ are conjugate. 

Fix a sequence $(q_n)_{n\geq 0}$ where $q_0=p$, ${\rm lcm}(q_n)_{n\geq 0}=\mathsf{u}$, and for all $n\in\mathbb{N}$, $q_{n+1}>q_n$ and $q_n\,|\, q_{n+1}$. 

Consider a finite splitting tree $\mathcal{T}$ defined as follows. Each node of $\mathcal{T}$ is a finite sequence $(m_j)_{1\leq j<\ell}$, where $\ell$ is a positive integer, satisfying
\begin{enumerate}
\item[(i)] for all $1\leq j<\ell$, $1<m_j\leq q_j+1$;
\item[(ii)] for all $1\leq j<\ell-1$, $q_j\,|\,(m_{j+1}-m_j)$;
\item[(iii)] for all $1\leq j<\ell$, $m_j$ witnesses that $x$ has nice symmetries with respect to $(p, q_j)$.
\end{enumerate}
By (i), $\mathcal{T}$ is finite splitting. By (3) and the remarks preceding Theorem~\ref{thm:inv}, $\mathcal{T}$ is infinite. By K\"{o}nig's lemma, $\mathcal{T}$ has an infinite branch. Thus there is an infinite sequence $(m_n)_{n\geq 1}$ such that for any positive integer $\ell$, $(m_j)_{1\leq j<\ell}\in \mathcal{T}$. 

We now define $\phi\in{\rm Sym}(\mathsf{A}^p)$. Let $u\in \mathsf{A}^p$. If $u\neq x[ip, (i+1)p)$ for any $i\in\mathbb{Z}$, then define $\phi(u)=u$. Otherwise, assume $u=x[ip, (i+1)p)$. Let $n\geq 1$ be such that $[ip, (i+1)p)\subseteq {\rm Per}_{q_n}(x)$. Define
$$ \phi(u)=\phi(x[ip, (i+1)p))=\bigl( x[m_n-(i+1)p, m_n-ip)\bigr)^{\bot}. $$

To see that $\phi$ is well defined, we need to show that the definition of $\phi$ does not depend on the choice of either $i$ or $n$. Let $i, i'\in\mathbb{Z}$ be such that 
$$x[ip, (i+1)p)=x[i'p, (i'+1)p). $$
Let $n\geq 1$ be such that $[ip, (i+1)p)\subseteq {\rm Per}_{q_n}(x)$, and let $n'\geq 1$ be such that $[i'p, (i'+1)p)\subseteq {\rm Per}_{q_{n'}}(x)$. Without loss of generality, assume $n'\geq n$. Let $0\leq k_n<q_n/p$ be unique such that $q_n\,|\,(k_n-i)$. Then $$x[ip, (i+1)p)=x[k_np, (k_n+1)p). $$ By Definition~\ref{def:sym} (a), $[m_n-(k_n+1)p, m_n-k_np)\subseteq {\rm Per}_{q_n}(x)$. Thus $$x[m_n-(k_n+1)p, m_n-k_np)=x[m_n-(i+1)p, m_n-ip).$$ Similarly, let $0\leq k_{n'}<q_{n'}/p$  be unique such that $q_{n'}\,|\, (k_{n'}-i')$. Then $$x[i'p, (i'+1)p)=x[k_{n'}p, (k_{n'}+1)p)$$ and $$x[m_{n'}-(k_{n'}+1)p, m_{n'}-k_{n'}p)=x[m_{n'}-(i'+1)p, m_{n'}-i'p).$$
Then 
$$ x[k_np, (k_n+1)p)=x[k_{n'}p, (k_{n'}+1)p). $$
By Definition~\ref{def:sym} (b), we have
$$ x[m_{n'}-(k_n+1)p, m_{n'}-k_np)=x[m_{n'}-(k_{n'}+1)p, m_{n'}-k_{n'}p).$$
Since $q_n\,|\, (m_{n'}-m_n)$, we have
$$ x[m_{n'}-(k_n+1)p, m_{n'}-k_np)=x[m_n-(k_n+1)p, m_n-k_np). $$
Thus 
$$ x[m_n-(k_n+1)p, m_n-k_np)=x[m_{n'}-(k_{n'}+1)p, m_{n'}-k_{n'}p). $$
Thus $\phi$ is well defined. 

It remains to show that $\widehat{\phi}(x)\in \overline{\mathcal{O}(x^{\bot})}$. Fix any $c\in\mathbb{N}$. It suffices to find $\ell_c$ such that 
$$ \widehat{\phi}(x)[-cp, cp)=S^{\ell_c}(x^{\bot})[-cp, cp). $$
Let $n\geq 1$ be such that $[-cp, cp)\subseteq {\rm Per}_{q_n}(x)$. Then by the definition of $\widehat{\phi}$, we have
$$ \widehat{\phi}(x)[-cp, cp)=\bigl( x[m_n-cp, m_n+cp)\bigr)^{\bot}. $$
Let $\ell_c=1-m_n$. Then
$$\begin{array}{rcl}
\widehat{\phi}[-cp, cp)& =& \big( x[m_n-cp, m_n+cp)\bigr)^{\bot} \\
&=& x^{\bot}[1-m_n-cp, 1-m_n+cp) \\
&=& S^{\ell_c}(x^{\bot})[-cp, cp). 
\end{array}
$$
\end{proof}

As a corollary, we obtain a criterion for a positive solution of the Inverse Problem for a Toeplitz subshift with single holes. This has previously been proved in Yu \cite{Yu} for the case $\mathsf{A}=\{0,1\}$. We use the following terminology and notation. Let $\theta\in {\rm Sym}(\mathsf{A})$. A word $u\in \mathsf{A}^*$ is a \textbf{$\theta$-palindrome} if $\theta(u)=u^{\bot}$. If $x\in\mathsf{A}$ is a Toeplitz sequence with single holes with respect to a period structure $(p_n)_{n\geq 0}$, then for any $n<m$, the restriction of ${\rm Skel}(x, p_n)$ between two consecutive $p_m$-holes is of the form
$$ u\,\square\, u\,\square \cdots \square\, u; $$
the \textbf{$(n,m)$-filling} is the unique word $w\in \mathsf{A}^*$ so that the ${\rm Skel}(x, p_m)$ between two consecutive $p_m$-holes is
$$ u\,w(0)\,u\,w(1)\cdots w(|w|-1)\,u. $$
Note that the length of the $(n,m)$-filling is $(p_m/p_n)-1$.

\begin{corollary} Let $x\in\mathsf{A}^\mathbb{Z}$ be a Toeplitz sequence with single holes with respect to a period structure $(p_n)_{n\geq 0}$. Then $(\overline{\mathcal{O}(x)}, S)$ is conjugate to $(\overline{\mathcal{O}(x)}, S^{-1})$ if and only if there exist $n\in\mathbb{N}$ and $\theta\in{\rm Sym}(\mathsf{A})$ such that for all $m>n$, the $(n,m)$-filling is a $\theta$-palindrome.
\end{corollary}

\begin{proof} By Theorem~\ref{thm:inv}, $(\overline{\mathcal{O}(x)}, S)$ is conjugate to $(\overline{\mathcal{O}(x)}, S^{-1})$ if and only if there exist $n\in\mathbb{N}$ such that for any $m>n$, $x$ has nice symmetries with respect to $(p_n, p_m)$. It is easy to see that $x$ has nice symmetries with respect to $(p_n, p_m)$ if and only if the $(n,m)$-filling is a $\theta$-palindrome for some $\theta\in\mathsf{A}$. Moreover, $\theta$ does not depend on $m>n$.
\end{proof}

Next we prove that the set of all Toeplitz subshifts which are not conjugate to their inverses is generic in the space of all infinite minimal subshifts. For this, we consider strong Toeplitz subshifts of rank $2$. 

Let $\mathsf{A}$ be a finite alphabet, let $\ell$ be a positive integer, let $W\subseteq \mathsf{A}^\ell$, and let $0\leq i<\ell$. We say that $i$ is a \textbf{coincidence} of $W$ if the $i$-th letters of all words in $W$ are the same. The set of all coincidences of $W$ is denoted ${\rm coinc}(W)$.

\begin{definition}\label{def:sym2} Let $\mathsf{A}$ be a finite alphabet. Let $p<q<r$ be positive integers such that $p\,|\,q$ and $r>3q$.  Let $W\subseteq \mathsf{A}^q$. Let $u\in \mathsf{A}^{2r}\cap W^+_1$ and $w=u[r-q, r+q+1)$. Let $r-q<a\leq r$ be the beginning position of an occurrence of a word in $W$ in the unique building of $u$ from $W$. Let 
$$Q=\left[{\rm coinc}(W^3)\cap [r-a, r-a+2q+1)\right]-(r-a).$$  
We say that $u$ has \textbf{nice symmetries} with respect to $(p, q)$ and $W$ if there is $1<m\leq q+1$ such that
\begin{enumerate}
\item[(a)] for any $0\leq k<q/p$, $$[kp, (k+1)p)\subseteq Q$$ 
if and only if 
$$[q+m-(k+1)p, q+m-kp)\subseteq Q, $$
and
\item[(b)] for any $0\leq k, k'<q/p$, if $$[kp, (k+1)p), [k'p, (k'+1)p)\subseteq Q,$$ then $$w[kp, (k+1)p)=w[k'p, (k'+1)p)$$ if and only if $$w[q+m-(k+1)p, q+m-kp)=w[q+m-(k'+1)p, q+m-k'p).$$
\end{enumerate}
\end{definition}

The point of this definition is that it is a finitary property of a finite word. But we use it to approximate Toeplitz sequences having nice symmetries with respect to $(p, q)$.  

\begin{lemma}\label{lem:STS2inv} Let $\mathsf{A}$ be a finite alphabet. Let $X\subseteq \mathsf{A}^\mathbb{Z}$ be a strong Toeplitz subshift of rank 2. Let $\mathsf{u}$ be the scale of $X$. Then $(X, S)$ is conjugate to $(X, S^{-1})$ if and only if for a nonmeager set of $x\in X$, 
there exists $p\,|\,\mathsf{u}$ such that for all integers $q,r$ and words $u, v\in \mathsf{A}^*$ satisfying (i)--(iv) of Lemma~\ref{lem:STS2}, $x[-r, r)$ has nice symmetries with respect to $(p, q)$ and $\{u,v\}$.
\end{lemma}

\begin{proof} $(\Rightarrow)$ Let $x\in X$ be a Toeplitz sequence. By Theorem~\ref{thm:inv} there is $p\,|\, \mathsf{u}$ such that for all $q>p$ with $p\,|\, q\,|\, \mathsf{u}$, $x$ has nice symmetries with respect to $(p, q)$. Suppose $q, r, u, v$ satisfy (i)--(iv) of Lemma~\ref{lem:STS2}. Then $x$ is uniquely built from $\{u,v\}$. Let $r$ be the integer given by Lemma~\ref{lem:rec} for the recognizability of $\tau_{u,v}$. Without loss of generality we may assume $r>3q$. Let $u=x[-r,r)$. Then the set $Q$ in Definition~\ref{def:sym2} is ${\rm Per}_q(x)\cap[0, 2q+1)$. Thus $x[-r,r)$ has nice symmetries with respect to $(p,q)$ and $\{u,v\}$.

$(\Leftarrow)$ Since the set of Toeplitz sequences in $X$ is comeager, there exists a Toeplitz sequence $x\in X$ satisfying the property stated in the lemma. Fix  $p\,|\,\mathsf{u}$ witnessing this property. Let $q_0>p$ be such that $p\,|\, q_0$. Applying Lemma~\ref{lem:STS2} to $q_0$, we get that there are $q, r, u, v$ satisfying (i)--(iv) of Lemma~\ref{lem:STS2}. By our assumption, $x[-r, r)$ has nice symmetries with respect to $(p, q)$ and $\{u, v\}$. Note that $x[-r, r)\in \{u,v\}^+_1$ by Lemma~\ref{lem:STS2} (iv). We have that $x$ has nice symmetries with respect to $(p, q)$. Since $p\,|\,q_0\,|\, q$, by the remarks preceding Theorem~\ref{thm:inv}, $x$ has nice symmetries with respect to $(p, q_0)$. By Theorem~\ref{thm:inv}, $(X, S)$ is conjugate to $(X, S^{-1})$.
\end{proof}

\begin{theorem} The set of all strong Toeplitz subshifts of rank $2$ which are not conjugate to their inverses is generic in the space of all infinite minimal subshifts.
\end{theorem}

\begin{proof} Let $\boldsymbol{N}$ denote the subset of $\boldsymbol{M}$ consisting of all strong Toeplitz subshifts of rank $2$ which are not conjugate to their inverses. By Theorem~\ref{thm:gen} and \cite[Theorem 5.4]{PS}, it suffices to show that
\begin{enumerate}
\item[(a)] $\boldsymbol{N}$ is closed under any injective, constant-length morphism; and
\item[(b)] $\boldsymbol{N}$ is a relative $G_\delta$ subset of the class of all strong Toeplitz subshifts of rank $2$.
\end{enumerate}

For (a), suppose $X\subseteq A^\mathbb{Z}$ and $X\in \boldsymbol{N}$. Then $X$ is aperiodic. Let $x\in X$ be a Toeplitz sequence. Let $\{u, v\}$ be a non-Eucliean pair such that $x$ is uniquely built from $\{u, v\}$. Let $\tau\colon A^*\to B^*$ be an injective, constant-length morphism. Let $\tau^{\bot}\colon A^*\to B^*$ be the morphism defined by $\tau^{\bot}(a)=\tau(a)^\bot$ for any $a \in A$. Let $y=\tau(x)$. 
Let $Y=\overline{\mathcal{O}(y)}$ and $Y^{\bot}=\overline{\mathcal{O}(y^{\bot})}$. Note that $\tau^\bot(x^\bot)$ is a shift of $y^{\bot}$. Thus $$Y^\bot=\overline{\mathcal{O}(y^\bot)}=\overline{\mathcal{O}(\tau^\bot(x^\bot))}.$$ Also note that every element of $Y$ is uniquely built from $\{\tau(u), \tau(v)\}$, and every element of $Y^{\bot}$ is uniquely built from $\{\tau(u)^{\bot}, \tau(v)^{\bot}\}$.

 Toward a contradiction, assume $Y$ is conjugate to $Y^{\bot}$. Suppose $\varphi\colon Y\to Y^{\bot}$ is a conjugacy map. Without loss of generality, we may assume that in the unique building of $\varphi(y)$ from $\{\tau(u)^{\bot}, \tau(v)^{\bot}\}$, $0$ is the beginning position of an occurrence of either $\tau(u)^{\bot}$ or $\tau(v)^{\bot}$. By Theorem~\ref{thm:DKL}, there exist a positive integer $p$ and a permutation $\phi\in {\rm Sym}(B^p)$ such that $\varphi(y)=\widehat{\phi}(y)$. Without loss of generality, we may assume $|\tau|\,\big|\,p$. Let $p=k|\tau|$. Then $\phi$ can be viewed as a bijection from $\tau(A)^k$ to $\tau^\bot(A)^k$. Define $\psi\in {\rm Sym}(A^k)$ by $\psi(\alpha)=\beta$ if $\phi(\tau(\alpha))=\tau^{\bot}(\beta)$. Then we have $\tau^\bot(\widehat{\psi}(x))=\widehat{\phi}(y)$. Since $\widehat{\phi}(y)=\varphi(y)\in \overline{\mathcal{O}(\tau^\bot(x^\bot))}$, we have that $\widehat{\psi}(x)\in \overline{\mathcal{O}(x^\bot)}$. By Theorem~\ref{thm:DKL}, $(X, S)$ and $(X, S^{-1})$ are conjugate, a contradiction. This proves (a).

For (b), we first note that  the following property 
$$P(X, x, p, q, r, u, v)$$
of parameters $(X, x, p, q, r, u, v)$  is a closed subset of the Polish space $$ \boldsymbol{M}\times \mathsf{A}^\mathbb{Z}\times \mathbb{N}^3\times (\mathsf{A}^*)^2: $$
\begin{quote} $x\in X$, $u=x[-r, r)$, and if $p, q, r, u, v$ satisfy (i)--(iv) of Lemma~\ref{lem:STS2}, then $u$ has nice symmetries with respect to $(p, q)$ and $\{u, v\}$.
\end{quote}
Then by Lemma~\ref{lem:STS2inv}, if $X$ is a strong Toeplitz subshift of rank $2$, then $X\in\boldsymbol{N}$ if and only if 
\begin{quote}
for a comeager set of $x\in X$, for all positive interger $p$ there are $q, r, u, v$ such that $P(X, x, p, q, r, u, v)$ fails.
\end{quote}
By Lemma~\ref{lem:categoryquantifiers} (i), $\boldsymbol{N}$ is a relative $G_\delta$ subset of $\boldsymbol{M}$.
\end{proof}

\begin{corollary} The set of all Toeplitz subshifts which are not conjugate to their inverses is generic in the space of all infinite minimal subshifts.
\end{corollary}

\section{Automorphism Groups\label{sec:8}}

In this section we present some results about automorphism groups of Toeplitz subshifts of finite rank. 

It follows from Donoso--Durand--Maass--Petite \cite[Theorem 3.1]{DDMP16} and \cite[Corollary 7.4]{DDMP21} that if $(X, S)$ is a minimal Cantor system of topological rank $2$, then the automorphism group of $(X, S)$ is exactly $\langle S\rangle$. For $\mathcal{S}$-adic subshifts of finite alphabet rank, it was proved as Espinoza--Maass \cite[Theorem 1.1]{EM} that their automorphism groups are viturally $\mathbb{Z}$.

On the other hand, it follows from Donoso--Durand--Maass--Petite \cite[Theorem 3.2]{DDMP17} (see also \cite[Lemmas 2.1 and 2.4]{DDMP16}) that the automorphism group of a Toeplitz subshift is abelian. 

Combining these results together, we have the following fact. 

\begin{proposition}\label{prop:aut} Let $(X, S)$ be a Toeplitz subshift of finite rank. Let $\mathsf{u}$ be the scale for$(X, S)$. Then the automorphism group of $(X, S)$ is isomorphic to $\mathbb{Z}\oplus (\mathbb{Z}/n\mathbb{Z})$, where $n\,|\, \mathsf{u}$ is such that $p^\infty\!\!\!\not{|}\,\mathsf{u}$ for all prime $p\,|\, n$.
\end{proposition}

\begin{proof} By \cite[Lemmas 2.1 and 2.4]{DDMP16}, the automorphism group of $(X, S)$ is a subgroup of the automorphism group of $({\rm Odo}(\mathsf{u}), S)$, which is in turn a subgroup of the additive group $({\rm Odo}(\mathsf{u}), +)$. It follows that the automorphism group of $(X, S)$ is isomorphic to $\mathbb{Z}\oplus F$, where $F$ is a finite subgroup of $({\rm Odo}(\mathsf{u}), +)$. It is easy to see that the torsion part of $({\rm Odo}(\mathsf{u}), +)$ is of the form $\bigoplus_{p\in T}p^{n_p}$, where $T=\{p\in P\colon p^\infty\!\!\!\not{|}\,\mathsf{u}\}$ and $n_p$ is a positive integer. Any subgroup of it is cyclic and has the form $\mathbb{Z}/n\mathbb{Z}$, where the prime factors of $n$ are in $T$.
\end{proof}

One naturally wonders if the automorphism group of $(X, S)$ is always isomorphic to $\mathbb{Z}$ as in the rank $2$ case. In the following we present an example which gives a negative answer.

\begin{example}
There is a Toeplitz subshift of topological rank at most 4 whose automorphism group is isomorphic to $\mathbb{Z}\oplus (\mathbb{Z}/n\mathbb{Z})$ where $n>1$.
\end{example}
\begin{proof}
For every $u\in \{0,1\}^*$, by $\widetilde{u}$ we denote the word such that $|\widetilde{u}|=|u|$ and $\widetilde{u}(m)=1-u(m)$ for every $0\le m< |u|$. Equivalently, let $\varphi$ be the {\em flip morphism}, i.e., the morphism given by $\varphi(0)=1$ and $\varphi(1)=0$; then $\tilde{u}=\varphi(u)$. $\varphi$ also induces an isomorphism $\varphi^*\colon \{0,1\}^\mathbb{Z}\to \{0,1\}^\mathbb{Z}$, where $\varphi^*(x)(k)=\varphi(x(k))$ for all $k\in\mathbb{Z}$.

Indutively define $\{w_{i,1}, w_{i,2}, w_{i,3}, w_{i,4}\}_{n\geq 0}$ as follows. Let 
$$u_0=0001, \ v_0=0111, $$
and let 
$$\begin{array}{rcl}
w_{0,1}\!\!\!\!&=&\!\!\!\! u_0\,\widetilde{u_0}, \\
w_{0,2}\!\!\!\!&=&\!\!\!\!u_0\,\widetilde{v_0\,}, \\
w_{0,3}\!\!\!\!&=&\!\!\!\!v_0\,\widetilde{u_0}, \\
w_{0,4}\!\!\!\!&=&\!\!\!\!v_0\,\widetilde{v_0\,}. 
\end{array}
$$ 
Suppose for $i\geq 0$ we have defined $u_i$, $v_i$ and $w_{i,j}$ for $1\le j\le 4$, and we have 
$$\begin{array}{rcl}
w_{i,1}\!\!\!\!&=&\!\!\!\!u_i\,\widetilde{u_i\,}, \\
w_{i,2}\!\!\!\!&=&\!\!\!\! u_i\widetilde{\,v_i\,}, \\
w_{i,3}\!\!\!\!&=&\!\!\!\!v_i\,\widetilde{u_i\,}, \\
w_{i,4}\!\!\!\!&=&\!\!\!\!v_i\widetilde{\,v_i\,}.
\end{array}
$$
 Then let $$u_{i+1}=u_i\,\widetilde{u_i\,}\,u_i\widetilde{\,v_i\,}\,v_i\,\widetilde{u_i\,}\,v_i\widetilde{\,v_i\,}\,u_i\,\widetilde{u_i\,}\,u_i$$ and $$v_{i+1}=u_i\,\widetilde{u_i\,}\,u_i\widetilde{\,v_i\,}\,v_i\,\widetilde{u_i\,}\,v_i\widetilde{\,v_i\,}\,v_i\,\widetilde{u_i\,}\,u_i.$$ Note that both $u_{i+1}$ and $v_{i+1}$ begin with $w_{i,1}w_{i,2}w_{i,3}w_{i,4}$ and end with $\widetilde{w_{i,1}}$. We define 
$$\begin{array}{rcl}
w_{i+1,1}\!\!\!\!&=&\!\!\!\!u_{i+1}\,\widetilde{u_{i+1}}, \\
w_{i+1,2}\!\!\!\!&=&\!\!\!\!u_{i+1}\,\widetilde{v_{i+1}}, \\
w_{i+1,3}\!\!\!\!&=&\!\!\!\!v_{i+1}\,\widetilde{u_{i+1}}, \\
w_{i+1,4}\!\!\!\!&=&\!\!\!\!v_{i+1}\,\widetilde{v_{i+1}}.
\end{array}$$
 Then in every $w_{i+1,j}$ where $1\le j\le4$, the $u_i$ or $v_i$ item and the $\widetilde{u_i\,}$ or $\widetilde{\,v_i\,}$ item occur in turn, so $w_{i+1,j}$ is built from $\{w_{i,j}:1\le j\le4\}$. Moreover, the building of $w_{i+1,j}$ from $\{w_{i,j}:1\le j\le4\}$ begins and ends with $w_{i,1}$. 

There is a natural directive sequence $\boldsymbol{\tau}=(\tau_{i}\colon A_{i+1}^*\to A_i^*)_{i\geq 0}$ given by the above definition. We have $\mathsf{A}=A_0=\{0,1\}$ and $|A_{i+1}|=4$ for all $i\geq 0$. Since $\boldsymbol{\tau}$ is constant-length, primitive and proper, $X_{\boldsymbol{\tau}}$ is a Toeplitz subshift by Proposition~\ref{prop:ADE} (5). Obviously $\boldsymbol{\tau}$ has alphabet rank $4$, hence $X_{\boldsymbol{\tau}}$ has topological rank at most $4$ by Theorem~\ref{thm:DDMP21}. 

It is now easy to check that $\varphi^*$ is an automorphism of $X_{\boldsymbol{\tau}}$. Clearly $\varphi^*$ is non-trivial and has order $2$. 
\end{proof}

The next theorem shows that for any $n\geq 1$, the group $\mathbb{Z}\oplus(\mathbb{Z}/n\mathbb{Z})$ can be realized as the isomorphism type of the automorphism group of a Toeplitz subshift of finite rank.

\begin{theorem} \label{thm:aut}
For any $n\geq 1$, there is a strong Toeplitz subshift $(X, S)$ of rank $\max\{2, n^n\}$ such that the automorphism group of $(X, S)$ is isomorphic to $\mathbb{Z}\oplus (\mathbb{Z}/n\mathbb{Z})$. 
\end{theorem}

The rest of this section is devoted to a proof of Theorem~\ref{thm:aut}. 

By \cite[Corollary 7.4]{DDMP21}, any Toeplitz subshift of topological rank $2$ has an automorphism group isomorphic to $\mathbb{Z}$. This is the $n=1$ case. For the rest of this proof, assume $n\geq 2$.

Fix a sufficiently large integer $m$. We first construct a subset $\mathcal{U}$ of finite words in alphabet $\mathsf{B}=\{0,1,2,\dots,n-1\}$ with the following properties:
\begin{enumerate}
\item each word in $\mathcal{U}$ has length $mn+1$, i.e., $U\subseteq \mathsf{B}^{mn+1}$, and $|\mathcal{U}|=n$;
\item for any $u,u'\in \mathcal{U}$, $u[0,n)=u'[0,n)$ and
$$\begin{array}{rcl} u[(m-1)n+1,mn+1)\!\!\!\!&=&\!\!\!\!u((m-1)n, mn] \\
&=&\!\!\!\!u'((m-1)n,mn]=u'[(m-1)n+1,mn+1); 
\end{array}$$
\item for any $\alpha, \beta\in\mathsf{B}^n$ and any $u\in\mathcal{U}$, there is $k<m$ such that $u[kn,(k+2)n)=\alpha\beta$;
\item for any $u,u',u''\in\mathcal{U}$, if $u$ occurs in $u'u''$ at position $i$, then $i=0$ (and $u=u'$) or $i=mn+1$ (and $u=u''$).
\end{enumerate}
We construct the elements of $\mathcal{U}$ as $(u_t)_{0\leq t<n}$. Arbitrarily fix $\alpha_0\in \mathsf{B}^n$ which begins and ends with $01\in \mathsf{B}^*$. Fix a word $\gamma_0\in \mathsf{B}^*$ which is a concatenation of all the words in $\{\alpha\beta\colon \alpha, \beta\in\mathsf{B}^n\}$ without repetititions; in particulat $|\gamma_0|$ is a multiple of $n$. Without loss of generality, we may assume that $\gamma_0$ begins and ends with $\alpha_0$. For each $0\leq t<n$, let $\eta_t$ be the word $t^{6n+1}\in \mathsf{B}^*$. Then define
$$ u_t=\alpha_0\gamma_0\eta_t\alpha_0 $$
for all $0\leq t<n$. It is easy to see that (1)--(3) are satisfied. For (4), we note that $\gamma_0$ does not contain an occurrence of $t^{6n+1}$ for any $t\in\mathsf{B}$. Moreover, for $t\in\{0,1\}\subseteq\mathsf{B}$, $u_t$ contains a unique occurrence of $t^{6n+2}$ which contains the demonstrated occurrence of $\eta_t$; for $t\in\mathsf{B}\setminus\{0,1\}$, the demonstrated occurrence of $\eta_t$ is the unique occurrence of $t^{6n+1}$ in $u_t$. For $t, t'\in \mathsf{B}$ with $t\neq t'$, $u_{t'}$ does not contain any occurrence of $t^{6n+1}$. Property (4) follows from these observations.
 
In the rest of the proof we use the enumeration $(u_{t})_{0\le t< n}$ of the elements of $\mathcal{U}$. 

We will use the follwing notation. For any integer $k$, let $(k)_n$ denote the remainder $k\ (\!\!\!\!\mod n)$, i.e., $(k)_n$ is the unique $0\leq j<n$ with $k\equiv j\ (\!\!\!\!\mod n)$. 

Let $\mathsf{A}$ be an alphabet with $|\mathsf{A}|=n^2$. Fix a bijection $f$ from $\{0,1,2,\dots, n-1\}\times\{0,1,2,\dots, n-1\}$ to $\mathsf{A}$. Define a morphism $\phi\colon\mathsf{A}^*\to\mathsf{A}^*$ by letting $\phi(f(r,s))=f(r,(s+1)_n)$.

Since $|\{0,1,2,\dots,n-1\}^n|=n^n$, for notational convenience, we define 
$$(w_{i,\alpha})_{i\geq 0,\,\alpha\in \{0,1,2,\dots,n-1\}^n}$$ by induction on $i$. First let $v_{0,r,s}=f(r,s)$ for $0\le r,s< n$. Then let
$$w_{0,\alpha}=v_{0,\alpha(0),0}v_{0,\alpha(1),1}\cdots v_{0,\alpha(n-1),n-1}=f(\alpha(0),0)f(\alpha(1),1)\cdots f(\alpha(n-1),n-1)$$
for all $\alpha\in\{0,1,2,\dots, n-1\}^n$. It is clear that for distinct $\alpha, \beta\in\{0,1,2,\dots,n-1\}^n$, we have $w_{0,\alpha}\neq w_{0,\beta}$. Also, for any word $w\in \mathsf{A}^n$ of the form
$$ w=v_{0,r_0,0}v_{0,r_1,1}\cdots v_{0,r_{n-1},n-1}$$
with $r_0, r_1,\dots, r_{n-1}\in \{0,1,\dots, n-1\}$, if we let $\alpha=r_0r_1\cdots r_{n-1}$, then $w=w_{0,\alpha}$. It follows that for any word $w\in \mathsf{A}^*$ with $n\,|\,|w|$, if $w$ is written as a concatenation of $v_{0,r,s}$'s and the indices $s$ are ordered as $0,1,2,\dots, n-1,0, 1,2,\dots$, then $w$ is built from 
$$W_0=\big\{w_{0,\alpha}\colon \alpha\in\{0,1,2,\dots, n-1\}^n\big\}.$$

Suppose for $i\geq 0$, $(v_{i,r,s})_{0\le r,s< n}$ and $(w_{i,\alpha})_{\alpha\in \{0,1,2,\cdots,n-1\}^n}$ have been defined so that the following hold:
\begin{itemize}
\item[(i)] $w_{i,\alpha}=v_{i,\alpha(0),0}v_{i,\alpha(1),1}\cdots v_{i,\alpha(n-1),n-1}$;
\item[(ii)] for any $0\le r, s<n $, $v_{i,r,s}=\phi^s(v_{i,r,0})$.
\end{itemize} 
Then for $0\leq t<n$, define $$u_{i+1,t}=v_{i,u_t(0),(0)_n}v_{i,u_t(1),(1)_n}\cdots v_{i,u_t(k),(k)_n}\cdots v_{i,u_t(mn),(mn)_n}.$$ For $0\leq r, s<n$, define 
$$v_{i+1,r,0}=u_{i+1,r} \mbox{ and } v_{i+1,r,s}=\phi^s(v_{i+1,r,0}). $$ For $\alpha\in \{0,1,2,\dots,n-1\}^n$, define 
$$w_{i+1,\alpha}=v_{i+1,\alpha(0),0}v_{i+1,\alpha(1),1}\cdots v_{i+1,\alpha(n-1),n-1}.$$ 
This finishes the inductive definition.

It is easy to see that for any $i\geq 0$, $\alpha\in\{0,1,2,\cdots, n-1\}^n$, $0\leq r, s<n$, we have
$$ |v_{i,r,s}|=(mn+1)^i \mbox{ and } |w_{i,\alpha}|=n(mn+1)^i. $$
For each $i\geq 0$, let
$$ W_i=\big\{w_{i,\alpha}\colon \alpha\in\{0,1,2,\dots, n-1\}^n\big\}.$$
Now each element of $W_1$ has length $n(mn+1)$, which is a multiple of $n$. Moreover, if we write each element of $W_1$ as a concatenation of $v_{0,r,s}$'s, then the indices $s$ are ordered as $0,1,2,\dots, n-1, 0,1,2,\dots$. It follows from the observation we made in the definition of case $i=0$ that each element of $W_1$ is built from $W_0$. By a similar argument, we conclude that for any $i\geq 0$, each element of $W_{i+1}$ is built from $W_i$. 

The definition naturally gives rise to a directive sequence $\boldsymbol{\tau}=(\tau_i\colon A_{i+1}^*\to A_i^*)_{i\geq 0}$, where $A_0=\mathsf{A}$ and $|A_{i+1}|=n^n$ for all $i\geq 0$. 
It is clear that $\boldsymbol{\tau}$ has constant length. Moreover, $\boldsymbol{\tau}$ is proper by (2), primitive by (3), and recognizable by (4). Thus $(X_{\boldsymbol{\tau}}, S)$ is a strong Toeplitz subshift of rank $n^n$.
 
Fix arbitray $i\geq 0$, $\alpha\in \{0,1,2,\dots,n-1\}^n$ and $0<s<n$, and consider $\phi^s(w_{i,\alpha})$. By induction on $i\geq 0$, one can easily verify that
$$ \phi^s(w_{i,\alpha})=v_{i, \alpha(0), (s)_n}v_{i,\alpha(1), (1+s)_n}\cdots v_{i, \alpha(n-1), (n-1+s)_n}. $$
It follows that if $\beta, \gamma\in \{0,1,2,\cdots,n-1\}^n$ and $(\beta\gamma)[s,s+n)=\alpha$, then $\phi^s(w_{i,\alpha})$ occurs in $w_{i,\beta}w_{i,\gamma}$ at position $s(mn+1)^i$. By (3), for any $\beta,\gamma\in\{0,1,2,\dots,n-1\}^n$, we have that $w_{i,\beta}w_{i,\gamma}$ is a subword of $w_{i+1,\alpha}$. Thus in particular $\phi^s(w_{i,\alpha})$ is a subword of $w_{i+1,\alpha}$.

Let $\varphi\colon \mathsf{A}^\mathbb{Z}\to \mathsf{A}^\mathbb{Z}$ be defined by 
$$\varphi(x)(k)=\phi(x(k))$$ 
for any $x\in \mathsf{A}^\mathbb{Z}$ and $k\in\mathbb{Z}$. Then by the above observation, $\varphi$ is an automorphism of $(X_{\boldsymbol{\tau}},S)$. We obviously have that $\varphi^s\ne \mbox{id}$ for $1\le s<n$ and $\varphi^n=\mbox{id}$. So the order of $\varphi$ is $n$. By Proposition~\ref{prop:aut}, the automorphism group of $(X_{\boldsymbol{\tau}},S)$ is isomorphic to $\mathbb{Z}\oplus C$, where $C$ is a finite cyclic group and $|C|$ is a multiple of $n$.

Now we note the following generalization of Lemma~\ref{gcd}.

\begin{lemma} Let $K\geq 2$ and let $\boldsymbol{\tau}=(\tau_i\colon A^*_{i+1}\to A^*_i)_{i\geq 0}$ be a primitive, proper and recognizable directive sequence of alphabet rank $K$. Suppose that $(X_{\boldsymbol{\tau}},S)$ is an aperiodic Toeplitz subshift. For each $i\geq 0$, let $d_i=\mbox{\rm gcd}(|\tau_{[0,i+1)}(a)|\colon a\in A_{i+1})$. Then $\mbox{\rm lcm}(d_i)_{i\geq 0}$ is the scale for $(X_{\boldsymbol{\tau}}, S)$. 
\end{lemma}

\begin{proof} The proof is similar to that of Lemma~\ref{gcd}.
\end{proof}

Since $|w_{i, \alpha}|=n(nm+1)^n$ for all $i\geq 0$ and $\alpha\in\{0,1,2,\dots, n-1\}^n$, the lemma gives that the scale for $(X_{\boldsymbol{\tau}}, S)$ is the supernatural number $n\times (mn+1)^{\infty}$. By Proposition~\ref{prop:aut}, the automorphism group of $(X_{\boldsymbol{\tau}}, S)$ is isomorphic to $\mathbb{Z}\oplus C$, where $C$ is a cyclic group and $|C|$ is a factor of $n$. Hence $|C|=n$ and $C$ is isomorphic to $\mathbb{Z}/n\mathbb{Z}$. 

The proof of Theorem~\ref{thm:aut} is complete.

\section{Open Problems\label{sec:7}}

In this final section we list some open problems for future research.

The following problem is still open.

\begin{problem} Determine the automorphism group of an arbitrary Toeplitz subshift of finite rank.
\end{problem}

The issue is to determine the order of the finite part of the automorphism group from a Toeplitz sequence.

Concerning the Characterization Problem, we still have the following question.

\begin{problem} Is there a function $h\colon \mathbb{N}\to\mathbb{N}$ such that any Toeplitz subshift of topological rank $K\geq 2$ is a strong Toeplitz subshift of rank $h(K)$? 
\end{problem}

Concerning the classification problems, we have the following questions.

\begin{problem} Is the conjugacy relation for all Toeplitz subshifts of finite topological rank hyperfinite?
\end{problem}

\begin{problem} Is the bi-factor relation for Toeplitz subshifts of finite rank the same as the conjugacy relation?
\end{problem}

In this paper we did not consider the orbit equivalence at all. But it is known from Giordano--Putnam--Skau \cite{GPS} that the orbit equivalence for minimal Cantor systems is determined by the orbit equivalence of their corresponding sets of invariant probability measures, which are in turn determined by their corresponding sets of ergodic invariant measures. We are able to show that any Toeplitz subshift of finite rank has only finitely many ergodic invariant measures (this will appear elsewhere). In addition, for unique ergodic systems, it is known (see \cite[Corollary 1]{GPS}) that the orbit equivalence of a single measure is reducible to the equality of countable sets of real numbers. It is therefore conceivable that the orbit equivalence for Toeplitz subshifts of finite rank is Borel reducible to the equality of countable sets of real numbers. But the following problem remains.

\begin{problem} What is the complexity of the orbit equivalence relation among Toeplitz subshifts of finite rank?
\end{problem}


\begin{thebibliography}{9}

\bibitem{ADE2024}
F. Arbul\'{u}, F. Durand,  B. Espinoza, The Jacobs–Keane theorem from the $\mathcal{S}$-adic viewpoint,
Discrete Contin. Dyn. Syst. 44 (2024), no. 10, 3077--3108.

\bibitem{Au}
J. Auslander, Minimal Flows and Their Extensions. North-Holland Mathematics Studies, vol. 153, North-Holland Publishing Co., Amsterdam, 1988, Notas de Matem\'{a}tica [Mathematical Notes], 122.

\bibitem{BSTY19}
V. Berth\'e, W. Steiner, J. M. Thuswaldner, R. Yassawi, 
Recognizability for sequences of morphisms, Ergodic Theory Dynam. Systems 39 (2019), no. 11, 2896--2931.

\bibitem{DDMP16} S. Donoso, F. Durand, A. Maass, S. Petite, On automorphism groups of low complexity subshifts, Ergodic Theory Dynam. Systems 36 (2016), no. 1, 64--95.


\bibitem{DDMP17} S. Donoso, F. Durand, A. Maass, S. Petite, On automorphism gruoups of Toeplitz subshifts, Discrete Analysis (2017), Paper No. 11, 19pp.


    \bibitem{DDMP21} S. Donoso, F. Durand, A. Maass, S. Petite, Interplay between finite topological rank minimal Cantor systems, $\mathcal{S}$-adic subshifts and their complexity, Trans. Amer. Math. Soc. 374 (2021), no. 5, 3453--3489.

 \bibitem{DJK}
R. Dougherty, S. Jackson, A. S. Kechris, The structure of hyperfinite Borel equivalence relations, Trans. Amer. Math. Soc. 341 (1994), no. 1, 193--225.  


 \bibitem{DSurvey} T. Downarowicz, Survey of odometers and Toeplitz flows, In: Algebraic and Topological Dynamics, 7–37. Contemp. Math. 385, American Mathematical Society, Providence, RI, 2005. 

    \bibitem{DKL} T. Downarowicz, T. Kwiatkowski, J. Lacroix, A criterion for Toeplitz flows to be topologically isomorphic and applications, Colloq. Math. 68 (1995), no. 2, 219--228.

\bibitem{DM}
T. Downarowicz, A. Maass,
Finite rank Bratteli--Vershik diagrams are expansive,
Ergodic Theory Dynam. Systems 28 (2008), no. 3, 739--747.


\bibitem{Du10} F. Durand, Combinatorics on Bratteli diagrams and dynamical systems. In Combinatorics, Automata and Number Theory (Encyclopedia of Mathematics and its Applications, 135). Ed. V. Berth\'e and M. Rigo. Cambridge University Press, Cambridge, 2010, pp. 324-372.

\bibitem{DFM15} F. Durand, A. Frank, A. Maass, Eigenvalues of Toeplitz minimal systems of finite topological rank, Ergodic Theory Dynam. Systems 35 (2015), no. 8, 2499--2528.

\bibitem{DL12} F. Durand, J. Leroy, S-adic conjecture and Bratteli diagrams, C. R. Acad. Sci. Paris, Ser. I 350 (2012), 979--983.




\bibitem{Es23}
B. Espinoza, Symbolic factors of $\mathcal{S}$-adic subshifts of finite alphabet rank, Ergodic Theory Dynam. Systems 43 (2023), no. 5, 1511--1547.


\bibitem{EM}
B. Espinoza, A. Maass, On the automorphism group of minimal $\mathcal{S}$-adic subshifts of finite alphabet rank, Ergodic Theory Dynam. Systems 42 (2022), no. 9, 2800--2822.


\bibitem{Fe96} S. Ferenczi, Rank and symbolic complexity, Ergodic Theory Dynam. Systems 16 (1996), no. 4, 663--682.

\bibitem{Ga} S. Gao, Invariant Descriptive Set Theory. Pure and Applied Mathematics, vol. 293. CRC Press, Baton Rouge, 2009.


\bibitem{GJJLLSW25} S. Gao, L. Jacoby, W. Johnson, J. Leng, R. Li, C. E. Silva, Y. Wu, On finite spacer rank for words and subshifts,
Discrete Contin. Dyn. Syst. 45 (2025), no. 1, 248--285.

\bibitem{GL25} S. Gao, R. Li, Subshifts of finite symbolic rank, Ergodic Theory Dynam. Systems 45 (2025), 807--848.

\bibitem{GJ00} R. Gjerde, \O. Johansen, Bratteli--Vershik models for Cantor minimal systems: applications to Toeplitz flows, Ergodic Theory Dynam. Systems 20 (2000), no. 6, 1687--1710.


\bibitem{GPS}
T. Giordano, I. F. Putnam, C. F. Skau, 
Topological orbit equivalence and $C^*$-crossed products, J. Reine Angew. Math. 469 (1995), 51--111.

\bibitem{Hj}
G. Hjorth, Classification and Oribt Equivalence Relations. Mathematical Surveys and Monographs, vol. 75. American Mathematical Society, Providence, RI, 1991.


\bibitem{HPS}
R. H. Herman, I. F. Putnam, C. F. Skau,
Ordered Bratteli diagrams, dimension groups and topological dynamics,
Internat. J. Math. 3 (1992), no. 6, 827--864.

\bibitem{JKL}
S. Jackson, A. S. Kechris, A. Louveau, Countable Borel equivalence relations, J. Math. Logic 2 (2002), no. 1, 1--80.

\bibitem{JK}
K. Jacobs, M. Keane, $0$-$1$ sequences of Toeplitz type, Z. Wahrscheinlichkeitstheorie verw. Geb. 13 (1969), 123--131.

    \bibitem{Kaya} B. Kaya,  The complexity of the topological conjugacy problem for Toeplitz subshifts, Israel J. Math. 220 (2017), no. 2, 873--897.


\bibitem{Ke}
A. S. Kechris, Classical Descriptive Set Theory. Graduate Texts in Mathematics, vol. 156. Springer--Verlag, New York, 1995.

\bibitem{KL}
A. S. Kechris, A. Louveau, The classification of hypersmooth Borel equivalence relations, J. Amer. Math. Soc. 10 (1997), 215--242.






\bibitem{PS}
R. Pavlov, S. Schmieding,
On the structure of generic subshifts,
Nonlinearity 36 (2023), no. 9, 4904--4953.

\bibitem{SabokTsankov2017} M. Sabok, T. Tsankov, On the complexity of topological conjugacy of Toeplitz subshifts, Israel J. Math. 220 (2017), no. 2, 583--603.

\bibitem{Thomas2013} S. Thomas, Topological full groups of minimal subshifts and just-infinite groups, In: Proceedings of the 12th Asian Logic Conference, 298--313. World Sci. Publ., Hackensack, NJ, 2013.

\bibitem{Williams1984} S. Williams, Toeplitz minimal flows which are not uniquely ergodic, Z. Wahrsch. Verw. Gebiete 67 (1984), no. 1, 95--107.

\bibitem{Yu}
P. Yu, Topological Conjugacy Relation on the Space of Toeplitz Subshifts, PhD dissertation, University of North Texas, 2022.

    \end{thebibliography}
\end{document}